\newif\ifdraft		
\draftfalse

\documentclass[
	11pt, 
	a4paper, 
	\ifdraft draft,\else final,\fi
	oneside
]{article}

\usepackage[final]{graphicx} 
\usepackage{changes}
	\definechangesauthor[name=Kai, color=orange]{KR}
\usepackage[T1]{fontenc} 
\usepackage{lmodern}
\usepackage{mathrsfs}
\usepackage{fullpage}
\usepackage{color, xcolor} 
\usepackage{paralist} 
\usepackage[main=english, german]{babel} 
\usepackage{csquotes} 
\usepackage{amsmath, amsfonts, amssymb, amsthm} 
\usepackage{tabularx} 
\usepackage{subdepth} 
\usepackage{MnSymbol} 
\usepackage{dsfont} 
\usepackage{xfrac} 
\usepackage{todonotes} 
\ifdraft\else\renewcommand{\todo}[2][]{} \fi

\usepackage{mathtools} 
\usepackage{abstract} 
\usepackage[normalem]{ulem} 
\usepackage{esint} 
\usepackage{placeins} 
\usepackage{caption, subcaption} 

\usepackage{hyperref} 
\hypersetup{colorlinks, linkcolor=blue, urlcolor=red, filecolor=green, citecolor=blue, final}
\usepackage[capitalise]{cleveref} 
\usepackage[open, openlevel=2, numbered]{bookmark}

\usepackage[backend=biber, style=alphabetic, giveninits=true]{biblatex} 
\usepackage{tikz} 
 	\usetikzlibrary{decorations.pathreplacing, decorations.pathmorphing, patterns, mindmap, trees, calc}
 	\tikzset{>=stealth}

\definecolor{darkblue}{RGB}{0,35,102}
\definecolor{darkgreen}{RGB}{15,140,15}

\theoremstyle{plain} 
\newtheorem{theorem}{Theorem}[section]
\newtheorem{lemma}[theorem]{Lemma}

\newtheorem{proposition}[theorem]{Proposition}
\newtheorem{corollary}[theorem]{Corollary}

\newtheorem{remark}[theorem]{Remark}

\theoremstyle{definition} 
\newtheorem{definition}[theorem]{Definition}
\newtheorem{assumption}[theorem]{Assumption}

\newcommand{\stepemph}[1]{\textsc{#1}}

\newcommand{\dx}[1][x]{\mathrm{d}#1}
\DeclareMathOperator{\supp}{supp}
\DeclareMathOperator{\SO}{SO}
\DeclareMathOperator{\dist}{dist}
\DeclareMathOperator{\Cont}{C}

\DeclareMathOperator{\Gl}{Gl}
\DeclareMathOperator{\SobH}{H}
\DeclareMathOperator{\SobW}{W}
\DeclareMathOperator{\Leb}{L}
\DeclareMathOperator{\Oclass}{O}
\DeclareMathOperator{\oclass}{o}
\DeclareMathOperator{\sym}{sym}
\DeclareMathOperator{\skewMat}{skew}

\DeclareMathOperator{\divergence}{div}
\DeclareMathOperator{\trace}{tr}

\newcommand{\R}{\mathbb{R}}

\newcommand{\N}{\mathbb{N}}

\newcommand{\Z}{\mathbb{Z}}
\newcommand{\eps}{\varepsilon}
\newcommand{\abs}[1]{\left| #1 \right|}
\newcommand{\class}[2]{\left\lbrace #1 \,\middle|\, #2 \right\rbrace}
\newcommand{\set}[1]{\left\lbrace #1 \right\rbrace}
\newcommand{\norm}[1]{\left\Vert #1 \right\Vert}
\newcommand{\placeholder}{\makebox[1ex]{$\boldsymbol{\cdot}$}}
\newcommand{\dif}[1]{\mathrm{D} #1}
\newcommand{\wto}{\rightharpoonup}
\newcommand{\mean}[1]{\left\langle #1 \right\rangle}
\newcommand{\meanHarm}[1]{\left\langle #1 \right\rangle_{\mathrm{harm}}}

\newcommand{\EnergyEhe}[2][h]{\mathcal{E}^{#1}_{#2}}
\newcommand{\EnergyIhe}[2][h]{\mathcal{I}^{#1}_{#2}}
\newcommand{\EnergyIline}[1]{\mathcal{I}^\mathrm{lin}_{#1}}
\newcommand{\EnergyIhhom}[1][h]{\mathcal{I}^{#1}_\mathrm{hom}}
\newcommand{\EnergyIlinhom}[1][]{\mathcal{I}^{#1\mathrm{lin}}_\mathrm{hom}}
\newcommand{\EnergyIlints}{\mathcal{I}^\mathrm{lin}_\mathrm{ts}}

\makeatletter
\newenvironment{mathcenter}{%
  \@fleqnfalse%
  }{%
  \@fleqntrue\ignorespacesafterend%
}
\makeatother

\counterwithin{equation}{section}

\newlength\normalparindent

\author{Stefan Neukamm, Kai Richter}
\title{Linearization and Homogenization of nonlinear elasticity close to stress-free joints}

\hypersetup{
	pdftitle={Linearization and Homogenization of nonlinear elasticity close to stress-free joints},
	pdfsubject={Nonlinear Elasticity},
	pdfauthor={Kai Richter},
	pdfkeywords={nonlinear elasticity, Gamma-convergence, homogenization, linearization, prestrain, stress-free joint, two-scale convergence, geometric rigidity estimate, Jones domain, extension operator},
	pdfdisplaydoctitle
}

\DeclareDelimFormat[bib,biblist]{nametitledelim}{\addcolon\space}
\DeclareFieldFormat*{title}{\mkbibemph{#1}\isdot}
\DeclareFieldFormat*{journaltitle}{#1}

\AtEveryBibitem{%
  \clearfield{url}%
  \clearfield{issn}%
  \clearfield{isbn}%
  \clearfield{month}%
}
\allowdisplaybreaks 
\usetikzlibrary{arrows}

\begin{document}
\selectlanguage{english}

\makeatletter
\renewcommand{\thefootnote}{\fnsymbol{footnote}}
\begin{center}
  \LARGE
  \@title
  \bigskip

  \normalsize
  Stefan Neukamm\footnote{stefan.neukamm@tu-dresden.de, Faculty of Mathematics, Technische Universit{\"a}t Dresden, 01062 Dresden, Germany} and
  Kai Richter\footnote{kai.richter@tu-dresden.de, Faculty of Mathematics, Technische Universit{\"a}t Dresden, 01062 Dresden, Germany}
  \par \bigskip

  \today
\end{center}
\setcounter{footnote}{0}
\renewcommand{\thefootnote}{\arabic{footnote}}
\makeatother


\begin{abstract}
In this paper, we study a hyperelastic composite material with a periodic microstructure and a prestrain close to a stress-free joint. We consider two limits associated with linearization and homogenization. Unlike previous studies that focus on composites with a stress-free reference configuration, the minimizers of the elastic energy functional in the prestrained case are not explicitly known. Consequently, it is initially unclear at which deformation to perform the linearization.	

Our main result shows that both the consecutive and simultaneous limits converge to a single homogenized model of linearized elasticity. This model features a homogenized prestrain and provides first-order information about the minimizers of the original nonlinear model. We find that the homogenization of the material and the homogenization of the prestrain are generally coupled and cannot be considered separately.

Additionally, we establish an asymptotic quadratic expansion of the homogenized stored energy function and present a detailed analysis of the effective model for laminate composite materials. A key analytical contribution of our paper is a new mixed-growth version of the geometric rigidity estimate for Jones domains. The proof of this result relies on the construction of an extension operator for Jones domains adapted to geometric rigidity.
\smallskip

\textbf{Keywords:} nonlinear elasticity, $\Gamma$-convergence, homogenization, linearization, prestrain, stress-free joint, two-scale convergence, geometric rigidity estimate, Jones domain, extension operator.

\textbf{MSC-2020:} 35B27; 49J45; 74-10; 74B20; 74E30; 74Q05; 74Q15.
\end{abstract}

\paragraph{Acknowledgement.}
The authors received support from the German Research Foundation (DFG) via the research unit FOR 3013, “Vector- and tensor-valued surface PDEs” (project number 417223351).

\ifdraft \listoftodos \fi

{\hypersetup{hidelinks} \tableofcontents}


\section{Introduction}

One way to deal with the difficulties of nonlinear elasticity arising from its non-convex nature, is the derivation of simpler, effective models that capture the behavior from a macroscopic point of view. In this context, linearization and homogenization are two important concepts. Both have already been discussed by many authors (e.g.\ \cite{DNP02, MN11, MPT19, Schm07, ADD12}) and are well understood, in particular, in the case when the reference configuration is stress-free. However, for composites, this is not always a natural assumption. Composites may naturally feature prestrain due to varying properties of the material, be they unwanted side-effects or even desirable behavior.
Examples include wood composites \cite{hassani2015rheological, MJ22} (where changes of moisture content lead to swelling and shrinkage), liquid crystal elastomers \cite{War03} (which undergo a shape change due an ordering of their long molecules in a nematic phase), or residual stresses in additively manufactured composites \cite{zhang2017characterization}. One promising application of prestrained composites is the design of active materials, which change their shape upon activation via external stimuli such as light, humidity, temperature, electric fields, etc., see \cite{MJZ18, KES07}.
\medskip

In this paper we study periodic, elastic composites with prestrain.
Our starting point is the  energy functional of nonlinear elasticity,
\begin{equation} \label{Eq:EnergyFunctionalOnDeformation}
  \EnergyEhe{\eps}(\phi) := \int_\Omega W^h \left(\tfrac{x}{\eps}, \dif\phi(x)\right) \,\dx, \quad \phi \in \SobH^1(\Omega, \R^d),
\end{equation}
where $\Omega \subset \R^d$ denotes the reference domain of the elastic body, $\phi$ a deformation, and $W^h(y,F)$ a stored energy function parametrized by some scaling parameter $0<h\ll 1$. More precisely, we assume that 
\begin{equation} \label{Eq:MultiplicativeDecomposition}
  W^h(y,F) := W\left(y, F A_h(y)^{-1}\right),
\end{equation}
where $W(y,F)$ denotes a standard stored energy function that is $Y:=[0,1)^d$-periodic in $y$, and minimized and non-degenerate for $F\in\SO(d)$, see~\cref{Ass:ElasticMaterialLaw} for details. The stored energy function describes an elastic composite with a prestrain that is modeled (following \cite{BNS20})
with help of a multiplicative decomposition of the deformation gradient $F$ into an elastic part and a prestrain tensor $A_h: \R^d \to \R^{d\times d}$. The latter is assumed to be $Y$-periodic and a perturbation of a stress-free joint. Roughly speaking this means that
\begin{equation}\label{definition_prestrain}
  A_h = (I + h \tilde{B}_h) A,
\end{equation}
where the stress-free joint $A=Da$ is a tensor field with a Bilipschitz potential $a\in W^{1,\infty}_{\rm loc}(\R^d;\R^d)$ (see \cref{Def:StressFreeJoint} below) and $\tilde{B}_h$ is a bounded perturbation.
Stress-free joints have been considered by R.D.~James in \cite{Jam86} to model composites with a prestrain that can be accommodated for by a piece-wise affine deformation. The notion that we consider in this paper is a slight generalization of it.
To understand deformations of prestrained composites with (nearly) zero elastic energy is important for many applications. Examples include 
ceramics \cite{DT68}, where stress introduced during the drying process can lead to cracks,
thermal expansion of bi-material joints \cite{PWHPG94},
and equilibrium configurations of crystals \cite{CK88}. We also note that the theory of stress-free joints is related to models for twinning in crystals \cite{Zan90} and phase transitions of multi-phase materials such as shape-memory alloys \cite{Rue22}.
\medskip

We are interested in minimizers of \eqref{Eq:EnergyFunctionalOnDeformation} when $0<\eps,h\ll 1$. 
Therefore we study the limits $h\to0$ (linearization) and $\eps\to 0$ (homogenization), both successively and simultaneously. With regard to the stored energy function the successive limit ``linearization after homogenization'' is especially interesting. The limit $\eps\to 0$ leads to a homogenized stored energy function given by the multi-cell homogenization formula of \cite{Muel87}:
\begin{equation} \label{Eq:Multicell_formula}
 W_{\hom}^h (F) = \inf_{k \in \N} \inf_{\varphi \in \SobW^{1,\infty}_\mathrm{per}(kY, \R^d)} \fint_{kY} W^h\big(y, F + \dif{\varphi}(y)\big) \,\dx[y].
\end{equation}
Unfortunately, this formula is of limited use in practice: In addition to the difficulties of computing the two infima, the dependence on the prestrain is implicit. In particular, the minimizers of $W_{\hom}^h$ are unknown for $h>0$ due to the presence of the prestrain. These difficulties can be overcome in the limit $h \to 0$ where we shall obtain a unique minimizer and explicit formulas to compute it. We observe that the infimum of $\EnergyEhe{\eps}$ scales like $h^2$, see \cref{Rem:EnergyScaling} where it is shown that this follows from \eqref{definition_prestrain} in combination with the assumption that $A$ is a stress-free joint. This motivates us to study the scaled energy $h^{-2}\EnergyEhe{\eps}$ and the corresponding stored energy function $h^{-2}W^h$.
As a main result we show in \cref{Thm:Expansion_Whom} that the homogenized stored energy function admits a quadratic Taylor expansion at $\bar{A} := \int_Y A\dx[y]$. The expansion is of the form
\begin{equation} \label{Eq:Qhom}
  \tfrac{1}{h^2} W^h_{\hom}(\bar A+G)\approx R^A(B) + Q^A_{\hom}(G - B_{\hom}\bar{A}).
\end{equation}
Above, $Q^A_{\hom}$ denotes a quadratic form obtained by homogenizing the quadratic form 
\begin{equation*}
  Q^A(y,G):=Q(y, GA^{-1}(y)),\qquad\text{where}\qquad Q(y,G):=\lim_{h \to 0} \tfrac{1}{h^2} W(y,I + hG).
\end{equation*}
Furthermore, $B_{\hom}$ is an effective incremental prestrain tensor obtained by a weighted average of the incremental prestrain tensor $B$, which we define as the limit of $\tilde{B}_h$ for $h\to 0$. The term $R^A(B)$ is a residual energy that is independent of the displacement $G$, but depends on the stress-free joint $A$, the stored energy function $W$ and $B$. We refer to \cref{Sec:AsyExpansion} for the details. We further prove that (almost) minimizers $F_h^*$ of $W_{\hom}^h$ admit an expansion that is explicit up to an error term of order $\oclass(h)$, see \cref{Cor:ExpansionMinimizer}. We note that this is a nontrivial result, since in our setting minizers of $W^h_{\hom}$ are not explicitly known or may even not  exist. We also establish a commutative diagram that shows that linearization and homogenization of the stored energy function commute.
\medskip

In \cref{Sec:LinAndHomResults} we lift this commutative diagram to the level of a $\Gamma$-convergence result for the associated energy functionals, and we investigate the asymptotics of (almost) minimizers $\phi_{\eps,h}^*$ of $\EnergyEhe{\eps}(\phi)$ subject to well-prepared boundary conditions of the form $\phi = a_\eps + hg$ on $\Gamma \subset \partial\Omega$, where $a_\eps$ denotes the potential of the stress-free joint $A(\frac{\placeholder}{\eps})$, i.e.\ $\dif{a_\eps}(x) = A(\frac{x}{\eps})$.
In particular, in \cref{Prop:strong_conv} we prove an expansion of the form $\phi_{\eps,h}^* = a_\eps + hu_{\eps,h}^*$ and show that the displacement $u_{\eps,h}^*$ satisfies the commutative convergence diagram
\begin{center}
\begin{tikzpicture}[baseline={([yshift=-.5ex]current bounding box.center)}]
	\node (1) at (0,0) {$u_{\eps,h}^*$};
	\node (2) at (2,0) {$u_\eps^*$};
	\node (3) at (0,-2) {$u_h^*$};
	\node (4) at (2,-2) {$u^*$,};
	\draw[->]  (1) -- (2) node[above, midway] {\footnotesize{$h\to 0$}};
	\draw[->]  (3) -- (4) node[above, midway] {\footnotesize{$h\to0$}};
	\draw[->]  (1) -- (3) node[left, midway] {\footnotesize{$\eps\to 0$}};
	\draw[->]  (2) -- (4) node[right, midway] {\footnotesize{$\eps\to 0$}};
	\draw[->]  (1) -- (4) node[above, midway] {};
\end{tikzpicture}
\end{center}
where the arrows stand for weak convergence in $\SobH^1_{\Gamma, g}(\Omega, \R^d)$.
We show that the displacement $u^*_{\eps,h}$ and the limits $u_\eps^*$, $u_h^*$, $u^*$ are (almost) minimizers of the functionals 
\begin{subequations}\label{D:rescaledfunctionals}
\begin{align}
	&\EnergyIhe{\eps}(u) := \frac{1}{h^2}\int_\Omega W^h\left(\tfrac{x}{\eps}, A(\tfrac{x}{\eps}) + h \dif{u}(x)\right) \,\dx,&& u \in \SobH^1_{\Gamma,g}(\Omega, \R^d),\\
	&\EnergyIhhom(u) := \frac{1}{h^2}\int_\Omega W^h_{\hom}\left(\bar{A} + h \dif{u}(x)\right) \,\dx, && u \in \SobH^1_{\Gamma, g}(\Omega, \R^d), \\
	&\EnergyIline{\eps}(u) := \int_\Omega Q^A\left(\tfrac{x}{\eps}, \dif{u}(x) + B(\tfrac{x}{\eps})A(\tfrac{x}{\eps}) \right)\,\dx, && u \in \SobH^1_{\Gamma, g}(\Omega, \R^d), \\
	&\EnergyIlinhom(u) := \int_\Omega Q^{A}_{\hom}\left(\dif{u}(x) + B_{\hom}\bar{A} \right) \,\dx + |\Omega|R^A(B), && u \in \SobH^1_{\Gamma, g}(\Omega, \R^d).
\end{align}
\end{subequations}
In fact,  the convergence results are obtained by proving the validity of the diagram

\begin{center}
\begin{tikzpicture}[baseline={([yshift=-.5ex]current bounding box.center)}]
	\draw[->]  (0.5,0) -- (1.5,0);		
	\draw[->]  (0.5,-2) -- (1.5,-2);
	\draw[->]  (0,-0.5) -- (0,-1.5);
	\draw[->]  (2,-0.5) -- (2,-1.5);
	\draw[->]  (0.5,-0.5) -- (1.5,-1.5);
	\node (1) at (0,0) {$\EnergyIhe{\eps}$};
	\node (2) at (2,0) {$\EnergyIline{\eps}$};
	\node (3) at (0,-2) {$\EnergyIhhom$};
	\node (4) at (2,-2) {$\EnergyIlinhom$,};
\end{tikzpicture}
\end{center}
where each arrow stands for $\Gamma$-convergence,  and the horizontal direction corresponds to linearization ($h \to 0$), and the vertical direction to homogenization ($\eps \to 0$). This is done in \cref{Thm:gamma_convergence}. The diagram shows that the successive limits of linearization and homogenization commute and lead to the limit obtained by the simultaneous limit $(h,\eps)\to 0$. The $\Gamma$-convergence results are w.r.t.\ weak convergence in $\SobH^1(\Omega, \R^d)$ and thus only yield weak convergence of the minimizers. A posteriori we upgrade some of them to a stronger topology by utilizing the quadratic form of the linearized limit, see \cref{Prop:strong_conv}.

\paragraph{Key tools for the proofs and survey of the literature.}

The linearization of elasticity using De~Giorgi's $\Gamma$-convergence (see \cite{DGF75, Dal93}) goes back to G.~Dal~Maso, M.~Negri and D.~Percivale \cite{DNP02} and uses the geometric rigidity estimate \cite{FJM02} as a key ingredient. The analysis has been extended to include homogenization, see \cite{NeuPHD,MN11,GN11}. In our work we extend this result to prestrained materials, where the prestrain is a perturbation of a stress-free joint in the sense of \cref{Ass:Prestrain} below. In a mathematical context, stress-free joints have been studied by Ericksen \cite{Eri83} and James \cite{Jam86}. Our main idea to deal with such a prestrain is to utilize the (piece-wise) Bilipschitz potential of the stress-free joint to go back and forth to a transformed reference domain, where the prestrain is small, i.e.\ a perturbation of the identity. To illustrate this, suppose that the potential $a: \Omega \to \R^d$ of a stress-free joint $A$ is Bilipschitz and consider some deformation $\phi \in \SobH^1(\Omega, \R^d)$. Then,
\begin{align*}
	\EnergyEhe{1}(\phi) &= \int_\Omega W\left(x, \dif{\phi}(x)\dif{a}(x)^{-1}(I + h\tilde{B}^h(x))^{-1}\right) \,\dx \\
	&= \int_{a(\Omega)} W\left(a^{-1}(z), \dif{(\phi \circ a^{-1})}(z)(I + h\tilde{B}^h\circ a^{-1}(z))^{-1}\right) \det\dif{a^{-1}}(z)\,\dx[z].
\end{align*}
However, multiple problems arise from this representation. In particular, we cannot directly apply the geometric rigidity estimate of \cite{FJM02} to the transformed domain, since $a(\Omega)$ is not necessarily a Lipschitz domain, even if $a$ is Bilipschitz (see \cite{Lic19} for counterexample). Thus, we show first that the rigidity estimate also holds on the more general class of Jones domains. We achieve this by providing an extension operator on Jones domains that allows to control the distance to the set of rotations. This extension operator is based on the constructions in \cite{Jon81,DM04}.
\smallskip

In this paper, we extend the commutativity of homogenization and linearization established in \cite{NeuPHD,MN11,GN11} to prestrained composites. The first homogenization results for elasticity in this direction are due to Marcellini~\cite{Mar78} for convex integrands and Braides~\cite{Bra85} and M\"uller~\cite{Muel87} for non-convex integrands, where the multi-cell homogenization formula is invoked. The main ingredient to establish this commutativity is a quantitative quadratic expansion of the homogenized stored energy function, as established in \cite{MN11} for materials without prestrain. In this work, the presence of a prestrain leads to additional difficulties. The main ingredients to overcome these, are a connection between the expansions of $W^h$ at $A$ and $W^h_{\hom}$ at $\bar{A} := \fint_Y A(y)\,\dx[y]$, see \cref{Lem:homandprestrain}, as well as again a reduction to a small prestrain. Both ingredients are obtained from the fact that any periodic stress-free joint admits a representation $A = \bar{A} + \dif{\varphi}$ for some $Y$-periodic map $\varphi \in \SobW^{1,\infty}_\mathrm{per}(Y, \R^d)$, see \cref{Lem:periodic_derivative}. 
\smallskip

One major point that allows us to establish compactness for the simultaneous linearization and homogenization is the fact that periodic stress-free joints admit a Bilipschitz potential, see \cref{Prop:SFJBilipschitz}. Note that in general we require stress-free joints only to admit a piece-wise Bilipschitz potential, see \cref{Def:StressFreeJoint}, to conform with the definition in \cite{Jam86} where piece-wise affine maps are considered. The fact that periodicity in this setting implies global injectivity is not trivial. Our proof relies on a general transformation rule for not necessarily injective maps, see \cite[Thm. 3.8]{EG15} and \cite[Thm. B.3.10]{KR19}, which allows us to measure the non-injectivity of a map in terms of the determinant of its derivative.

\subsection{Notation}

Throughout this paper we use the following notation.
\begin{itemize}
	\item $Y := [0,1)^d$ denotes the representative cell of periodicity;
	\item Given $\bar{A} \in \Gl_+(d)$, we say that a measurable map $u: \R^d \to \R^n$ is $\bar{A}Y$-periodic or $\bar{A}$-periodic, if $u(x + \bar{A}k) = u(x)$ for all $k \in \Z^d$ and a.e.\ $x \in \R^d$;
	\item We denote by $\Leb^p_\mathrm{per}(\bar{A}Y, \R^n)$, $\SobW^{1,p}_\mathrm{per}(\bar{A}Y, \R^n)$ and $\SobH^1_\mathrm{per}(\bar{A}Y, \R^n)$ the set of all  $\bar{A}Y$-periodic maps in $\Leb^p_\mathrm{loc}(\R^d, \R^n)$, $\SobW^{1,p}_\mathrm{loc}(\R^d, \R^n)$ and $\SobH^1_\mathrm{loc}(\R^d, \R^n)$, respectively;
	\item $I \in \R^{d\times d}$ denotes the identity matrix, and $\sym G := \frac{1}{2}(G + G^T)$ the symmetric part of $G \in \R^{d\times d}$; we denote the euclidean scalar product in $\R^{d\times d}$ by $\placeholder : \placeholder$;
	\item We call $U \subset \R^d$ a Lipschitz domain, if $U$ is open, bounded, connected and has a Lipschitz boundary, i.e., $\partial U$ is locally the graph of a Lipschitz continuous function, cf.\ \cite[§4.5]{Ada75}.
\end{itemize}

\section{Modeling of prestrained composites}
We model periodic composites with prestrain by means of the elastic energy functional $\EnergyEhe{\eps}$ defined in \eqref{Eq:EnergyFunctionalOnDeformation}. Throughout the paper we assume that the reference domain $\Omega \subset \R^d$ is a Lipschitz domain.
The stored energy function $W^h: \R^d \times \R^{d\times d} \to [0,\infty]$ in \eqref{Eq:EnergyFunctionalOnDeformation} is defined by the expression \eqref{Eq:MultiplicativeDecomposition} and invokes
\begin{itemize}
\item a reference stored energy function $W$ that describes the elastic properties of the components of the composite relative to a virtual stress-free reference configuration,
\item a prestrain tensor $A_h: \R^d \to \R^{d\times d}$, which we assume to be a perturbation of a stress-free joint $A$ of order $h$.
\end{itemize}
In the following we present the precise assumptions on these quantities. We start with the assumptions on the stored energy function and then discuss the assumptions for the prestrain tensor.
To this end we introduce a class of nonlinear material laws that we consider for the composite:
\begin{definition}[Material class $\mathcal W$, {cf.\ \cite[Def.~2.2]{BNPS22}}] \label{Def:MaterialClass_W}
Let $0 < \alpha \leq \beta$, $\rho > 0$. We denote by $\mathcal{W}(\alpha, \beta, \rho)$ the class of functions $W:\R^{d\times d} \to [0,\infty]$, which satisfy
\begin{enumerate}
\renewcommand{\labelenumi}{\theenumi}
\renewcommand{\theenumi}{(W\arabic{enumi})}
	\item \label{Property:W_FrameIndifference}(Frame indifference): $W(RF) = W(F)$ for all $F \in \R^{d \times d}, R \in \SO(d)$;
	
	\item \label{Property:W_NonDegeneracy}(Non-degeneracy):
	\begin{alignat*}{2}
		W(F) &\geq \alpha \dist^2(F, \SO(d)) &&\qquad \text{for all } F \in \R^{d\times d}, \\
		W(F) &\leq \beta \dist^2(F, \SO(d)) &&\qquad \text{for all } F \in \R^{d\times d} \text{ with } \dist^2(F,\SO(d)) \leq \rho;
	\end{alignat*}
	
	\item \label{Property:W_Expansion}(Quadratic expansion): There exists a quadratic form $Q: \R^{d\times d} \to \R$ and an increasing map $r: [0,\infty) \to [0,\infty]$ with $\lim_{\delta \to 0} r(\delta) = 0$, such that
	\begin{equation*}
		\abs{W(I+G) - Q(G)} \leq \abs{G}^2 r(\abs{G}) \qquad \text{for all } G \in \R^{d\times d}.
	\end{equation*}
\end{enumerate}
\end{definition}

\begin{assumption}[Periodic composite] \label{Ass:ElasticMaterialLaw}
We assume that the following statements hold:
\begin{enumerate}[(i)]
	\item (Measurability): $W$ is a Carathéodory function such that for a.e.\ $y \in \R^d$ the map $F \mapsto W(y,F)$ is continuous.
	
	\item (Material law): There exist $0 < \alpha_\mathrm{el} \leq \beta_\mathrm{el}$, $\rho_\mathrm{el} > 0$, such that for a.e.\ $y \in \R^d$, we have $W(y, \placeholder) \in \mathcal{W}(\alpha_\mathrm{el}, \beta_\mathrm{el}, \rho_\mathrm{el})$.
	
	\item (Periodicity): For all $F \in \R^{d\times d}$ the map $y \mapsto W(y,F)$ is $Y$-periodic.
\end{enumerate}
\end{assumption}

A stored energy function $W$ that satisfies \cref{Ass:ElasticMaterialLaw} describes a composite material with a common stress-free reference state, that is, at a.e.\ material point $y \in \R^d$, $W(y,R)$ is minimized exactly for all rotations $R\in\SO(d)$. To model a prestrained composite we appeal to a multiplicative decomposition of the deformation gradient and introduce the prestrain tensor $A_h$, see \eqref{Eq:MultiplicativeDecomposition}. We note that an arbitrary prestrain tensor may lead to non-trivial energy minimizing deformations (ground states) of the elastic energy functional, and in general it is impossible to explicitly understand the dependence of the ground states on the prestrain. The situation is different, if the prestrain tensor is the deformation gradient of a Bilipschitz potential $a$. In that case a ground state is given by the potential $a$ and has zero energy. Roughly speaking, a stress-free joint is a prestrain with such a potential:

\begin{definition}[Stress-free joints] \label{Def:StressFreeJoint}
Let $U \subset \R^d$ be open, bounded and connected. We denote by $\operatorname{SFJ}(U)$ the set of all maps $A: U \to \R^{d\times d}$ such that for some $L>0$ the following holds:
\begin{enumerate}
\renewcommand{\labelenumi}{\theenumi}
\renewcommand{\theenumi}{(SFJ\arabic{enumi})}
	\item \label{Property:SFJUniformBound} $\det A(x) > 0$ and $\max\set{\abs{A(x)}, \abs{A(x)^{-1}}} \leq L$ for a.e.\ $x \in U$;
	\item \label{Property:SFJisDerivative} There exists a continuous map $a \in \SobW^{1,\infty}_\mathrm{loc}(U,\R^d)$ (called a potential of $A$) such that $A = \dif{a}$ a.e.;
	\item \label{Property:SFJisPiecewiseBilipschitz} $a$ is Bilipschitz or $U$ admits a finite decomposition into Lipschitz domains%
	\footnote{This regularity assumption on the domain can be weakened and has mainly the purpose to impose regularity on the intersections $\partial U_i \cap \partial U_j$, cf.\ Proof of \cref{Thm:RigidityEstimateJonesDomain}.}
	where $a$ is Bilipschitz (in the sense that there exist pair-wise disjoint Lipschitz domains $U_i$, $i=1,\dots,n$ such that $\overline{U} = \bigcup_{i=1}^n \overline{U_i}$ and $a|_{U_i}$ is Bilipschitz).
\end{enumerate}
\end{definition}
This definition is a generalization of stress-free joints as considered in \cite{Jam86}. There, one may think of a stress-free joint as a composite consisting of firmly joint bodies, where each component features a prestrain, given by the prestrain tensors $A_i$, respectively, joined in such a way that the composite can be deformed into some configuration with vanishing prestrain, globally. Hence, there exists some continuous, piece-wise affine map $a:\Omega \to \R^d$ with $a(x) = A_i x + c_i$ if $x \in \Omega_i$ for some decomposition $(\Omega_i)_{i=1}^n$ of $\Omega$.
This yields the necessary condition that for each neighboring domains $\Omega_i$ and $\Omega_j$, the matrizes $A_i$ and $A_j$ must have a rank one difference. More precisely, if $\mathcal{H}^{d-1}(\partial \Omega_i \cap \partial \Omega_j) > 0$ and $\partial \Omega_i \cap \partial \Omega_j \subset \set{n}^\perp$, $n \neq 0$, then
\begin{equation} \label{Eq:Rank_one_compatiblity}
  A_i v = A_j v, \quad \text{for all } v \in \R^d \text{ with } v \cdot n = 0.
\end{equation}

Since we are interested in periodic composites, we shall consider the special case of periodic stress-free joints and define
\begin{definition}[Periodic stress-free joints] \label{Def:StressFreeJointPeriodic}
  We denote by $\operatorname{SFJ}_{\rm per}$ the set of all maps $A:\R^d\to \R^{d\times d}$ that are $Y$-periodic and satisfy \ref{Property:SFJUniformBound} and \ref{Property:SFJisDerivative} for $U=Y$.
\end{definition}
The following proposition shows that any $A\in\operatorname{SFJ}_{\rm per}$ has in fact a unique Bilipschitz potential $a$. In particular, the restriction of $A$ to any admissible set $U$ belongs to $\operatorname{SFJ}(U)$.
\begin{proposition}[Existence of a Bilipschitz potential]\label{Prop:SFJBilipschitz}
Let $A\in\operatorname{SFJ}_{\rm per}$. Then there exists a unique potential $a: \R^d \to \R^d$  that is globally Bilipschitz and onto, and satisfies $A=Da$ a.e.\ in $\R^d$ and $a(0)=0$.
(For the proof see \cref{Sec:ProofSFJ}.)
\end{proposition}
In our paper, we consider a situation where the prestrain tensor $A_h$ is a perturbation of a periodic stress-free joint $A$ in the \replaced[id=KR]{following sense.}{sense that $A_h\to A$ in $L^\infty$ with a first-order correction $h^{-1}(A_h-A)$ that converges in $L^2$:}
\begin{definition}[Perturbation of a periodic stress-free joint]\label{Ass:Prestrain}
  We say $(A_h)\subset L^\infty(\R^d,\R^{d\times d})$ is a perturbation of a periodic stress-free joint $A\in\operatorname{SFJ}_{\rm per}$, if $A_h$ is $Y$-periodic, and as $h\to 0$ we have
  \begin{equation*}
    A_h\to A\text{ in }L^\infty(\R^d,\R^{d\times d}) \quad\text{and}\quad h^{-1}(A_h-A)\text{ converges in }L^2_{\rm loc}(\R^d,\R^{d\times d}).
  \end{equation*}
\end{definition}
\begin{remark}[Definition of the incremental prestrain tensors $B_h$ and $B$]
  Let $(A_h)$ be a perturbation of a stress-free joint $A$ in the sense of \cref{Ass:Prestrain}. We may write the prestrain tensor as a product:
  \begin{equation}\label{def:tildeB}
    A_h = (I + h\tilde{B}_h)A,\qquad\text{where }	\tilde{B}_h:= \tfrac{1}{h} (A_hA^{-1} - I),
  \end{equation}
  and we may define
  \begin{equation}\label{def:B}
	B := \lim_{h \to 0} \tilde{B}_h\text{ (in $L^2_{\rm loc}(\R^d,\R^{d\times d})$)}.
  \end{equation}
  Furthermore, the Neumann series implies that for small $h \ll 1$,
  \begin{equation}\label{def:alternative1}
    A_h(y)^{-1} = A(y)^{-1}(I - h B_h(y)),
  \end{equation}
  where $B_h(y) := \sum_{k=0}^{\infty} (-h)^k \tilde{B}_h(y)^{k+1}$. From \cref{Ass:Prestrain} we conclude that
  \begin{equation}\label{def:alternative2}
    B_h\to B\text{ in }L^2_{\rm loc}(\R^d,\R^{d\times d}),\qquad \limsup_{h \to 0} h\norm{B_h}_{\Leb^\infty(\R^d)} = 0.
  \end{equation}
  In fact, \eqref{def:alternative1} together with \eqref{def:alternative2} are equivalent to the notion introduced in \cref{Ass:Prestrain}. In the paper we shall frequently work with this representation, since it eases the presentation.
\end{remark}

\begin{remark}[Energy scaling in the case of a perturbed stress-free joint.] \label{Rem:EnergyScaling}
Let $(A_h)$ be a perturbation of a stress-free joint $A$ in the sense of \cref{Ass:Prestrain}. Let $a_\eps$ be the potential of the stress-free joint $A(\tfrac{\placeholder}{\eps})$.
Then, using the ansatz $\phi = a_\eps + hu$ for some $u \in \SobW^{1,\infty}(\Omega, \R^d)$, we can show that the minimal energy scales like $h^2$. Indeed, by \eqref{Eq:EnergyFunctionalOnDeformation}, \eqref{Eq:MultiplicativeDecomposition} and \ref{Property:W_Expansion},
\begin{align*}
	\EnergyEhe{\eps}(\phi) &= \int_\Omega W \left( \tfrac{x}{\eps}, (I + h\dif{u}A(\tfrac{x}{\eps}))(I - h B_h(\tfrac{x}{\eps})) \right) \,\dx \\
	&= h^2 \int_\Omega Q\left( \tfrac{x}{\eps}, \dif{u}A(\tfrac{x}{\eps})^{-1} - B_h(\tfrac{x}{\eps})\right)\,\dx + \oclass(h^2).
\end{align*}
Especially, $a_\eps$ is a minimizer, if the prestrain is a pure stress-free joint and its energy scales like $h^2$ in the presence of a perturbation. On the other hand, \ref{Property:W_NonDegeneracy} and a suitable geometric rigidity estimate (see \cref{Thm:RigidityEstimateJonesDomain}) imply that any sequence $(\phi_{\eps,h}) \subset \SobH^1(\Omega, \R^d)$ satisfying $\EnergyEhe{\eps}(\phi_{\eps, h}) \leq ch^2$ is of the form $\phi_{\eps,h} = R_{\eps,h} a_\eps + c_{\eps,h} + \Oclass(h)$ for some constant $c_{\eps, h} \in \R^d$ and rotation $R_{\eps,h} \in \SO(d)$. This motivates us to linearize at the known low-energy state $a_\eps$ and study the minimizers of $u \mapsto h^{-2}\EnergyEhe{\eps}(a_\eps + hu) = \EnergyIhe{\eps}(u)$. 
\end{remark}

One can find a variety of non-trivial (periodic) stress-free joints. Some examples are presented in \cref{Fig:StressFreeJoints}. The simplest example is a laminate, which is depicted in \cref{Fig:StressFreeJoints:laminate}. There it is sufficient to satisfy \eqref{Eq:Rank_one_compatiblity}. We discuss laminates in greater detail in \cref{Sec:Example}. Also a simple example is shown in \cref{Fig:StressFreeJoints:2Dlaminate}, where the periodicity cell decomposes into a checkerboard. It is not hard to think of and draw stress-free joints that arise from this situation. However, these examples might not be to relevant in practice, since very specific kinds of materials would need to be joined for this to work. An interesting and complex stress-free joint relevant for practical purposes is given in \cref{Fig:StressFreeJoints:periodicSFJ}. This examples was found using the theories developed in \cite{Jam86} and \cite{Eri83}. It depicts the joint of blocks of a single material in various orientations. Finally, our theory also allows for situations as presented in \cref{Fig:StressFreeJoints:smoothSFJ} featuring a smooth course of the prestrain. The precise formulas used in \cref{Fig:StressFreeJoints} can be found in \cref{Sec:formulas-of-the-stress-free-joints}.
\smallskip

An important special case is the situation, when $A\equiv I$, i.e., when we have \emph{small prestrain} of order $h$. In fact, the essence of most proofs relies on reducing the situation to this case by transforming the reference domain with help of the Bilipschitz potential $a$ of \cref{Prop:SFJBilipschitz}. Let us anticipate that a technical difficulty, that emerges in this context, is that we need to show that certain functional inequalities for Sobolev functions (in particular, the geometric rigidity estimate) are stable w.r.t.\ Bilipschitz transformations. We discuss this in \cref{Sec:ProofSFJ,Sec:Proof:ExtensionOperatorRigidity}.

\begin{figure}
\centering
\begin{subfigure}[t]{0.3\linewidth}
	\centering
	\includegraphics[width=\textwidth]{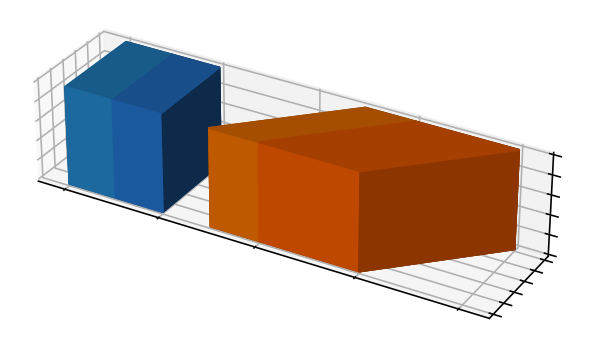}
	\caption{A laminate.}
	\label{Fig:StressFreeJoints:laminate}
\end{subfigure}
\begin{subfigure}[t]{0.3\linewidth}
	\centering
	\includegraphics[width=\textwidth]{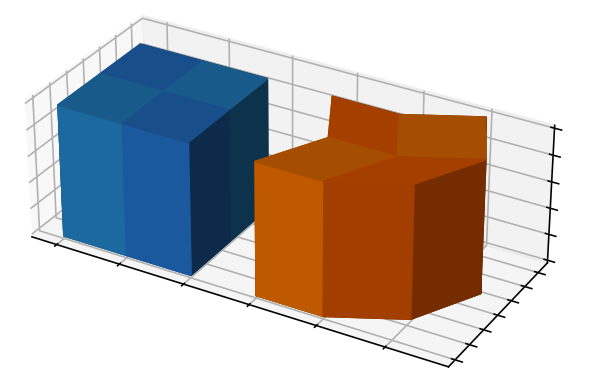}
	\caption{A \enquote{2D-laminate}.}
	\label{Fig:StressFreeJoints:2Dlaminate}
\end{subfigure}
\\
\begin{subfigure}[t]{0.3\linewidth}
	\centering
	\includegraphics[width=\textwidth]{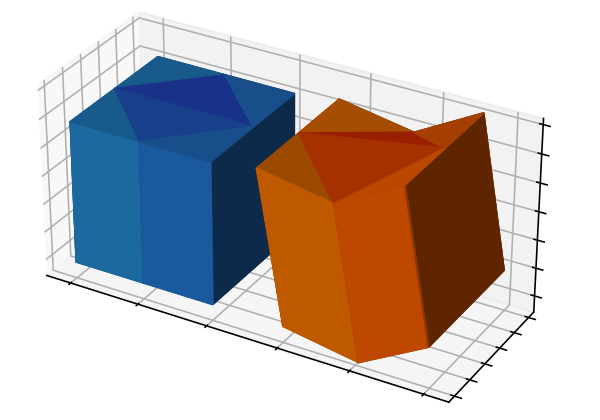}
	\caption{A single material stress-free joint.}
	\label{Fig:StressFreeJoints:periodicSFJ}
\end{subfigure}
\begin{subfigure}[t]{0.3\linewidth}
	\centering
	\includegraphics[width=\textwidth]{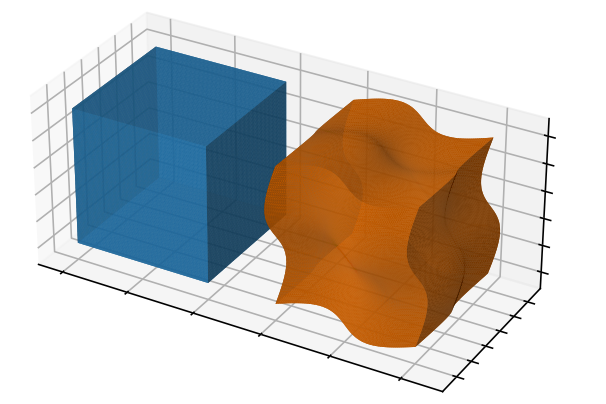}
	\caption{A smooth stress-free joint.}
	\label{Fig:StressFreeJoints:smoothSFJ}
\end{subfigure}
\caption{Images depicting periodic stress-free joints. The images always include the periodicity cell $Y$ on the left-hand side and the deformed cell, according to the deformation $a:Y \to \R^d$, on the right-hand side. The different shades depict the areas where the deformation is affine.}
\label{Fig:StressFreeJoints}
\end{figure}

\section{Main results}
We first discuss homogenization and linearization on the level of the stored energy function. Then, we introduce the effective quantities and study properties of the quadratic term $Q^A_{\hom}$ in the expansion of $W^h_{\hom}$. Finally, we discuss $\Gamma$-convergence of the energy functionals. 

\subsection{Homogenization and linearization of the stored energy function}\label{Sec:AsyExpansion}
In this section, we discuss homogenization and linearization of the stored energy function $W^h$, see \eqref{Eq:MultiplicativeDecomposition}. We assume that $W$ satisfies Assumption~\ref{Ass:ElasticMaterialLaw} and that $(A_h)$ is a perturbed stress-free joint in the sense of \cref{Def:StressFreeJoint}. We are especially interested in the energy well of the homogenized stored energy function $W^h_{\hom}$ (see \eqref{Eq:Multicell_formula}) and its dependence on the prestrain tensor $A_h$ and the material law $W$. To understand the latter is difficult, since $W$ is non-convex and since $A_h$ enters the definition of $W^h$ non-linearly. We therefore focus on the regime $0<h\ll 1$, i.e., when $A_h$ is close to a periodic, stress-free joint $A$.
\smallskip

For the presentation it is useful to view \textit{homogenization on the level of the energy density} as an operation $[\placeholder]_{\hom}$ that maps a measurable, $Y$-periodic function $V:\R^d\times\R^{d\times d}\to[0,\infty]$ to a function $[V]_{\hom}:\R^{d\times d}\to[0,\infty]$ by appealing to the so-called \textit{multi-cell homogenization formula}:
\begin{equation}\label{def:mchom}
  [V]_{\hom}(F):=\inf_{k \in \N}\inf_{\varphi \in \SobW^{1,\infty}_{\rm per}(kY,\R^d)} \fint_{kY} V\big(y, F + \dif{\varphi}(y)\big) \,\dx[y],\qquad F\in\R^{d\times d}.
\end{equation}
The motivation of this definition is the following:

\begin{remark}[Non-convex homogenization]
  Let $V:\R^d\times\R^{d\times d}\to[0,\infty]$ be $Y$-periodic, measurable and suppose that it satisfies the $p$-growth- and $p$-Lipschitz-condition
\begin{equation}\label{Eq:growth_cond}
  \begin{cases}
    \tfrac{1}{C} \abs{F}^p - C \leq V(y,F) \leq C (1 + \abs{F}^p), & F \in \R^{d\times d},\\ 
    \abs{V(y,F) - V(y,G)} \leq C(1 + \abs{F}^{p-1} + \abs{G}^{p-1}) \abs{F - G}, & F,G \in \R^{d\times d},
  \end{cases}
\end{equation}
for some $1<p<\infty$. Then the classical result of S.~M\"uller on non-convex homogenization implies that the integral functional $\varphi\mapsto\int_\Omega V(\tfrac{x}{\eps},D\varphi(x))\,\dx$ $\Gamma$-converges to the homogenized functional $\varphi\mapsto\int_\Omega [V]_{\hom}(D\varphi(x))\,\dx$. We note that in the definition of $[V]_{\hom}$ the space $W^{1,\infty}_\mathrm{per}(kY,\R^d)$ can be replaced by $\SobW^{1,p}_{\rm per}(kY,\R^d)$ \cite{Muel87} thanks to the $p$-growth- and $p$-Lipschitz-condition. We defined $[\cdot]_{\hom}$ with $\SobW^{1,\infty}$, since we want to highlight that the homogenization procedure is (to some extend) independent of the growth exponent of the integrand.
\end{remark}
\medskip

To determine what we can expect, we first review what is already known about the homogenized stored energy function in the simplest setting, namely, in the case without prestrain, i.e., $A_h=A=I$. In that case it was shown in \cite{MN11,GN11} that $[W]_{\hom}$ admits a quadratic Taylor expansion at identity:
\begin{equation}\label{eq:unprestrained:expansion}
  [W]_{\hom}(I+G)=[Q]_{\hom}(G)+o(|G|^2)
\end{equation}
where $Q$ is defined by
\begin{equation}\label{eq:def:Q}
    Q(y,G) := \lim_{h \to 0} \tfrac{1}{h^2} W(y,I + hG), \qquad G \in \R^{d\times d}.
\end{equation}
In a nutshell this means that homogenization and linearization commute: The quadratic term in the expansion of the homogenized stored energy function is given by the homogenization of the quadratic term in the expansion of $W$ at identity.
We note that the homogenization of $Q$ is much more simple than the one for $W$. Indeed, thanks to \ref{Property:W_FrameIndifference} -- \ref{Property:W_Expansion}, the limit in \eqref{eq:def:Q} exists and defines a positive quadratic form that satisfies the ellipticity condition
\begin{equation}\label{Eq:UniformBounds_Q}
  \alpha_{\rm el} \abs{\sym G}^2 \leq Q(y,G) \leq \beta_{\rm el}\abs{\sym G}^2, \qquad \text{for all } G \in \R^{d\times d},
\end{equation}
and a.e.\ $y\in\R^d$. In particular, $Q$ is convex and thus, as shown by \cite[Lem.~4.1]{Muel87}, the multi-cell homogenization formula reduces to a single-cell homogenization formula that can be represented with help of a corrector field: For all $G\in\R^{d\times d}$ we have
\begin{equation}\label{eq:quadratic-corrector}
  [Q]_{\hom}(G)=\min_{\varphi\in W^{1,\infty}_{\rm per}(Y,R^d)}\fint_Y Q\big(y,G+D\varphi(y)\big)\,dy=\fint_Y Q\big(y,G+D\varphi_G(y)\big)\,dy,
\end{equation}
where $\varphi_G\in H^1_{\rm per}(Y,\R^d)$ denotes the unique (up to an additive constant) solution to the periodic corrector equation
\begin{equation*}
 -\operatorname{Div}\big(\mathbb{L}_Q(G+D\varphi_G)\big) = 0 \qquad\text{in }\mathscr D'(\R^d),
\end{equation*}
where $\mathbb{L}_Q$ denotes the 4th-order tensor associated with $Q$ via the polarization identity
\begin{equation}
	F : \mathbb{L}_Q(y) G = \tfrac{1}{2} \Big( Q(y,F + G) - Q(y,F) - Q(y,G) \Big).
\end{equation}
In \cite*{NS18, NS19} it is shown that under additional regularity assumptions on $W$ (both w.r.t.\ $F$ and $y$), $[W]_{\hom}$ is of class $C^3$ and admits a representation via a single-cell homogenization~-- both in an open neighborhood of $\SO(d)$. On the other hand, since nonlinear laminates may buckle under compression (see \cite[Thm.~4.3]{Muel87}) it is clear that the validity of the single-cell formula and the commutativity property fail for deformations $F$ away from $\SO(d)$.
In view of this it is reasonable to focus on the case of deformations that are asymptotically close to a ground state and on prestrains that are asymptotically close to a stress-free joint.
It is instructive to first discuss a special case of our result, namely the stored energy function $W^0(y,F)=W(y,FA(y)^{-1})$ whose prestrain is of the form of an unperturbed stress-free joint $A\in \operatorname{SFJ}_{\rm per}$. In that case the ground states are explicitly known:
\begin{equation*}
  \operatorname{Arg\,min}[W^0]_{\hom}=\class{F=R\bar A}{R\in\SO(d)},\qquad  \bar A:=\fint_YA(y)\dx[y].
\end{equation*}
Indeed, as shown in \cref{Lem:periodic_derivative} below, this follows from the fact that $A \equiv \bar A+\dif{\varphi}$ for some $\varphi \in \SobW^{1,\infty}_\mathrm{per}(Y, \R^d)$ and the identity $W^0(y,A(y))=W(y,I)=0\leq \inf [W^0]_{\hom}$. Furthermore, we shall see that  $[W^0]_{\hom}$ admits a quadratic expansion at $\bar A$ of the form
\begin{equation}\label{eq:exp:stress-free-joint}
  [W^0]_{\hom}(\bar A+G)=[Q^A]_{\hom}(G)+o(|G|^2),
\end{equation}
where $Q^A$ denotes the quadratic form obtained by linearizing $W^0$ at $A$:
\begin{equation}
  Q^A(y,G) := \lim_{h \to 0} \tfrac{1}{h^2} W^0(y,A(y) + hG) = Q(y, GA(y)^{-1}).
\end{equation}
While the expansion \eqref{eq:exp:stress-free-joint} is rather similar to \eqref{eq:unprestrained:expansion}, the situation changes in the case of a perturbed stress-free joint $A_h=(I+h B_h)A$ as considered in the definition of $W^h$. The next theorem is the main result of this section. Roughly speaking it establishes the expansion
\begin{equation*}
  \tfrac{1}{h^2}[W^h]_{\hom}(\bar A+hG)= R^A(B)+ [Q^A]_{\hom}(G-B_{\hom}\bar A)+\oclass(\abs{G}^2),
\end{equation*}
which holds for $h\to 0$ in a quantitative sense that is made precise in the theorem. The quadratic form on the right-hand side of the expansion is the same as in \eqref{eq:exp:stress-free-joint}. However, in contrast to \eqref{eq:exp:stress-free-joint}, the quadratic term on the right-hand side features an effective incremental prestrain tensor $B_{\hom}$ and includes a residual energy $R^A(B)$. We define these quantities in \cref{Def:Formula_Bhom} with help of the homogenization correctors associated with $Q^A$. We note that in the special case $A \equiv I$, this form of the expansion already appeared implicitly in previous works in the context of simultaneous homogenization and dimension reduction (cf.\ \cite{BNS20,BNPS22,BGNPP23}).

\begin{theorem}[Non-degeneracy and quadratic expansion of $W^h_{\hom}$]\label{Thm:Expansion_Whom}
Let $W$ satisfy Assumption~\ref{Ass:ElasticMaterialLaw}, let $(A_h)$ be a perturbation of a periodic stress-free joint $A\in \operatorname{SFJ}_{\rm per}$ (see~Definition~\ref{Ass:Prestrain}), and define $W^h$ by \eqref{Eq:MultiplicativeDecomposition}. Set $W^h_{\hom}:=[W^h]_{\hom}$, $Q^A_{\hom}:=[Q^A]_{\hom}$, and $\bar A:=\int_Y A(y)\dx[y]$.  Then the following statements hold.
\begin{enumerate}[(a)]
	\item (Frame indifference) For all $F \in \R^{d \times d}$, $h > 0$ and $R \in \SO(d)$ we have
	\begin{equation} \label{Eq:FrameIndifference_Whom}
		W^h_{\hom}(RF) = W^h_{\hom}(F).
	\end{equation}
	\item (Non-degeneracy) There exists some $\alpha > 0$ and $h_0 > 0$ such that for all $F \in \R^{d \times d}$ and $0 < h \leq h_0$
	\begin{equation} \label{Eq:NonDegeneracy_Whom}
		W^h_{\hom}(F) \geq \tfrac{1}{\alpha} \dist^2(F\bar{A}^{-1}, \SO(d)) - \alpha h^2.
	\end{equation}
	\item (Asymptotic expansion) There exists a continuous, increasing map $\rho: [0,\infty) \to [0,\infty]$ with $\rho(0) = 0$, such that for all $h > 0$ and $G \in \R^{d \times d}$
	\begin{equation}\label{Eq:TaylorExpansion_Whom}
		\abs{\tfrac{1}{h^2} W^h_{\hom}(\bar{A} + hG) - \left(Q^A_{\hom}(G - B_{\hom}\bar{A}) + R^A(B)\right)} \leq (1 + \abs{G}^2) \rho(h + \abs{hG}),
  \end{equation}
  where $B_{\hom}$ and $R^A(B)$ are given by \cref{Def:Formula_Bhom}.
\end{enumerate}
(For the proof see \cref{Sec:Proof:AsyExpansionWhom2}.)
\end{theorem}
We note that the ground state of $W^h_{\hom}$ is not explicitly known. In fact, under the assumptions of the theorem (which does not impose growth conditions on $W$), it is even not clear that $W^h_{\hom}$ attains its minimum. For this reason, in the theorem we consider an expansion of $W^h_{\hom}$ around the deformation $\bar{A}$, which is an asymptotic ground state. This is made precise in the following corollary, which yields an asymptotic expansion for almost minimizers of $W^h_{\hom}$:

\begin{corollary}\label{Cor:ExpansionMinimizer}
  In the situation of \cref{Thm:Expansion_Whom} let $(F_h^*) \subset \R^{d\times d}$ denote a sequence of almost minimizers for $(W^h_{\hom})$ in the sense that
\begin{equation*}
	\limsup_{h \to 0} \tfrac{1}{h^2} \Big| W^h_{\hom}(F_h^*) - \inf_{F \in \R^{d\times d}} W^h_{\hom}(F) \Big| = 0.
\end{equation*}
Then there exist rotations $R_h \in \SO(d)$, such that
\begin{equation}
	F_h^* = R_h(I + h B_{\hom})\bar{A} + \oclass(h).
\end{equation}
(For the proof see \cref{Sec:Proof:AsyExpansionWhom2}.)
\end{corollary}
Furthermore, we observe that the quadratic term in the expansion (\ref{Eq:TaylorExpansion_Whom}) is in fact the homogenization of the quadratic term in the expansion of $W^h$ at $A(y)$:
\begin{lemma}\label{lemma:homogenizationaftelinearization}
  In the situation of \cref{Thm:Expansion_Whom} we have up to a subsequence
  \begin{align}
    \lim\limits_{h\to0}\tfrac{1}{h^2}W^h(A(y)+h\placeholder)=&\,Q^A(y,\placeholder-B(y)A(y)),
    \label{lemma:homogenizationaftelinearization:2}
   \intertext{and}
    \label{lemma:homogenizationaftelinearization:1}
    \big[(y,G) \mapsto Q^A\big(y,G-B(y)A(y)\big)\big]_{\hom}=&\,R^A(B)+Q^A_{\hom}(\placeholder-B_{\hom}\bar A),
  \end{align}
  where $B_{\hom}$ and $R^A(B)$ are given by \cref{Def:Formula_Bhom}.
  (For the proof see \cref{Sec:ProofHomogenPerturbation}.)
\end{lemma}

Finally, the following lemma establishes a rigorous connection between the expansion of $W^h$ at $A(y)$ and the expansion of $W^h_{\hom}$ at $\bar{A}$.

\begin{lemma}\label{Lem:homandprestrain}
  In the situation of Theorem~\ref{Thm:Expansion_Whom} we have
  \begin{equation*}
    \Big[(y,G) \mapsto W^h(y,A(y)+G)\Big]_{\hom} = \Big[W^h\Big]_{\hom}(\bar A+\placeholder).
  \end{equation*}
  (For the proof see \cref{Sec:Proof:AsyExpansionWhom1}.)
\end{lemma}

We may summarize the structural implications of \cref{Thm:Expansion_Whom,lemma:homogenizationaftelinearization,Lem:homandprestrain} with help of the following commuting diagram:
\begin{equation}\label{Eq:Diagram_perturbation}
\begin{tikzpicture}[baseline={([yshift=-.5ex]current bounding box.center)}]
	\node (1) at (0,0) {$\tfrac{1}{h^2} W^h\left(y, A(y) + h\placeholder\right)$};
	\node (2) at (5,0) {$Q^A\left(y, \placeholder - B(y)A(y)\right)$};
	\node (3) at (0,-2) {$\tfrac{1}{h^2} W^h_{\hom}\left(\bar{A} + h\placeholder\right)$};
	\node (4) at (5,-2) {$R^A(B)+Q^A_{\hom}(\placeholder - B_{\hom}\bar{A})$.};
	\draw[->]  (1) -- (2) node[above, midway] {\footnotesize\upshape(1)};		
	\draw[->]  (1) -- (3) node[left, midway] {\footnotesize\upshape(3)};
	\draw[->]  (2) -- (4) node[right, midway] {\footnotesize\upshape(2)};
	\draw[->]  (3) -- (4) node[above, midway] {\footnotesize\upshape(4)};
\end{tikzpicture}
\end{equation}
(1) and (4) stand for linearization (i.e.\ taking the limit $h\to 0$), while (2) and (3) stand for homogenization (i.e., applying $[\placeholder]_{\hom}$). Linearization (4) is part of \cref{Thm:Expansion_Whom}. (1) and (2) are provided by \cref{lemma:homogenizationaftelinearization} and (3) is shown in \cref{Lem:homandprestrain}. In a nutshell (2) states that $\big[Q^A(y,\placeholder-BA)\big]_{\hom}$ can be decomposed into a residual energy term and a quadratic form. This result uses a corrector representation of the homogenized quadratic form similar to \eqref{eq:quadratic-corrector}. We explain this in detail in the next section. 
Thus, it is possible to extract the dependence on the incremental prestrain $B$ from the homogenized energy. The energy can however not be decoupled from the stress-free joint $A$, as an example in \cref{Sec:Example} shall reveal.

\subsection{Properties of \texorpdfstring{$Q^A_{\hom}$}{QAhom} and definition of the effective quantities}
\label{Sec:Characterization_Bhom}

In this section we present the definition of the effective quantities that appear in \cref{Thm:Expansion_Whom,lemma:homogenizationaftelinearization}. We recall that the definition of $Q^A$ invokes the quadratic term $Q$, which thanks to \hyperref[Def:MaterialClass_W]{(W1) - (W3)}, satisfies the ellipticity conditions \eqref{Eq:UniformBounds_Q}.
The homogenized quadratic form $Q^A_{\hom}$ inherits this non-degeneracy property in the following form:
\begin{lemma}[Non-degeneracy]\label{L:coercivity}
  There exists a constant $c>0$ (only depending on $\alpha_{\rm el}$ and $\beta_{\rm el}$, $\norm{A}_{\Leb^\infty(Y)}$ and $\norm{A(\placeholder)^{-1}}_{\Leb^\infty(Y)}$) such that for all $G \in \R^{d\times d}$,
  \begin{equation}\label{Eq:Qhom_properties}
    \tfrac{1}{c}\abs{\sym_{\bar{A}} G}^2 \leq Q^A_{\hom}(G) \leq c \abs{\sym_{\bar{A}} G}^2,
  \end{equation}
  where
  \begin{equation}\label{Eq:Def_symA}
    \sym_{\bar{A}}: \R^{d\times d} \to \R^{d\times d}_{\sym}\bar{A}, \quad \sym_{\bar{A}} := \sym(G\bar{A}^{-1})\bar{A}.
  \end{equation}
  \textit{(For the proof see \cref{Sec:Proof:AsyExpansionWhom1}.)}
\end{lemma}
Especially, \eqref{Eq:Qhom_properties} shows that the left-hand side of \eqref{lemma:homogenizationaftelinearization:1} admits a unique minimizer up to the symmetry $\sym_{\bar{A}}$, which is given by $B_{\hom}\bar{A}$.
We introduce the effective incremental prestrain $B_{\hom}$ following the approach in \cite[§3]{BNS20}. The main point of this construction is to obtain the decomposition \eqref{lemma:homogenizationaftelinearization:1} which establishes the direction (2) in the diagram \eqref{Eq:Diagram_perturbation}. The main idea is to utilize an orthogonal decomposition of the Hilbert space $\Leb^2(Y, \R^{d\times d}_{\sym})$ equipped with the scalar product
\begin{equation}
  \left( \Phi, \Psi \right)_Q := \int_Y \Phi(y) : \mathbb{L}_Q(y) \Psi(y) \,\dx[y].
\end{equation}
Thanks to \eqref{Eq:UniformBounds_Q}, $(\cdot,\cdot)_Q$ defines a scalar product that is equivalent to the standard one. We denote the associated norm by $\norm{\placeholder}_Q$ and write $\operatorname{P}_U$ for the orthogonal projection in this Hilbert space onto some closed, convex subset $U \subset \Leb^2(Y, \R^{d\times d}_{\sym})$.
We consider the subspace
\begin{equation}\label{def:S}
  \mathcal{S} := \class{\sym(\dif{\varphi}A(\placeholder)^{-1})}{\varphi \in \SobH^1_{\mathrm{per}}(Y,\R^d)},
\end{equation}
and define $\mathcal{O}$ as the orthogonal complement of $\mathcal S$ in $\mathcal S+\R^{d\times d}_{\sym}$ (which we consider as subspace of $(\Leb^2(Y,\R^{d\times d}_{\sym}), \norm{\placeholder}_Q)$). We thus obtain a decomposition of $\Leb^2(Y,\R^{d\times d}_{\sym})$ into three orthogonal subspaces:
\begin{equation*}
  \Leb^2(Y,\R^{d\times d}_{\sym})=\mathcal S+\mathcal O+(\mathcal S+\R^{d\times d}_{\sym})^\perp.
\end{equation*}
Note that $\mathcal S$ is closed as follows from Korn's inequality, see \cref{Cor:KornPeriodic} below.
With this at hand we are ready to define the residual energy and the effective incremental prestrain.
\begin{lemma}[Effective incremental prestrain $B_{\hom}$] \label{Thm:FormulaBhom}
  The projection $\R^{d\times d}_{\sym}\to\mathcal O$, $G\mapsto \operatorname{P}_\mathcal{O}(G)$ is an isomorphism. In particular, for every $B\in L^2(Y,\R^{d\times d})$ there exists a unique $B_{\hom}\in \R^{d\times d}_{\sym}$ such that
  \begin{equation}\label{Eq:DefinitionBhom}
    \operatorname{P}_\mathcal{O}(B_{\hom}) = \operatorname{P}_\mathcal{O}(\sym B).
  \end{equation}
  \textit{(For the proof see \cref{Sec:ProofHomogenPerturbation}.)}
\end{lemma}
\begin{definition}[Definition of $B_{\hom}$ and $R^A(B)$]\label{Def:Formula_Bhom}
Let $B$ be given by \eqref{def:alternative2}.  We define the \emph{effective incremental prestrain} as the unique matrix $B_{\hom} \in \R^{d\times d}_{\sym}$ that satisfies \eqref{Eq:DefinitionBhom} and we define the \emph{residual energy} by
\begin{equation}
	R^A(B) := \norm{\operatorname{P}_{(\mathcal{S}+\R^{d\times d}_{\sym})^\perp}(\sym B)}_Q^2.
\end{equation}
\end{definition}

In \cref{Sec:ProofHomogenPerturbation} we show that with these definitions \cref{lemma:homogenizationaftelinearization} is satisfied.
The definition of $B_{\hom}$ and $R^A(B)$ is rather abstract. Using the method of correctors, we obtain an algorithmic characterization.

\newcommand{\Qhom}{\mathbf{Q}}

\begin{proposition}[Algorithmic characterization]  \label{Prop:Bhom_algorithmic}
Let $s := \frac{d(d+1)}{2}$, $\class{G_i}{i=1,\dots,s}$ denote a basis of $\R^{d\times d}_{\sym}$ and
\begin{equation*}
	\operatorname{emb}: \R^s \to \R^{d\times d}_{\sym},\quad \xi\mapsto \sum_{i=1}^s \xi_i G_i,
\end{equation*}
the linear isomorphism describing the vector representation of matrizes in $\R^{d\times d}_{\sym, \bar{A}}$ subject to this basis. For $G \in \R^{d \times d}$ we define the \emph{corrector} $\varphi_G \in \SobH^1_{\mathrm{per},0}(Y,\R^d)$ as the unique minimizer of
\begin{equation}\label{Eq:correctors}
	\varphi \in \SobH^1_{\mathrm{per},0}(Y,\R^d) \mapsto \int_Y Q\left(y, G + \dif{\varphi}(y)A(y)^{-1}\right) \,\dx[y].
\end{equation}
Furthermore, we define the symmetric and positive definite matrix $\Qhom \in \R^{s \times s}$ by
\begin{equation}\label{Eq:matrix_representing_Q}
  \Qhom_{ij} := \int_Y \big(G_i + \dif{\varphi_{G_i}}(y)A(y)^{-1}\big) : \mathbb{L}_Q(y) \big(G_j + \dif{\varphi_{G_j}}(y)A(y)^{-1}\big) \,\dx[y],
\end{equation}
and $b \in \R^s$ by
\begin{equation}\label{Eq:vector_representing_B}
	b_i := \int_Y \big(G_i + \dif{\varphi_{G_i}}(y)A(y)^{-1}\big) : \mathbb{L}_Q(y) B(y) \,\dx[y].
\end{equation}
Then
\begin{equation}\label{Eq:char_Bhom}
  \begin{aligned}
    &B_{\hom} = \operatorname{emb}\big(\Qhom^{-1}b\big)\quad\text{and}\\
    &\forall G\in\R^{d\times d}\,:\,Q^A_{\hom}(G\bar{A}) = \xi\cdot \Qhom\,\xi\text{ with }\xi=\operatorname{emb}^{-1}\big(\sym G\big).
  \end{aligned}
\end{equation}
(For the proof see \cref{Sec:ProofHomogenPerturbation}.)
\end{proposition}

\subsection{Linearization and homogenization of the integral functionals}
\label{Sec:LinAndHomResults}

In this section, we upgrade the convergences of the diagram \eqref{Eq:Diagram_perturbation} to the level of $\Gamma$-convergence for the respective integral functionals. We are interested in the minimization of $\EnergyEhe{\eps}$, see \eqref{Eq:EnergyFunctionalOnDeformation}. To study this, we recall the functionals $\EnergyIhe{\eps}, \EnergyIline{\eps}, \EnergyIhhom$ and $\EnergyIlinhom$ defined in \eqref{D:rescaledfunctionals}. These functionals describe the elastic energy on the level of the scaled displacement $u$ which is associated with a deformation $\phi$ by means of the expansion $\phi = a_\eps + hu$, where $a_\eps$ denotes the asymptotic minimizer defined by $\dif{a_\eps} = A(\tfrac{\placeholder}{\eps})$. In particular, we have $\EnergyIhe{\eps}(u)=h^{-2}\EnergyEhe{\eps}(\varphi)$. In the case without prestrain (that is $A_h\equiv I$) it has already been shown in several contributions, e.g.\ \cite{DNP02,NeuPHD,MN11,ADD12}, that these functionals can be rigorously obtained in the sense of $\Gamma$-convergence. We extend these results to the case of a perturbed stress-free joint. 
To be compatible with the ansatz $\phi = a_\eps + hu$ in \cref{Rem:EnergyScaling}, we consider well-prepared boundary conditions. On the level of $u$ this can be expressed by fixing $u = g$ for some fixed $g \in \SobW^{1,\infty}(\Omega, \R^d)$ on some part of the boundary $\Gamma \subset \R^d$. We suppose that $\Gamma$ is closed with positive $(d-1)$-dimensional Hausdorff-measure.

\begin{definition}[Boundary conditions]\label{Def:boundary_conditions}
We denote the closure in $\SobH^1(\Omega,\R^d)$ of the set of functions $u \in \SobW^{1,\infty}(\Omega, \R^d)$ with $u = g$ on $\Gamma$ by $\SobH^1_{\Gamma,g}(\Omega, \R^d)$.
\end{definition}

Our main results can be displayed in the following diagram.

\begin{theorem}[$\Gamma$-convergence] \label{Thm:gamma_convergence}
We have
\begin{mathcenter}
\begin{equation}
\begin{tikzpicture}[baseline={([yshift=-.5ex]current bounding box.center)}]
	\node (1) at (0,0) {$\EnergyIhe{\eps}$};
	\node (2) at (2,0) {$\EnergyIline{\eps}$};
	\node (3) at (0,-2) {$\EnergyIhhom$};
	\node (4) at (2,-2) {$\EnergyIlinhom$,};
	\node (5) at (1.2,-0.8) {\footnotesize\upshape(5)};
	\draw[->]  (1) -- (2) node[above, midway] {\footnotesize\upshape(1)};
	\draw[->]  (3) -- (4) node[above, midway] {\footnotesize\upshape(4)};
	\draw[->]  (1) -- (3) node[left, midway] {\footnotesize\upshape(3)};
	\draw[->]  (2) -- (4) node[right, midway] {\footnotesize\upshape(2)};
	\draw[->]  (1) -- (4);
\end{tikzpicture}
\end{equation}
\end{mathcenter}
where {\upshape(1)} and {\upshape(4)} mean $\Gamma$-convergence w.r.t.\ weak convergence in $\SobH^1_{\Gamma, g}(\Omega, \R^d)$ as $h \to 0$, {\upshape(2)} as $\eps \to 0$ and {\upshape(5)} is the simultaneous limit as $(\eps,h) \to 0$. Finally, {\upshape(3)} means $\Gamma$-convergence w.r.t.\ weak convergence in $\SobW^{1,p}_{\Gamma, g}(\Omega, \R^d)$ as $\eps \to 0$, provided $W$ satisfies the additional growth and Lipschitz conditions \eqref{Eq:growth_cond}.
(For the proofs see \cref{Sec:ProofLinearization,Sec:ProofSimultaneous}.)
\end{theorem}

Note that (3) has been shown in \cite{Muel87}. In \cref{Sec:ProofLinearization}, we show (1) and (4) as consequences of a more general statement. Convergence (2) is a standard result of homogenization of a quadratic functional. We sketch the proof in \cref{Sec:ProofSimultaneous} and use it to establish the simultaneous limit (5). One can still make sense of (3) and (4), if \eqref{Eq:growth_cond} is not satisfied. For the situation without prestrain, i.e.\ $A_h = A = I$, it was shown in \cite{NeuPHD} that (4) can still be shown with $\EnergyIhhom := \Gamma\text{-}\liminf_{\eps \to 0} \EnergyIhe{\eps}$. We propose that the arguments can be adapted to the more general situation considered in this paper. However, we omit proof for this claim in this work.
Furthermore, we establish the following equi-coercivity estimates.

\begin{theorem}[Equi-coercivity]\label{Thm:equi-coercivity}
There exists a constant $C > 0$ such that for all small $\eps, h > 0$ and $u \in \SobH^1_{\Gamma,g}(\Omega, \R^d)$ we have
\begin{subequations}
\begin{align}
\label{Eq:equi_coercivity_Ihe}
	\norm{u}^2_{\SobH^1(\Omega)} &\leq C \left(\EnergyIhe{\eps}(u) + 1 \right), \\
\label{Eq:equi_coercivity_Iline}
	\norm{u}^2_{\SobH^1(\Omega)} &\leq C \left(\EnergyIline{\eps}(u) + 1 \right), \\
\label{Eq:equi_coercivity_Ihhom}
	\norm{u}^2_{\SobH^1(\Omega)} &\leq C \left(\EnergyIhhom(u) + 1 \right).
\end{align}
\end{subequations}
(For the proof see \cref{Sec:ProofLinearization,Sec:ProofSimultaneous}.)
\end{theorem}

A direct consequence of the $\Gamma$-convergences are the convergences of infima and (almost) mini\-mizers of the functional sequences towards minima and minimizers of the limits (see \cite[§7]{Dal93}).
Since we establish $\Gamma$-convergences w.r.t.\ weak convergence in $\SobH^1_{\Gamma, g}(\Omega, \R^d)$, the sequences of (almost) minimizers a priori only converge weakly in $\SobH^1$. We prove, that some of these convergences can be improved to strong convergence a posteriori. For the homogenization limits we consider the notion of strong two-scale convergence, which we state with help of the periodic unfolding operator $\mathcal{T}_\eps$ (see \cref{Sec:ProofSimultaneous} for its precise definition).

\begin{proposition}\label{Prop:strong_conv}
Let $u \in \SobH^1(\Omega, \R^{d\times d})$ and $\varphi$ denote the unique minimizer of \eqref{Eq:homogenized_limit} for $u$. The following statements hold.
\begin{enumerate}[(a)]
	\item Let $u_h \wto u$ weakly in $\SobH^1(\Omega, \R^d)$, $\eps > 0$ be fixed and assume $\EnergyIhe{\eps}(u_h) \to \EnergyIline{\eps}(u)$ as $h \to 0$. Then, $\dif{u_h} \to \dif{u}$ strongly in $\Leb^2(\Omega, \R^{d\times d})$.
	\item Let $u_h \wto u$ weakly in $\SobH^1(\Omega, \R^d)$ and assume $\EnergyIhhom(u_h) \to \EnergyIlinhom(u)$. Then, $\dif{u_h} \to \dif{u}$ strongly in $\Leb^2(\Omega, \R^{d\times d})$.
	\item Let $u_\eps \wto u$ weakly in $\SobH^1(\Omega, \R^d)$ and assume $\EnergyIline{\eps}(u_\eps) \to \EnergyIlinhom(u)$. Then, $\mathcal{T}_\eps\dif{u_\eps} \to \dif{u} + \dif{_y \varphi}$ strongly in $\Leb^2(\Omega \times Y, \R^{d\times d})$.
	\item Let $u_{\eps,h} \wto u$ weakly in $\SobH^1(\Omega, \R^d)$ and assume $\EnergyIhe{\eps}(u_{\eps,h}) \to \EnergyIlinhom(u)$ as $(\eps,h) \to 0$. Then, $\mathcal{T}_\eps\dif{u_{\eps,h}} \to \dif{u} + \dif{_y \varphi}$ strongly in $\Leb^2(\Omega \times Y, \R^{d\times d})$ as $(\eps, h) \to 0$.
\end{enumerate}
(For the proof see \cref{Sec:ProofLinearization,Sec:ProofSimultaneous}.)
\end{proposition}

\subsection{Geometric rigidity estimate and Korn's inequality on Jones domains} \label{Sec:GeometricRigidity}
An essential ingredient to establish the equi-coercivity estimates above is the geometric rigidity estimate due to Friesecke, James and M\"uller \cite{FJM02}. However, in this paper we are faced with some additional complexity. The argument for the proofs depend heavily on the fact that $A$ is the gradient of a Bilipschitz map $a$ and an application of the geometric rigidity estimate on $a(\Omega)$. One complexity here is that $a(\Omega)$ is not necessarily a Lipschitz domain, see \cite{Lic19} for a counter example. The proof given in \cite{FJM02} does not extend easily to such domains. We show that the rigidity estimate does indeed hold on domains like this with controlled constants and does even hold on more general domains. For this, we introduce the class of Jones domains.

\begin{definition}[Jones domain, cf.\ {\cite{Jon81}}]
Let $U \subset \R^d$, $\delta > 0$, $e \in (0,1]$. We say $U$ is a Jones domain or more precisely an $(e,\delta)$-domain\footnote{We use $e$ here instead of the standard notation $\eps$, since $\eps$ is already reserved for the periodicity.}, if for all $x,y \in U$ with $\abs{x-y} < \delta$, there exists a rectifiable curve $\gamma: [0,1] \to U$ with $\gamma(0) = x$, $\gamma(1) = y$ and
\begin{equation}
	\operatorname{len}(\gamma) \leq \tfrac{1}{e} \abs{x - y}, \qquad
	\dist(z, \partial U) \geq \frac{e \abs{x - z}\abs{y - z}}{\abs{x - y}} \text{ for all } z \in \gamma([0,1]).
\end{equation}
\end{definition}

Note that Lipschitz domains are Jones domains and $e$ and $\delta$ are controlled by transformation of the domain by a Bilipschitz map. Furthermore, as shown in \cite{ADD12}, for the strong convergence of (almost) minimizers in $\SobH^1$ we require a slightly more general version of the rigidity estimate with mixed growth conditions (see \cite{CDM14}). This version provides estimates for decompositions into parts with lower and higher integrability. For this we introduce the notation of decompositions $V = F + G$ in $\Leb^p+\Leb^q(U,\R^m)$, which is shorthand for
\begin{equation} \label{Eq:NotationMixedGrowthDecompositions}
	V = F + G \text{ a.e.\ in } U, \qquad F \in \Leb^p(U, \R^m), G \in \Leb^q(U, \R^m)
\end{equation}
for $U \subset \R^d$ and $V: U \to \R^m$ measurable. We use analogous notation for more decompositions into more than two terms. A suitable application is usually a decomposition like $V = \mathds{1}_{\set{V > c}}V + \mathds{1}_{\set{V \leq c}}V$. Such a decomposition is applied  e.g.\ in \cite{CDM14} to show a uniform integrability statement related to the rigidity estimate, see \cref{Prop:uniform_integrability}.

The following theorem extends the geometric rigidity estimate and Korn's inequality to Jones domains with mixed growth conditions and prestrained deformations:

\begin{theorem} \label{Thm:RigidityEstimateJonesDomain}
Let $U \subset \R^d$ a bounded, connected $(e,\delta)$-domain, $1 < p \leq q < \infty$ and $A \in \operatorname{SFJ}(U)$. Then, there exists a constant $C = C(U,A,p,q) > 0$, such that for all $u \in \SobW^{1,1}(U, \R^d)$ the following statements hold.
\begin{enumerate}[(a)]
	\item (Geometric rigidity) Given a decomposition $\dist(\dif{u}A(\placeholder)^{-1}, \SO(d)) = F_{\dist(\dif{u}A(\placeholder)^{-1}, \SO(d))} + G_{\dist(\dif{u}A(\placeholder)^{-1}, \SO(d))}$ in $\Leb^p+\Leb^q(U)$, there exist $R \in \SO(d)$ and a decomposition $\dif{u}A(\placeholder)^{-1} - R = F_{\dif{u}A(\placeholder)^{-1} - R} + G_{\dif{u}A(\placeholder)^{-1} - R}$ in $\Leb^p+\Leb^q(U, \R^{d\times d})$, such that
	\begin{equation}
	\begin{aligned}
		\norm{F_{\dif{u}A(\placeholder)^{-1} - R}}_{\Leb^p(U)} &\leq C \norm{F_{\dist(\dif{u}A(\placeholder)^{-1}, \SO(d))}}_{\Leb^p(U)}, \\
		\norm{G_{\dif{u}A(\placeholder)^{-1} - R}}_{\Leb^q(U)} &\leq C \norm{G_{\dist(\dif{u}A(\placeholder)^{-1}, \SO(d))}}_{\Leb^q(U)}.
	\end{aligned}
	\end{equation}
	
	\item (Korn's inequality) Given a decomposition $\sym(\dif{u}A(\placeholder)^{-1}) = F_{\sym(\dif{u}A(\placeholder)^{-1})} + G_{\sym(\dif{u}A(\placeholder)^{-1})}$ in $\Leb^p+\Leb^q(U, \R^{d\times d})$, there exist $S \in \R^{d\times d}_{\skewMat}$ and a decomposition $\dif{u}A(\placeholder)^{-1} - S = F_{\dif{u}A(\placeholder)^{-1} - S} + G_{\dif{u}A(\placeholder)^{-1} - S}$ in $\Leb^p+\Leb^q(U, \R^{d\times d})$, such that
	\begin{equation}
	\begin{aligned}
		\norm{F_{\dif{u}A(\placeholder)^{-1} - S}}_{\Leb^p(U)} &\leq C \norm{F_{\sym(\dif{u}A(\placeholder)^{-1})}}_{\Leb^p(U)}, \\
		\norm{G_{\dif{u}A(\placeholder)^{-1} - S}}_{\Leb^q(U)} &\leq C \norm{G_{\sym(\dif{u}A(\placeholder)^{-1})}}_{\Leb^q(U)}.
	\end{aligned}
	\end{equation}
\end{enumerate}
Moreover, let $L \geq 1$, $r := \operatorname{diam}(U) := \sup_{x,y \in U} \abs{x - y}$ and $\rho := \min\set{\frac{r}{2}, \delta}$. The constant $C$ can be chosen to be of the form $C = (\frac{r}{\rho})^d c$ for some $c = c(d,e,p,q,L)$ uniformly for all $A$ that admit a Bilipschitz potential with Bilipschitz constant not greater than $L$. 
(For the proof see \cref{Sec:Proof:ExtensionOperatorRigidity}.)
\end{theorem}

Recall that especially all periodic stress-free joints $A \in \operatorname{SFJ}_\mathrm{per}$ admit a Bilipschitz potential by \cref{Prop:SFJBilipschitz} and the Bilipschitz constants of the potentials of $A(\tfrac{\placeholder}{\eps})$ coincide for all $\eps > 0$. Moreover, since $e$ and $\frac{r}{\rho}$ are invariant under scaling of the domain, so is the constant $C$. We obtain the following versions of Korn's inequality as corollaries.

\begin{corollary}[Korn's inequality] \label{Cor:KornWithBoundaryValue}
Let $U \subset \R^d$ a Lipschitz domain, $\Gamma \subset \partial U$ with $\mathcal{H}^{d-1}(\Gamma) > 0$, $1 < p < \infty$ and $A \in \operatorname{SFJ}(U)$. Then, there exists a constant $C = C(U,\Gamma,p,A) > 0$, such that for all $u \in \SobW^{1,1}_{\Gamma,0}(U, \R^d)$,
\begin{equation} \label{Eq:KornWithBoundaryValue}
	\norm{u}_{\SobW^{1,p}(U)} \leq C \norm{\sym(\dif{u}A(\placeholder)^{-1})}_{\Leb^p(U)}.
\end{equation}
Moreover, given $L \geq 1$, the constant $C$ can be chosen uniformly for all $A$ that admit a Bilipschitz potential with Bilipschitz constant not greater than $L$. 
(For the proof see \cref{Sec:Proof:ExtensionOperatorRigidity}.)
\end{corollary}

\begin{corollary}[Korn's second inequality] \label{Cor:KornWithLpNormOfu}
Let $U \subset \R^d$ a Lipschitz domain, $1 < p < \infty$ and $A \in \operatorname{SFJ}(U)$. Then, there exists a constant $C = C(U,A,p) > 0$, such that for all $u \in \SobW^{1,1}(U, \R^d)$,
\begin{equation} \label{Eq:KornWithLpNormOfu}
	\norm{u}_{\SobW^{1,p}(U)} \leq C \left(\norm{u}_{\Leb^p(U)} + \norm{\sym(\dif{u}A(\placeholder)^{-1})}_{\Leb^p(U)}\right).
\end{equation}
Moreover, given $L \geq 1$, the constant $C$ can be chosen uniformly for all $A$ that admit a Bilipschitz potential with Bilipschitz constant not greater than $L$. 
(For the proof see \cref{Sec:Proof:ExtensionOperatorRigidity}.)
\end{corollary}

\begin{corollary}[Periodic Korn's inequality] \label{Cor:KornPeriodic}
Let $1 < p < \infty$ and $A \in \operatorname{SFJ}_\mathrm{per}$. There exists a constant $C=C(p, A) > 0$ such that for all $\varphi \in \SobW^{1,p}_{\mathrm{per},0}(Y, \R^d)$,
\begin{equation}\label{Eq:korn_periodic}
	\norm{\varphi}_{\SobW^{1,p}(Y)} 
	\leq C \norm{\sym(\dif{\varphi}A(\placeholder)^{-1})}_{\Leb^p(Y)}.
\end{equation}
(For the proof see \cref{Sec:Proof:ExtensionOperatorRigidity}.)
\end{corollary}

\begin{remark}
The mixed growth versions for the geometric rigidity estimate and Korn's inequality have been proven in \cite{CDM14} for the case of Lipschitz domains and Korn's inequality has been shown in \cite{DM04} to hold on Jones domains. However, up to our knowledge, the result for Korn's inequality on Jones domains is novel in the mixed growth version and the geometric rigidity estimate on Jones domains even in the standard case without mixed growth ($p=q$). Furthermore, as an application of the mixed growth versions, we obtain Korn's inequality and the geometric rigidity estimate in the Lorentz space $\Leb^{p,q}(U)$ for $1 < p < \infty$ and $1 \leq q \leq \infty$, see \cite[Cor.~4.1]{CDM14}.
\end{remark}

The proof of \cref{Thm:RigidityEstimateJonesDomain} relies on the construction of an extension operator $E$ that allows us to control $\dist(\dif{Eu}, \SO(d))$ by $\dist(\dif{u}, \SO(d))$ (resp.\ distance to $\R^{d\times d}_{\skewMat}$). The construction is adapted from \cite{Jon81,DM04} and presented together with the proofs for \cref{Thm:RigidityEstimateJonesDomain,Cor:KornWithBoundaryValue,Cor:KornWithLpNormOfu,Cor:KornPeriodic} in \cref{Sec:Proof:ExtensionOperatorRigidity}. This extension operator is interesting in its own right. In the proof we only require the rigidity estimate to hold on cubes. Thus, the procedure provides an alternative to the second part of the proof of \cite[Thm.~3.1]{FJM02}, where the rigidity estimate is lifted from cubes to arbitrary Lipschitz domains. Moreover, as presented in \cref{Thm:RigidityEstimateJonesDomain}, from our procedure we obtain fairly good control over the constant depending on the regularity of the domain.

\section{Example: Isotropic Laminates}
\label{Sec:Example}

In this section we study the linearized and homogenized energy density $Q^A_{\hom}$ and the homogenized perturbation $B_{\hom}$ in dependence of the microstructure for the example of an isotropic laminate in $\R^3$. The symmetries present in this case reduce the complexity tremendously, so that we are able to compute the quantities by hand.

\subsection{Formulas for three dimensional isotropic laminates}

Let $d = 3$. We suppose
\begin{equation}
	Q(y,G) = \lambda(y_1) \trace(G)^2 + 2\mu(y_1) \abs{\sym G}^2,
\end{equation}
with Lamé constants $\lambda, \mu \in \Leb^\infty_\mathrm{per}([0,1))$ satisfying  $\operatorname{ess\,inf}_{[0,1)} \mu > 0$ and $\operatorname{ess\,inf}_{[0,1)} (\lambda + \frac{2\mu}{3}) > 0$. We also suppose that the stress-free joint $A$ only depends on $y_1$. We denote by $a$ the Bilipschitz potential of $A$. It is not hard to show that then necessarily there exists a map $\bar{a} \in \SobW^{1,\infty}_\mathrm{per}([0,1), \R^3)$, such that (cf.\ \cref{Lem:periodic_derivative})
\begin{equation}
	a(y) = \bar{a}(y_1) + \bar{A}y.
\end{equation}
We want to compute $B_{\hom}$ and $Q^A_{\hom}$ for this situation using \cref{Prop:Bhom_algorithmic}. Thus, our first goal is to give explicit formulas for the correctors. Here, it is convenient to change variables first and compute different correctors than proposed in \cref{Prop:Bhom_algorithmic}. The procedure is summarized in the following remark.

\begin{remark} \label{Rem:AlgoCharacterization_alternatives}
By changing variables and applying \cref{Lem:periodic_derivative}, we can also represent $Q^A_{\hom}$ by
\begin{equation*}
	Q^A_{\hom}(G) = \int_{\bar{A}Y} \tilde{Q}\left(z, G\bar{A}^{-1} + \dif{\tilde{\varphi}_{G\bar{A}^{-1}}}(z)\right) \,\dx[z]
	= \int_Y \hat{Q}\left(y, (G + \dif{\hat{\varphi}_{G\bar{A}^{-1}}}(y))\bar{A}^{-1}\right) \,\dx[y],
\end{equation*}
where
\begin{subequations}
\begin{align}
	&\tilde{Q}(z,G) := Q\big(a^{-1}(z), G\big) \det A(a^{-1}(z))^{-1}, && z \in \R^d, &&\\
	&\hat{Q}(y,G) := Q\big(a^{-1}(\bar{A}y), G\big) \det A(a^{-1}(\bar{A}y))^{-1} \det\bar{A}, && y \in \R^d.&&
\end{align}
\end{subequations}
and the correctors are defined as
\begin{subequations}
\begin{gather}
	\tilde{\varphi}_G := \operatorname{argmin} \class{\int_{\bar{A}Y} \tilde{Q}\big(z, G+\dif{\varphi}(z)\big)\,\dx[y]}{\varphi \in \SobH^1_{\mathrm{per},0}(\bar{A}Y,\R^d)}, \\
	\hat{\varphi}_G := \operatorname{argmin} \class{\int_Y \hat{Q}\big(y, G+\dif{\varphi}(y)\bar{A}^{-1}\big)\,\dx[y]}{\varphi \in \SobH^1_{\mathrm{per},0}(Y,\R^d)}.
\end{gather}
\end{subequations}
Indeed, we get
\begin{subequations}
\begin{align}
	(\Qhom)_{ij} &= \int_{\bar{A}Y} \big(G_i + \dif{\tilde{\varphi}_{G_i}}(z)\big) : \mathbb{L}_{\tilde{Q}}(z) \big(G_j + \dif{\tilde{\varphi}_{G_j}}(z)\big) \,\dx[z] \\
	&= \int_Y \big(G_i + \dif{\hat{\varphi}_{G_i}}(y)\bar{A}^{-1}\big) : \mathbb{L}_{\hat{Q}}(y) \big(G_j + \dif{\hat{\varphi}_{G_j}}(y)\bar{A}^{-1}\big) \,\dx[y],
\end{align}
\end{subequations}
and
\begin{subequations}
\begin{align}
	b_i &= \int_{\bar{A}Y} \big(G_i + \dif{\tilde{\varphi}_{G_i}}(z)\big) : \mathbb{L}_{\tilde{Q}}(z) B(a^{-1}(z)) \,\dx[z] \\
	&= \int_Y \big(G_i + \dif{\hat{\varphi}_{G_i}}(y)\bar{A}^{-1}\big) : \mathbb{L}_{\hat{Q}}(y) B(a^{-1}(\bar{A}y)) \,\dx[y].
\end{align}
\end{subequations}
Note that the different correctors are connected via the formulas
\begin{equation}
	\tilde{\varphi}_G = \varphi_G \circ a^{-1}, \quad \hat{\varphi}_G = \varphi_G \circ a^{-1} \circ \bar{A}\placeholder.
\end{equation}
\end{remark}

One version or another may be more useful for a certain purpose. We shall see that here, it is convenient to use the version b). One advantage of b) is that it basically reduces the problem to the case where $\dif{a} \equiv \bar{A}$ which helps us later to compute the correctors. First, note that $\hat{Q}$ is again an isotropic, linearized elastic energy density with Lamé constants
\begin{gather*}
	\hat{\lambda}(y) := \lambda([a^{-1}(\bar{A}y)]_1) \det A([a^{-1}(\bar{A}y)]_1)^{-1}\det\bar{A}, \\ \hat{\mu}(y) := \mu([a^{-1}(\bar{A}y)]_1) \det A([a^{-1}(\bar{A}y)]_1)^{-1} \det \bar{A}.
\end{gather*}
Moreover, $\hat{Q}$ is still a laminate, since $\hat{\lambda}$ and $\hat{\mu}$ only depend on $y_1$, in view of
\begin{equation}
	a^{-1}(\bar{A}y) = a^{-1}(y_1 \bar{A}e_1) + (0,y_2, y_3)^T.
\end{equation}
Indeed, let $z := a^{-1}(y_1 \bar{A}e_1)$. Then,
\begin{equation*}
	a(z + (0,y_2, y_3)^T) = \bar{a}(z_1) + \bar{A}z + \bar{A}(0,y_2, y_3)^T = a(z) + \bar{A}(0,y_2, y_3)^T = \bar{A}y.
\end{equation*}
We denote the mean over $Y$ (resp. $[0,1)$) of some map $f \in \Leb^1(Y, \R^m)$, $m \in \N$ (resp. $f \in \Leb^1([0,1),\R^m$) with $\mean{f}$ and the harmonic mean with $\meanHarm{f}$. Recall the following standard moduli of elastic, isotropic materials:
\begin{align*}
	&K := \lambda + \tfrac{2}{3}\mu && \text{(Bulk modulus)}, &&\\
	&M := K + \tfrac{4}{3}\mu = \lambda + 2\mu && \text{(P-wave modulus)}, &&
\end{align*}
and analogously $\hat{K}, \hat{M}$. As a basis for $\R^{d\times d}_{\sym}$, we consider
\begin{align*}
	G_1 &= \begin{pmatrix} \tfrac{1}{3} & 0 & 0 \\ 0 & \tfrac{1}{3} & 0 \\ 0 & 0 & \tfrac{1}{3} \end{pmatrix}, &
	G_2 &= \begin{pmatrix} \tfrac{2}{3} & 0 & 0 \\ 0 & -\tfrac{1}{3} & 0 \\ 0 & 0 & -\tfrac{1}{3} \end{pmatrix}, &
	G_3 &= \begin{pmatrix} 0 & 0 & 0 \\ 0 & -\tfrac{1}{2} & 0 \\ 0 & 0 & \tfrac{1}{2} \end{pmatrix}, &&&& \\
	G_4 &= \begin{pmatrix} 0 & \tfrac{1}{2} & 0 \\ \tfrac{1}{2} & 0 & 0 \\ 0 & 0 & 0 \end{pmatrix}, &
	G_5 &= \begin{pmatrix} 0 & 0 & \tfrac{1}{2} \\ 0 & 0 & 0 \\ \tfrac{1}{2} & 0 & 0 \end{pmatrix}, &
	G_6 &= \begin{pmatrix} 0 & 0 & 0 \\ 0 & 0 & \tfrac{1}{2} \\ 0 & \tfrac{1}{2} & 0 \end{pmatrix}.
\end{align*}
Then, for every matrix $F \in \R^{d\times d}$, we get the decomposition $\sym F = \sum_{i=1}^6 a_i G_i$, where
\begin{align*}
	a_1 &= F_{11} + F_{22} + F_{33},& a_2 &= F_{11} - \tfrac{1}{2} F_{22} - \tfrac{1}{2} F_{33},& a_3 &= F_{33} - F_{22}, &&&&\\
	a_4 &= F_{12} + F_{21},& a_5 &= F_{13} + F_{31},& a_6 &= F_{23} + F_{32}.
\end{align*}
Moreover, we get
\begin{align*}
	\hat{Q}(y, G_1) &= \hat{\lambda}(y_1) + \tfrac{2}{3} \hat{\mu}(y_1) = \hat{K}(y_1), & \hat{Q}(y, G_2) &= \tfrac{4}{3} \hat{\mu}(y_1), &&&&\\
	\hat{Q}(y, G_3) &= \hat{Q}(y, G_4) = \hat{Q}(y,G_5) = \hat{Q}(y,G_6) = \hat{\mu}(y_1).
\end{align*}

\begin{proposition} \label{Prop:CorrectorsIsotropicLaminate}
In the situation as above, the correctors $\hat{\varphi}_{G_i}$ (cf.~\cref{Rem:AlgoCharacterization_alternatives}) depend only on $y_1$ and satisfy
\begin{align*}
	\dif{\hat{\varphi}_{G_1}} &= \frac{\meanHarm{M} \mean{\tfrac{K}{M}} - K}{\abs{\alpha}^2 M} \alpha e_1^T, &
	\dif{\hat{\varphi}_{G_2}} &= \frac{4\beta_1}{3\abs{\alpha}^2} \begin{pmatrix} \alpha_1 \\ -\tfrac{1}{2}\alpha_2 \\ -\tfrac{1}{2}\alpha_3\end{pmatrix}e_1^T + \frac{4\alpha_1^2 - 2\alpha_2^2 - 2\alpha_3^2}{3\abs{\alpha}^4} \beta_2 \alpha e_1^T,  \\
	\dif{\hat{\varphi}_{G_3}} &= \frac{\beta_1}{\abs{\alpha}^2} \begin{pmatrix} 0 \\ -\alpha_2 \\ \alpha_3\end{pmatrix}e_1^T + \frac{\alpha_3^2 - \alpha_2^2}{\abs{\alpha}^4} \beta_2 \alpha e_1^T, &
	\dif{\hat{\varphi}_{G_4}} &= \frac{\beta_1}{\abs{\alpha}^2} \begin{pmatrix} \alpha_2 \\ \alpha_1 \\ 0 \end{pmatrix}e_1^T + \frac{2\alpha_1\alpha_2}{\abs{\alpha}^4} \beta_2 \alpha e_1^T, \\
	\dif{\hat{\varphi}_{G_5}} &= \frac{\beta_1}{\abs{\alpha}^2} \begin{pmatrix} \alpha_3 \\ 0 \\ \alpha_1 \end{pmatrix}e_1^T + \frac{2\alpha_1\alpha_3}{\abs{\alpha}^4} \beta_2 \alpha e_1^T, &
	\dif{\hat{\varphi}_{G_6}} &= \frac{\beta_1}{\abs{\alpha}^2} \begin{pmatrix} 0 \\ \alpha_3 \\ \alpha_2 \end{pmatrix}e_1^T + \frac{2\alpha_2\alpha_3}{\abs{\alpha}^4} \beta_2 \alpha e_1^T,
\end{align*}
where $\alpha := \bar{A}^{-T}e_1$ and
\begin{equation*}
	\beta_1 := \frac{\meanHarm{\mu}}{\mu} - 1, \quad\beta_2 := \frac{\meanHarm{M}\mean{\tfrac{\mu}{M}}}{M} - \frac{\meanHarm{\mu}}{\mu} + 1 - \frac{\mu}{M}.
\end{equation*}
(For the proof see \cref{Sec:correctors-for-isotropic-laminates}.)
\end{proposition}

With these formulas it is straight-forward to calculate $\Qhom$ and $B_{\hom}$ using \cref{Rem:AlgoCharacterization_alternatives}. We omit the calculations and state the result for the special case $\alpha = \abs{\alpha}e_1$.

\begin{proposition} \label{Prop:Laminate_homogenized_values}
Consider the situation as above, where $a$ is such that $\alpha = \abs{\alpha}e_1$, $\alpha := \bar{A}^{-T}e_1$. Then the matrix $\Qhom$ from \cref{Prop:Bhom_algorithmic} is the symmetric block-diagonal matrix given by
\begin{equation*}
	\Qhom = \begin{pmatrix} A_1 & 0 \\ 0 & A_2 \end{pmatrix},
\end{equation*}
where
\begin{align*}
	A_1 &:= \begin{pmatrix}
		\mean{\tfrac{\hat{K}}{\hat{M}}}\meanHarm{\hat{M}} - \tfrac{4}{3}\mean{\gamma \hat{\mu}} & \tfrac{4}{3}\mean{\gamma \hat{\mu}} \\
		\tfrac{4}{3}\mean{\gamma \hat{\mu}} & \tfrac{4}{3}\mean{\tfrac{\hat{\mu}}{\hat{M}}}\meanHarm{\hat{M}} - \tfrac{4}{3}\mean{\gamma \hat{\mu}}
	\end{pmatrix}, \\
	A_2 &:= \operatorname{diag}\Big(\mean{\hat{\mu}}, \meanHarm{\hat{\mu}}, \meanHarm{\hat{\mu}}, \mean{\hat{\mu}}\Big).
\end{align*}
Moreover, $B_{\hom} := \sum_{i=1}^{6} \mathbf{B}_i G_i$ with coefficients
\begin{alignat*}{3}
	&\mathrlap {\mathbf{B}_1 = \mean{\hat{B}_{11} + \tfrac{\hat{\lambda}}{\hat{M}}\hat{B}_{22} + \tfrac{\hat{\lambda}}{\hat{M}}\hat{B}_{33}} + 2 \mean{(\hat{B}_{22} + \hat{B}_{33})\tfrac{\hat{K}\hat{\mu}}{\hat{M}}} \frac{\mean{\tfrac{\hat{\mu}}{\hat{M}}}}{\mean{\tfrac{\hat{K}\hat{\mu}}{\hat{M}}}},} \\
	&\mathrlap {\mathbf{B}_2 = \mean{\hat{B}_{11} + \tfrac{\hat{\lambda}}{\hat{M}}\hat{B}_{22} + \tfrac{\hat{\lambda}}{\hat{M}}\hat{B}_{33}} - \tfrac{3}{2} \mean{(\hat{B}_{22} + \hat{B}_{33})\tfrac{\hat{K}\hat{\mu}}{\hat{M}}} \frac{\mean{\tfrac{\hat{K}}{\hat{M}}}}{\mean{\tfrac{\hat{K}\hat{\mu}}{\hat{M}}}},} \\
	&\mathbf{B}_3 = \frac{\mean{\hat{\mu} (\hat{B}_{33} - \hat{B}_{22})}}{\mean{\hat{\mu}}}, &&\qquad
	\mathbf{B}_4 = \mean{\hat{B}_{12} + \hat{B}_{21}}, \\
	&\mathbf{B}_5 = \mean{\hat{B}_{13} + \hat{B}_{31}}, &&\qquad
	\mathbf{B}_6 = \frac{\mean{\hat{\mu} (\hat{B}_{23} + \hat{B}_{32})}}{\mean{\hat{\mu}}}.
\end{alignat*}
Here $\hat{B}(y) = B(a^{-1}(\bar{A}y))$ and $\gamma := \frac{\mean{\tfrac{\hat{K}}{\hat{M}}}\meanHarm{\hat{M}} - \hat{K}}{\hat{M}}$.
\end{proposition}

\begin{remark}
By means of the transformation rule and \cref{Lem:periodic_derivative}, we have $\mean{\hat{\mu}} = \mean{\mu}$ and analogously for $\lambda$, $K$, $M$ and $\frac{K\mu}{M}$. But the harmonic mean, as well as other entities, in generality depend on $A$.
\end{remark}

\subsection{Isotropic bilayers with bilayered prestrain}

In this last section we want to visualize the dependence of the prestrain on the microstructure for the special case of an isotropic bilayer with bilayered prestrain. This means, we consider a laminate consisting of two homogeneous, isotropic materials that on each phase feature a homogeneous prestrain, i.e.,
\begin{alignat*}{3}
	&\lambda(y) = \begin{cases} \lambda_1 & y_1 \in [0,\theta), \\ \lambda_2 & y_1 \in [\theta,1),\end{cases} &&\qquad
	\mu(y) = \begin{cases} \mu_1 & y_1 \in [0,\theta), \\ \mu_2 & y_1 \in [\theta,1),\end{cases} && \\
	& A(y) = \begin{cases} A_1 & y_1 \in [0,\theta), \\ A_2 & y_1 \in [\theta,1),\end{cases} &&\qquad
	B(y) = \begin{cases} B_1 & y_1 \in [0,\theta), \\ B_2 & y_1 \in [\theta,1),\end{cases}
\end{alignat*}
where $\theta \in [0,1]$ is the volume fraction of the first material. With this definition $A$ is a stress-free joint, if and only if $A_1 = A_2 + c e_1^T$ for some $c \in \R^3$, such that $e_1^T A_2^{-1} c > 0$. Indeed, in view of \eqref{Eq:Rank_one_compatiblity}, $A_1 = A_2 + c e_1^T$ is required to ensure that $A$ is a gradient and $e_1^T A_2^{-1} c > 0$ is equivalent to $\det A_1, \det A_2 > 0$. We can then define the potential $a$ by
\begin{equation*}
	a(y) := \begin{cases} A_1 y & y_1 \in [0,\theta), \\ A_2 y + \theta c & y_1 \in [\theta,1).\end{cases}
\end{equation*}
By applying the inversion formula in \cite{Mil81}, we obtain $A_1^{-1} = A_2^{-1} - \frac{1}{1 + e_1^T A_2^{-1} c} A_2^{-1}ce_1^TA_2^{-1}$. Using this, we can explicitly calculate $a^{-1}(\bar{A}y)$, $y \in Y$. We get
\begin{equation*}
	a^{-1}(\bar{A}y) = \begin{cases}
		A_1^{-1}\bar{A}y & y_1 \in [0,\hat{\theta}), \\
		A_2^{-1}(\bar{A}y - \theta c) & y_1 \in [\hat{\theta}, 1),
	\end{cases}
\end{equation*}
with the distorted volume fraction
\begin{equation*}
	\hat{\theta} := \frac{\theta}{1-(1-\theta)e_1^T A_1^{-1}c} = \frac{\theta (1 + e_1^T A_2^{-1}c)}{1 + \theta e_1^T A_2^{-1}c}.
\end{equation*}
Note that $\hat{\theta}$ is exactly defined, such that
\begin{align*}
	&\frac{\theta}{\hat{\theta}} = 1-(1-\theta)e_1^T A_1^{-1}c = \det A_1^{-1}\det\bar{A}, &
	&\frac{1-\theta}{1-\hat{\theta}} = 1 +\theta e_1^T A_2^{-1}c = \det A_2^{-1}\det\bar{A},
\intertext{and hence}
	&e_1^T A_1^{-1}\bar{A}y \in [0,\theta) \Leftrightarrow y_1 \in [0,\hat{\theta}), &
	&e_1^T A_2^{-1}(\bar{A}y - \theta c) \in [\theta, 1) \Leftrightarrow y_1 \in [\hat{\theta}, 1).
\end{align*}
Using this, the mean values in \cref{Prop:Laminate_homogenized_values} can be explicitly calculated. For the readers convenience we give an example for each case:
\begin{gather*}
	\mean{\hat{\mu}} = \mean{\mu} = \theta \mu_1 + (1-\theta)\mu_2, \qquad \mean{\tfrac{\hat{K}}{\hat{M}}} = \hat{\theta} \tfrac{K_1}{M_1} + (1 - \hat{\theta})\tfrac{K_2}{M_2}, \\
	\meanHarm{\hat{\mu}} = \left( \frac{\hat{\theta}^2}{\theta\mu_1} + \frac{(1 - \hat{\theta})^2}{(1-\theta)\mu_2} \right)^{-1} = \left( \frac{\theta}{\mu_1(1 - (1-\theta)e_1^TA_1^{-1}c)^2} + \frac{1 - \theta}{\mu_2 (1 + \theta e_1^T A_2^{-1}c)^2} \right)^{-1}.
\end{gather*}

\paragraph{Dependence of $Q^A_{\hom}$ and $B_{\hom}$ on the microstructure.}
We want to study the dependence of $\Qhom$ and $\mathbf{B}$ on the microstructure. For this we look at the special case $A \equiv I$. Our results from \cref{Sec:LinAndHomResults} state that the deformation away from $\Gamma$ is given up to order $\oclass(h)$ by
\begin{equation*}
	\phi(x) := x + hB_{\hom}x.
\end{equation*}
We can easily calculate the coefficients using \cref{Prop:Laminate_homogenized_values}. Note that $\hat{\theta} = \theta$, $\hat{\lambda} = \lambda$, $\hat{\mu} = \mu$. \cref{Fig:ExampleDependenceOnMicrostructure} displays these coefficients in dependence of the volume fraction $\theta$ and on the Lamé constant $\mu_2$ for fixed volume fraction $\theta = \frac{1}{2}$. These graphs show that $Q^A_{\hom}$ and $B_{\hom}$ depend non-linearly on $\theta$ for heterogeneous materials. Especially does the homogenized prestrain differ from just taking the mean value in favor of the stronger material.

\begin{figure}[htp]
\centering
\includegraphics[width=0.72\linewidth]{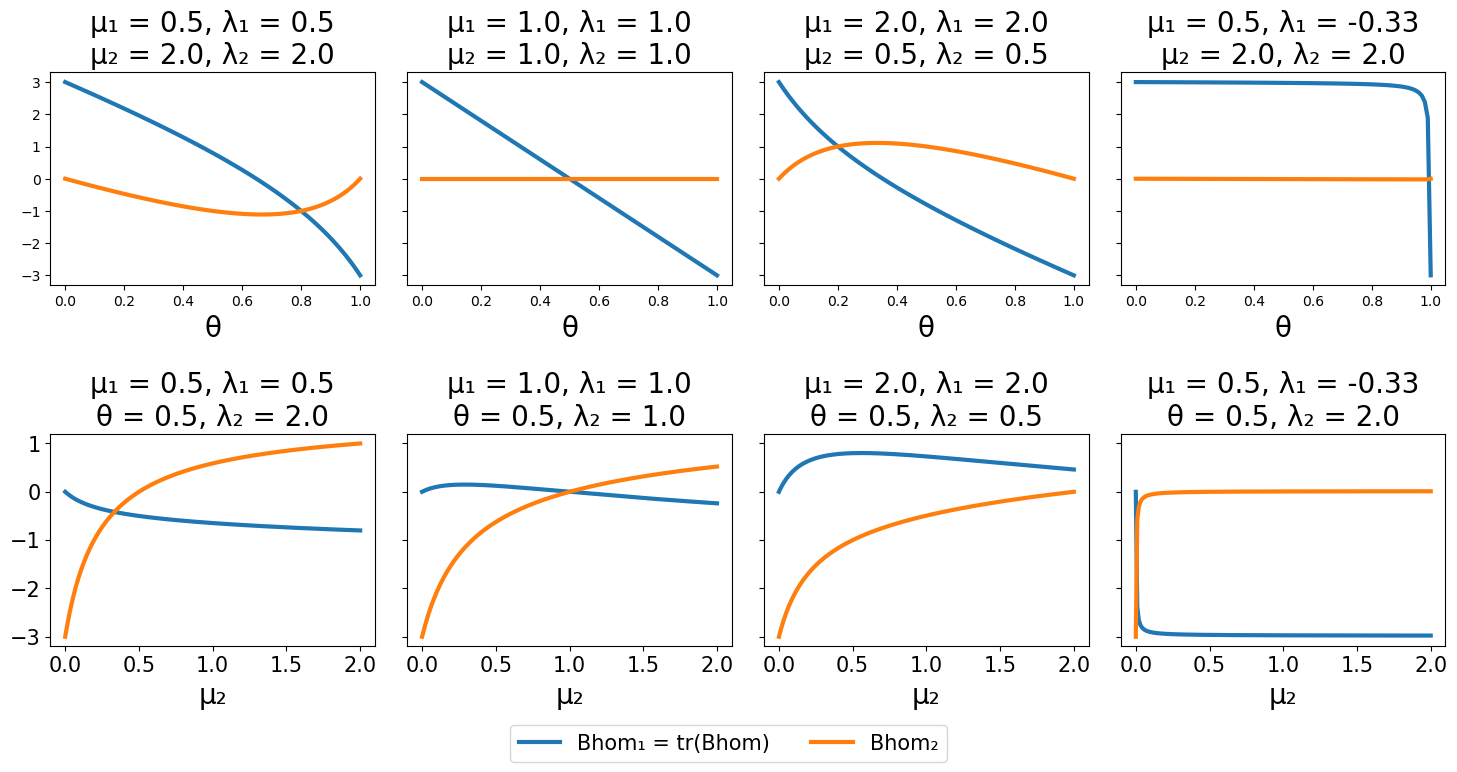}
\includegraphics[width=0.72\linewidth]{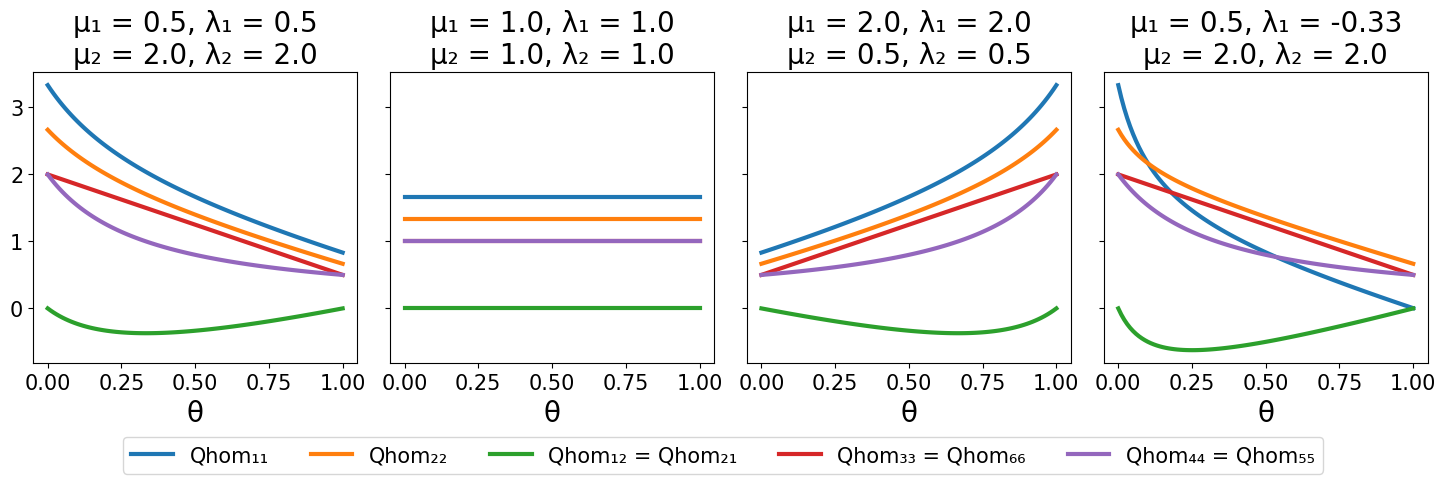}
\caption{The graphs depict the coefficients of $\mathbf{B} \in \R^6$ and $\Qhom \in \R^{6\times 6}$ for an isotropic laminate with $B_1 = -I$, $B_2 = I$ and $A \equiv I$. We display the dependences on the volume fraction $\theta$ and the Lamé constant $\mu_2$. In this situation only the coefficients $\mathbf{B}_1$ and $\mathbf{B}_2$ are non-zero. The blue curve for $\mathbf{B}$ is a measure for the volume expansion, since the trace is the first order term in the expansion of the determinant. The second row of graphs show a non-linear influence of the material law on $B_{\hom}$.}
\label{Fig:ExampleDependenceOnMicrostructure}
\end{figure}

\paragraph{Dependence of $Q^A_{\hom}$ and $B_{\hom}$ on the stress-free joint.}
We are interested in the dependencies which don't come from the mean matrix $\bar{A}$. For this, we consider the following one-parameter-family:
\begin{equation*}
	A_{\beta}(y) := \begin{cases} \operatorname{diag}(\beta, 1, 1) & y_1 \in [0, \tfrac12) \\
	\operatorname{diag}(2 - \beta, 1, 1) & y_1 \in [\tfrac12,1) \end{cases}, \quad 0 < \beta < 2.
\end{equation*}
Thus, $\theta = \frac{1}{2}$, $\bar{A}_\beta \equiv I$, $\hat{\theta} = \frac{\beta}{2}$. We can easily calculate the perturbation $B_{\hom}$ and the homogenized energy $Q^A_{\hom}$ in dependence of $\beta$. The results are displayed in \cref{Fig:ExampleDependenceOnSFJ} shows that the homogenization of the perturbation and the energy cannot be decoupled from the homogenization of the stress-free joint.

\begin{figure}[htp]
\centering
\includegraphics[width=0.72\linewidth]{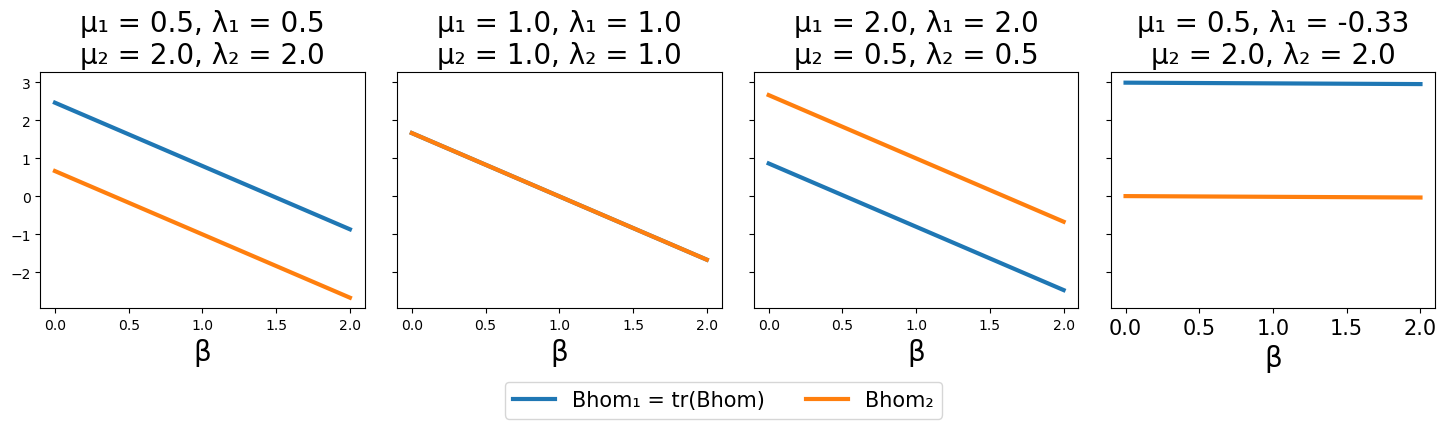}
\includegraphics[width=0.72\linewidth]{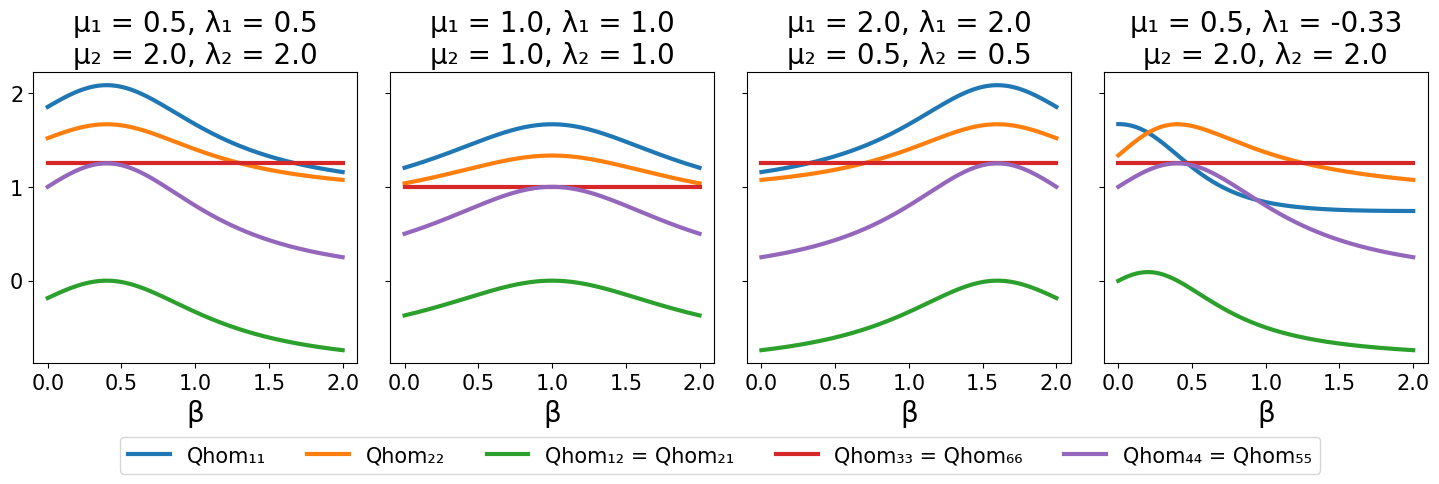}
\caption{The graphs display the dependence of $\mathbf{B} \in \R^6$ and $\Qhom \in \R^{6\times 6}$ on the parameter $\beta$ for the family $A_\beta$ and $B_1 = -I$, $B_2 = I$. We see a non-constant dependence, which especially shows that also the homogenization of the energy and the perturbation cannot be decoupled from the homogenization of the stress-free joint. The dependence of $\mathbf{B}$ on $\beta$ is linear.}
\label{Fig:ExampleDependenceOnSFJ}
\end{figure}

\FloatBarrier

\section{Proofs}

\subsection{Properties of stress-free joints and proof of Proposition~\ref{Prop:SFJBilipschitz}} \label{Sec:ProofSFJ}

In this section, we provide some properties of periodic maps, maps with periodic derivative and stress-free joints, as defined in \cref{Def:StressFreeJoint}, that we require later. Especially, we establish \cref{Prop:SFJBilipschitz}. We start by collecting some basic properties of maps with periodic derivative. Since they are standard statements, we only sketch the proof.

\begin{lemma}\label{Lem:periodic_derivative}
Let $1 \leq p \leq \infty$ and $a \in \SobW^{1,p}_\mathrm{loc}(\R^d, \R^d)$ be continuous with $a(0) = 0$ and $Y$-periodic derivative. Set $a_\eps(x) := \eps a(\frac{x}{\eps})$ and $\bar{A} := \fint_{Y} \dif{a}(y) \,\dx[y]$. Then
\begin{enumerate}[(a)]
	\item $a - \bar{A}\placeholder \in \SobW^{1,p}_\mathrm{per}(Y, \R^d)$, i.e., $a(y+k) = a(y) + \bar{A}k$ for all $k \in \Z^d$ and $y \in \overline{Y}$.
	
	\item $a_\eps \wto \bar{A}\placeholder$ weakly in $\SobW^{1,p}_\mathrm{loc}(\R^d,\R^d)$ (or weakly-$*$ for $p=\infty$). Especially, if $p > d$, then $\sup_{x \in \R^d} \abs{a_\eps(x) - \bar{A}x} \to 0$ as $\eps \to 0$.
	
	\item If $p > d$ and $\det \dif{a} > 0$ a.e.~in $Y$, then $\det \bar{A} = \fint_Y \det\dif{a}(y) \,\dx[y] > 0$.
\end{enumerate}
Suppose additionally, $a: \R^d \to \R^d$ is a homeomorphism and onto, $p > d$ and $\det \dif{a} > 0$ a.e.\ in $Y$. Then
\begin{enumerate}[(a)] \setcounter{enumi}{3}
	\item A measurable map $\varphi: \R^d \to \R$ is $\bar{A}$-periodic, if and only if $\varphi \circ a$ is $Y$-periodic. Moreover, in this case
	\begin{equation}\label{Eq:Integrals_coincide_periodic}
		\int_{a(Y)} \varphi(z) \,\dx[z] = \int_{\bar{A}Y} \varphi(z) \,\dx[z].
	\end{equation}
	
	\item Let $\varphi \in \SobW^{1,1}_\text{per}(\bar{A}Y)$. Then,
	\begin{equation}
		\fint_{a(Y)} \dif{\varphi}(z) \,\dx[z] = \fint_{\bar{A}Y} \dif{\varphi}(z) \,\dx[z] = 0.
	\end{equation}
	
	\item $a^{-1} - \bar{A}^{-1}\placeholder$ is $\bar{A}$-periodic, i.e., $a^{-1}(z + \bar{A}k) = a^{-1}(z) + k$ for all $z \in \R^d$, $k \in \Z^d$.
	
	\item If $a^{-1} \in \SobW^{1,1}_\mathrm{loc}(\R^d,\R^d)$, then $\bar{A}^{-1} = \fint_{a(Y)} \dif{a^{-1}}(z) \,\dx[z] = \fint_{\bar{A}Y} \dif{a^{-1}}(z) \,\dx[z]$.
	
	\item If $a^{-1} \in \SobW^{1,q}_\mathrm{loc}(\R^d,\R^d)$ for some $q > d$, then $\det \bar{A}^{-1} = \fint_{a(Y)} \det\dif{a^{-1}}(z) \,\dx[z]$\\ = $\fint_{\bar{A}Y} \det\dif{a^{-1}}(z) \,\dx[z]$.
\end{enumerate}
\end{lemma}

\begin{proof} We first sketch the proof of (a). Since $\dif{a}$ is $Y$-periodic, for all $k \in \mathbb{Z}^d$, there exists a constant $c_k > 0$, such that
\begin{equation*}
	a(y + k) = a(y) + c_k, \quad y \in \overline{Y}.
\end{equation*}
By setting $y = 0$, we obtain $c_k = a(k)$. Thus, $a: \Z^d \to \R^d$ is linear and can be represented as $a(k) = \tilde{A}k$ for some matrix $\tilde{A} \in \R^{d\times d}$. Note that we can represent $\tilde{A}$ explicitly as $\tilde{A}_{ij} = a_j(e_i)$. We claim $\tilde{A} = \bar{A}$. Indeed, since $a - \tilde{A}\placeholder$ is $Y$-periodic, the mean over $Y$ of its derivative is zero and thus,
\begin{equation*}
	0 = \fint_{Y} \dif{a}(y) - \tilde{A} \,\dx[y] = \bar{A} - \tilde{A}.
\end{equation*}
(b) and (c) are a consequence of (a), the weak convergence of rescaled periodic maps to their mean and the weak continuity of the determinant. Let us sketch (d). Since $a$ is a homeomorphism, we obtain
\begin{alignat*}{2}
	&\varphi \text{ is $\bar{A}$-periodic} \\
	&\qquad\Leftrightarrow\varphi(z + \bar{A}k) = \varphi(z) &&\quad\text{for a.e.\ } z \in \R^d \text{ and all } k \in \Z^d \\
	&\qquad\Leftrightarrow \varphi(a(y + k)) \overset{(a)}{=} \varphi(a(y) + \bar{A}k) = \varphi(a(y)) &&\quad\text{for a.e.\ } y \in \R^d \text{ and all } k \in \Z^d \\
	&\qquad\Leftrightarrow \varphi \circ a \text{ is $Y$-periodic}.
\end{alignat*}
Moreover, since $(a(Y) + \bar{A}k)_{k \in \Z^d} = (a(Y + k))_{k \in \Z^d}$ generates a tesselation of $\R^d$ and $\varphi(z + \bar{A}k) = \varphi(z)$ for a.e.\ $z \in a(Y)$ and all $k \in \Z^d$, it is not hard to show that
\begin{equation*}
	\fint_{\bar{A}Y} \varphi = \lim_{n \to \infty} \fint_{B(0,n)} \varphi = \fint_{a(Y)} \varphi.
\end{equation*}
This implies the claim, since by (c) and a change of variables (cf.\ \cite[Thm. 3.8]{EG15} and \cite[Thm. B.3.10]{KR19}), we observe
\begin{equation*}
	\abs{a(Y)} = \int_{a(Y)} 1 = \int_Y \det \dif{a}(y) \,\dx[y] = \det \bar{A} = \abs{\bar{A}Y}.
\end{equation*}
We omit the proof of (e) -- (h), since they are easy consequences of (a) -- (d).
\end{proof}

We proceed by proving \cref{Prop:SFJBilipschitz}. In fact, we prove a slightly stronger statement, which emphasizes the structures we use for our reasoning.

\begin{proposition}[Injectivity] \label{Prop:invers_is_homeomorphism}
Let $p > d$ and $a \in \SobW^{1,p}_{\mathrm{loc}}(\R^d, \R^d)$ be a continuous function with $Y$-periodic derivative and $\det \dif{a}(y) > 0$ for a.e.\ $y \in Y$. Then $a$ is injective a.e., in the sense that for a.e.\ $z \in \R^d$ the preimage $a^{-1}\set{z}$ consists of at most one point. Moreover, assume $a$ has bounded distortion, that is, there exists $K > 0$, such that
\begin{equation} \label{Eq:bounded_distortion}
	\abs{\dif{a}(y)}^d \leq K \det \dif{a}(y), \quad \text{for a.e.\ } y \in \R^d.
\end{equation}
Then, $a$ is in fact a homeomorphism, onto and $a^{-1} \in \SobW^{1,1}_\mathrm{loc}(\R^d,\R^d)$ with $\dif{a^{-1}}(z) = \dif{a}(a^{-1}(z))^{-1}$ for a.e.\ $z \in \R^d$. If $a$ additionally satisfies
\begin{equation} \label{Eq:regularity_inverse}
	\abs{\dif{a}(\placeholder)^{-1}}^p \det\dif{a} \in \Leb^1(Y) \quad (\text{resp.\ } \abs{\dif{a}(\placeholder)^{-1}} \in \Leb^\infty(\R^d) \text{ for } p = \infty),
\end{equation}
then, $a^{-1} \in \SobW^{1,p}_\mathrm{loc}(\R^d,\R^d)$.
\end{proposition}

\begin{proof}
\stepemph{Step 1 -- Idea of the proof}: Let $a_\eps(x) := \eps a(\tfrac{x}{\eps})$. Without loss of generality, we assume $a(0) = 0$. Since $p > d$ and $a$ is continuous, we have for all open, bounded sets $U \subset \R^d$ the area formula (cf.\ \cite[Thm. 3.8]{EG15} and \cite[Thm. B.3.10]{KR19}),
\begin{equation}
	\int_U \det \dif{a}(x) \,\dx = \int_{a(U)} \mathcal{H}^0(U \cap a^{-1}\set{z}) \,\dx[z].
\end{equation}
Note that $\mathcal{H}^0(U \cap a^{-1}\set{z})$ is a measure of the non-injectivity of $a|_U$. Define the level set
\begin{equation*}
	O[a,U] := \class{x \in U}{\mathcal{H}^0(U \cap a^{-1}\set{a(x)}) \geq 2} = \class{x \in U}{\exists y \in U, x \neq y:\; a(x) = a(y)}.
\end{equation*}
Then $a(U\setminus O[a,U])$ and $a(O[a,U])$ are disjoint and
\begin{equation*}
	\int_U \det \dif{a}(x) \,\dx \geq \int_{a(U\setminus O[a,U])} 1 \,\dx[z] + \int_{a(O[a,U])} 2 \,\dx[z] = \abs{a(U)} + \abs{a(O[a,U])}.
\end{equation*}
Applying this to $a_\eps$, $a_\eps \to \bar{A}\placeholder$ uniformly and $\det \dif{a_\eps} \wto \det \bar{A}$ by \cref{Lem:periodic_derivative} imply that necessarily
\begin{equation} \label{Eq:ProofSFJ:contradiction}
	\abs{a_\eps(O[a_\eps,Y])} \leq \int_Y \det \dif{a_\eps}(x) \,\dx - \abs{a_\eps(Y)} \xrightarrow{\eps \to 0} \det \bar{A} - \abs{\bar{A}Y} = 0.
\end{equation}
Now, the idea of this proof is to show, that if $a$ is not injective (a.e.), then the periodicity of $\dif{a}$ implies that the points, where $a$ is not injective, are periodically distributed over $\R^d$ and thus the mass of non-injectivity points of the rescaled functions $a_\eps$ does not vanish in $Y$, i.e., $\abs{a_\eps(O[a_\eps, Y])} > \delta$ for some $\delta > 0$ independent of $\eps$ which is a contradiction to \eqref{Eq:ProofSFJ:contradiction}.
\medskip

\stepemph{Step 2 -- Injectivity a.e.}: Suppose $a$ is not injective a.e. Then,
\begin{equation*}
	O := O[a,\R^d] = \class{x \in \R^d}{\exists z \in \R^d, x \neq z:\; a(x) = a(z)}
\end{equation*}
satisfies $\abs{O} > 0$. This set is periodic, i.e.\ $O = k + O$ for any $k \in \Z^d$. Indeed, if $x \neq z$ satisfy $a(x) = a(z)$, then $a(k+x) = a(x) + \bar{A}k = a(z) + \bar{A}k = a(k+z)$. Thus, also $\abs{O \cap Y} > 0$, since otherwise
\begin{equation*}
	\abs{O} = \sum_{k \in \Z^d} \abs{O \cap (k + Y)} = \sum_{k \in \Z^d} \abs{(k + O) \cap (k + Y)} = \sum_{k \in \Z^d}\abs{k + (O \cap Y)} = \sum_{k \in \Z^d}\abs{(O \cap Y)} = 0.
\end{equation*}
Since $a_\eps$ is obtained from $a$ by rescaling, the set $\eps O = \bigcup_{k \in \Z^d} \eps (k + O \cap Y)$ consists of all points, where $a_\eps$ is not injective. Our goal is to find a suitable subset $O^* \subset O \cap Y$ with positive measure and a sufficiently large collection $K^*_\eps \subset \Z^d$, such that
\begin{equation} \label{Eq:ProofSFJ:K}
	\bigcup_{k \in K^*_\eps} \eps(k + O^*) \overset{!}{\subset} O[a_\eps, Y]
\end{equation}
yields a contradiction to \eqref{Eq:ProofSFJ:contradiction}.
Let $O_n := \class{y \in Y}{\exists z \in [-n,n)^d, z \neq y:\; a(z)=a(y)}$, $n \in \N$. Since $O_n \uparrow (O\cap Y)$, there exists $n_0 \in \N$, such that $\abs{O_{n_0}} > 0$. We set $O^* := O_{n_0}$. Lusin's condition $\mathcal{N}^{-1}$ (cf.\ \cite[Thm.~B3.13]{KR19}) implies that also $\abs{a(O^*)} > 0$. Note that this is a critical property for this proof; it is satisfied, since $\det \dif{a} > 0$ a.e.\ in $\R^d$ and $a \in \SobW^{1,p}(Y, \R^d)$ with $p > d$. Since $a$ is continuous and $O^*$ precompact, $a(O^*)$ is bounded. Thus, there exists $l_0 \in \N$, such that the sets $a(O^*) + l_0\bar{A}k$, $k \in \Z^d$ are pair-wise disjoint. We set
\begin{equation*}
	K^*_\eps := \class{k \in l_0\Z^d}{ \eps(k + [-n_0, n_0)^d) \subset Y},
\end{equation*}
which is non-empty for $\eps \ll 1$. We observe that then the sets $a_\eps(\eps(k + O^*)) = \eps (a(O^*) + \bar{A}k)$, $k \in K^*_\eps$ are pair-wise disjoint as well. We claim that with these definitions we obtain the desired contradiction to \eqref{Eq:ProofSFJ:contradiction}. We show that \eqref{Eq:ProofSFJ:K} holds. Let $y \in \eps(k + O^*)$ for some $k \in K^*_\eps$. Then, by definition of $K^*_\eps$, we have $y \in Y$. Moreover, from the definition of $O^*$ and the rescaling $a_\eps = \eps a(\tfrac{\placeholder}{\eps})$, we obtain some $z \in \eps(k + [-n_0, n_0)^d)$, $y \neq z$ with $a_\eps(y) = a_\eps(z)$. Since $z \in Y$ by definition of $K^*_\eps$, we find $y \in O[a_\eps, Y]$. We now show that $K_\eps^*$ is large enough and infer the contradiction. We can count the elements in $K^*_\eps$ and find exactly $\lfloor\frac{\eps^{-1} - 2 n_0}{l_0}\rfloor^d$ many. Hence, for $\eps \leq 2^{-1}(2 n_0 + l_0)^{-1}$, our construction yields
\begin{equation*}
	\abs{a_\eps(O[a_\eps,Y])} \geq \sum_{k \in K^*_\eps} \abs{\eps (a(O^*) + \bar{A}k)} = \eps^d \left\lfloor\frac{\eps^{-1} - 2 n_0}{l_0}\right\rfloor^d \abs{a(O^*)} \geq (2l_0)^{-d} \abs{a(O^*)},
\end{equation*}
a contradiction to \eqref{Eq:ProofSFJ:contradiction}. Hence, $a$ must be injective a.e.
\medskip

\stepemph{Step 3 -- Sobolev homeomorphism}: For the rest of the proof, we assume that $a$ has bounded distortion. Then $a$ is a strongly open map, see \cite[Thm.~I.4.1]{Ric93} and \cite[Thm.~3-18]{HK14}. We claim that any strongly open map that is injective a.e.\ is injective everywhere. Indeed, suppose there exist $x,y \in \R^d$, $x \neq y$, such that $a(x) = a(y)$. Let $\delta > 0$, such that $B(x,\delta)\cap B(y,\delta) = \emptyset$. Then
\begin{equation*}
	O := a(B(x,\delta)) \cap a(B(y,\delta)) \ni a(x)
\end{equation*}
is open, since $a$ is an open map, and not empty. Hence $\abs{O} > 0$. But the preimage of each point in $O$ contains a point in $B(x,\delta)$ and one in $B(y,\delta)$. This is a contradiction to injectivity a.e. Moreover, since $a$ maps open sets to open sets, preimages of open sets of $a^{-1}$ are open. Hence, the inverse $a^{-1}$ is continuous. Since $a$ has bounded distortion, \cite[Thm.~5.2]{HK14} shows $a^{-1} \in \SobW^{1,1}_\mathrm{loc}(a(\R^d), \R^d)$ and \cite[Thm.~3.1]{FG95} shows $\dif{a^{-1}}(z) = \dif{a}(a^{-1}(z))^{-1}$ for a.e.\ $z \in a(\R^d)$.
\medskip

\stepemph{Step 4 -- Surjectivity}: Since $a(\R^d)$ is non-empty and open, it suffices to show that $a(\R^d)$ is closed in $\R^d$ to conclude $a(\R^d) = \R^d$. Since $a$ is continuous $a(\overline{Y})$ is compact. Moreover, in view of \cref{Lem:periodic_derivative} we have
\begin{equation*}
	a(\R^d) = \bigcup_{k \in \Z^d} a(k+ \overline{Y}) = \bigcup_{k \in \Z^d} \bar{A}k + a(\overline{Y}).
\end{equation*}
Let $(z_n) \subset a(\R^d)$ with $z_n \to z \in \R^d$. By boundedness of $(z_n)$ and since $\det \bar{A} > 0$, there exist finitely many $k_1, \dots, k_m \in \Z^d$, such that $(z_n) \subset \bigcup_{i=1}^m (\bar{A}k_i + a(\overline{Y}))$. But since this set is closed, $z \in \bigcup_{i=1}^m (\bar{A}k_i + a(\overline{Y})) \subset a(\R^d)$. Hence, $a(\R^d)$ is closed.
\medskip

\stepemph{Step 5 -- Regularity of the inverse}: If $p = \infty$, then the formula $\dif{a^{-1}}(z) = \dif{a}(a^{-1}(z))^{-1}$ for a.e.\ $z \in \R^d$ and \eqref{Eq:regularity_inverse} imply $a^{-1} \in \SobW^{1,\infty}_\mathrm{loc}(\R^d,\R^d)$. For $p < \infty$, we additionally use the transformation rule, see \cite[Thm.~1]{Bal81}, to show
\begin{equation*}
	\int_{a(Y)} \abs{\dif{a^{-1}}(z)}^p \,\dx[z] = \int_Y \abs{\dif{a}(y)^{-1}}^p \det \dif{a}(y) \,\dx[y] < \infty.
\end{equation*}
Hence, $a^{-1} \in \SobW^{1,p}(a(Y),\R^d)$ and by periodicity $a^{-1} \in \SobW^{1,p}_\mathrm{loc}(\R^d,\R^d)$.
\end{proof}

The previous proposition implies that periodic stress-free joints are already Bilipschitz. Even non-periodic stress-free joints are by definition at least piece-wise Bilipschitz. This encourages us to study Bilipschitz maps in the last part of this section. We use the following lemmas later to show uniformity of some estimates w.r.t.\ transformation of the domain by Bilipschitz maps. We denote the set of \emph{Bilipschitz maps} on $U \subset \R^d$ with \emph{Bilipschitz constant} less than or equal to $1 \leq L < \infty$ by
\begin{equation}
	\operatorname{Bil}_L(U, \R^d) := \class{a:U \to \R^d}{\tfrac{1}{L} \abs{x - y} \leq \abs{a(x) - a(y)} \leq L \abs{x-y}, \text{ for all } x,y \in U}.
\end{equation}

\begin{lemma} \label{Lem:Bil_compact_in_W1infty}
Let $U \subset \R^d$ be a Lipschitz domain and $L \in [1,\infty)$. The set $\operatorname{Bil}_L(U, \R^d)$ is sequentially closed w.r.t.\ the weak-$*$ topology in $\SobW^{1,\infty}(U, \R^d)$.
\end{lemma}

\begin{proof}
Let $(a_k) \subset \operatorname{Bil}_L(U, \R^d)$ converging weakly-$*$ to some $a \in \SobW^{1,\infty}(U,\R^d)$. We have to show $a \in \operatorname{Bil}_L(U, \R^d)$. Lower semi-continuity of the norm implies that
\begin{equation*}
	\norm{\dif{a}}_{\Leb^\infty(U)} \leq \liminf_{k \to \infty} \norm{\dif{a_k}}_{\Leb^\infty(U)} \leq L.
\end{equation*}
Compactness implies that $(a_k)$ also uniformly converges to $a$. Moreover, we find an open ball $B \subset \R^d$, such that $\bigcup_{k \in \N} a_k(U) \subset B$. In view of \cite[Section 3.1.1]{EG15} we can extend the inverses $a_k^{-1}$ to Lipschitz maps on $\R^d$ with Lipschitz constants smaller than or equal $L$. Since the sequence of extended inverses $(a_k^{-1})$ is bounded in $\SobW^{1,\infty}(B, \R^d)$, there exists a subsequence (not relabeled), such that $(a_k^{-1})$ weakly-$*$ (and thus especially uniformly) converges to some $\hat{a} \in \SobW^{1,\infty}(B, \R^d)$ with $\norm{\dif{\hat{a}}}_{\Leb^\infty(B)} \leq L$. To conclude the proof we need to show that $\hat{a}$ is the inverse of $a$ on $U$. Indeed, for all $k \in \N$ and $x \in U$, we have
\begin{gather*}
	\abs{\hat{a}(a(x)) - x} = \abs{\hat{a}(a(x)) - a_k^{-1}(a_k(x))} \leq \abs{\hat{a}(a(x)) - \hat{a}(a_k(x))} + \abs{\hat{a}(a_k(x)) - a_k^{-1}(a_k(x))} \\
	\leq \abs{\hat{a}(a(x)) - \hat{a}(a_k(x))} + \norm{\hat{a} - a_k^{-1}}_{\infty} \to 0.
\end{gather*}
Hence $\hat{a} \circ a = \operatorname{id}$ on $U$. Especially, for $z \in a(U)$, we have $\hat{a}(z) \in U$ and thus $a_k^{-1}(z) \in U$ for sufficiently large $k \in \N$. Thus, also
\begin{gather*}
	\abs{a(\hat{a}(z)) - z} = \abs{a(\hat{a}(z)) - a_k(a_k^{-1}(z))} \leq \abs{a(\hat{a}(z)) - a(a_k^{-1}(z))} + \abs{a(a_k^{-1}(z)) - a_k(a_k^{-1}(z))} \\
	\leq \abs{a(\hat{a}(z)) - a(a_k^{-1}(z))} + \norm{a - a_k}_{\infty} \to 0.
\end{gather*}
We conclude, $\hat{a} = a^{-1}$ in $a(U)$ and $a \in \operatorname{Bil}_L(U, \R^d)$.
\end{proof}

\begin{lemma}[cf.\ {\cite[Lemma 3.3]{DNP02}}] \label{Lem:MatrixNorm}
Let $1 \leq p < \infty$, $L \in [1,\infty)$, $U \subset \R^d$ be a Lipschitz domain, $\mathcal{U} \subset \operatorname{Bil}_L(U, \R^d)$ weakly-$*$ closed w.r.t.\ $\SobW^{1,\infty}(U, \R^d)$ and $S \subset U$ be a bounded set with $0 < \mathcal{H}^m(S) < \infty$ for some $1 \leq m \leq d$. Let $S_0$ be the set of all points $x \in S$ with $\mathcal{H}^m(S \cap B(x,\delta)) > 0$ for all $\delta > 0$. Let $K \subset \R^{d\times d}$ be a closed cone, such that for all $F \in K \setminus \set{0}$ and $a \in \mathcal{U}$
\begin{equation*}
	\dim (\ker F) < \dim (\operatorname{aff} a(S_0)),
\end{equation*}
where $\operatorname{aff} a(S_0) \subset \R^d$ denotes the smallest affine space containing $a(S_0)$. Define
\begin{equation}
	\abs{F}_{S,a,p} := \left(\min_{\xi \in \R^d} \int_{S} \abs{F a(x) - \xi}^p \,\dx[\mathcal{H}^{m}(x)]\right)^{1/p}.
\end{equation}
There exists a constant $C > 0$, such that for all $F \in K$ and $a \in \mathcal{U}$
\begin{equation}
	\abs{F} \leq C \abs{F}_{S, a, p}.
\end{equation}
\end{lemma}

\begin{proof}
Suppose the contrary holds. Then, for all $k \in \N$ we find some $a_k \in \mathcal{U}$ and $F_k \in K$ with $\abs{F_k} = 1$, such that
\begin{equation*}
	\tfrac{1}{k} = \tfrac{1}{k}\abs{F_k}^p \geq \abs{F_k}_{S, a_k, p}^p = \int_{S} \abs{F_k a_k(x) - \xi_k}^p \,\dx[\mathcal{H}^{m}(x)],
\end{equation*}
where $\xi_k \in \R^d$ denotes a minimizer in $\abs{F_k}_{S, a_k, p}$. Since $\abs{\placeholder}_{S,a,p}$ and the assumption $\dim (\ker F) < \dim (\operatorname{aff} a(S_0))$ are translation invariant w.r.t.\ $a$, we may without loss of generality assume $(a_k)$ is bounded in $\SobW^{1,\infty}(U, \R^d)$. Hence, we find a subsequence (not relabeled), such that $a_k \to a$ uniformly, $F_k \to F$ and $\xi_k \to \xi$ for some $a \in \mathcal{U}$, $F \in K$ and $\xi \in \R^d$. Note that indeed, $(a_k)$ being bounded in $\SobW^{1,\infty}(U, \R^d)$, $\abs{F_k}$ and $\abs{F_k}_{S,a_k,p}$ being bounded, imply that $(\xi_k)$ is bounded. Then,
\begin{equation*}
	0 = \lim_{k \to \infty} \int_{S} \abs{F_k a_k(x) - \xi_k}^p \,\dx[\mathcal{H}^{m}(x)] = \int_{S} \abs{F a(x) - \xi}^p \,\dx[\mathcal{H}^{m}(x)].
\end{equation*}
Hence, $Fa(x) = \xi$ for all $x \in S_0$. This implies $\dim (\ker F) \geq \dim (\operatorname{aff} a(S_0))$ and thus $F = 0$ by assumption. But this is a contradiction to $\abs{F} = \lim_{k \to \infty} \abs{F_k} = 1$.
\end{proof}

\begin{corollary} \label{Cor:MatrixNorm}
Let $L \in [1,\infty)$, $1 \leq p < \infty$ and $K \subset \R^{d\times d}$ denote the union of the cone generated by $\SO(d) - I$ and the space of skew-symmetric matrizes. Then, there exists a constant $C=C(\Omega, \Gamma, L) > 0$ such that for all $F \in K$ and $a \in \operatorname{Bil}_L(\Omega, \R^d)$, we have
\begin{equation*}
	\abs{F} \leq C \abs{F}_{a(\Gamma), p}, \qquad
	\text{where } \abs{\placeholder}_{a(\Gamma), p} := \abs{\placeholder}_{a(\Gamma), \operatorname{id}, p}.
\end{equation*}
\end{corollary}

\begin{proof}
According to \cite[Chap.~3]{DNP02}, $K$, $S:= \Gamma$ and $\mathcal{U} := \operatorname{Bil}_L(\Omega, \R^d)$ satisfy the assumptions of the previous lemma for $m = d-1$. Hence, using the change of variables rule for boundary integrals, cf.~\cite[Chap.~1.1.3]{KR19}, we get
\begin{gather*}
	\abs{F}^p \leq c_1 \abs{F}_{\Gamma,a,p}^p \leq c_1 \int_{\Gamma} \abs{F a(x) - \xi}^p \,\dx[\mathcal{H}^{d-1}(x)] \leq c_2 \int_{\Gamma} \abs{Fa(x) - \xi}^p \abs{\operatorname{cof}\dif{a}(x)\nu(x)} \,\dx[\mathcal{H}^{d-1}(x)] \\
	= c_2 \int_{a(\Gamma)} \abs{Fz - \xi}^p \,\dx[\mathcal{H}^{d-1}(z)] = c_2 \abs{F}_{a(\Gamma), p}^p,
\end{gather*}
where $c_2$ does not depend on $a \in \operatorname{Bil}_L(\Omega,\R^d)$ and $\xi \in \R^d$ is a minimizer in $\abs{F}_{a(\Gamma), p}$.
\end{proof}

\begin{corollary} \label{Cor:MatrixNormOmega}
Let $L \in [1,\infty)$, $1 \leq p < \infty$. Then, there exists a constant $C= C(\Omega, L) > 0$ such that for all $F \in \R^{d\times d}$ and $a \in \operatorname{Bil}_L(\Omega, \R^d)$, we have
\begin{equation*}
	\abs{F} \leq C \abs{F}_{a(\Omega), p}, \qquad
	\text{where } \abs{\placeholder}_{a(\Omega), p} := \abs{\placeholder}_{a(\Omega), \operatorname{id}, p}.
\end{equation*}
\end{corollary}

\begin{proof}
$K := \R^{d\times d}$, $S:= \Omega$ and $\mathcal{U} := \operatorname{Bil}_L(\Omega, \R^d)$ trivially satisfy the assumptions of \cref{Lem:MatrixNorm} for $m = d$. Hence, the transformation rule implies
\begin{gather*}
	\abs{F}^p \leq c_1 \abs{F}_{\Omega,a,p}^p \leq c_1 \int_{\Omega} \abs{F a(x) - \xi}^p \,\dx \leq c_2 \int_{\Omega} \abs{Fa(x) - \xi}^p \abs{\det\dif{a}(x)} \,\dx \\
	= c_2 \int_{a(\Omega)} \abs{Fz - \xi}^p \,\dx[z] = c_2 \abs{F}_{a(\Omega), p}^p,
\end{gather*}
where $c_2$ does not depend on $a \in \operatorname{Bil}_L(\Omega,\R^d)$ and $\xi \in \R^d$ is a minimizer in $\abs{F}_{a(\Omega), p}$.
\end{proof}

\subsection{Extension operator for rigidity in Jones domains; Korn inequality and rigidity estimates}
\label{Sec:Proof:ExtensionOperatorRigidity}

In this section we prove \cref{Thm:RigidityEstimateJonesDomain} and conclude \cref{Cor:KornWithLpNormOfu,Cor:KornWithBoundaryValue,Cor:KornPeriodic}.

\paragraph{Extension operator for rigidity.}
As proposed in \cref{Sec:GeometricRigidity}, the results are based on an extension operator that we shall introduce first. 
Before we state the result, recall our notation for mixed growth decompositions introduced in \eqref{Eq:NotationMixedGrowthDecompositions}. We introduce the following notation for cubes.

\begin{definition}
For the cube $Q := a + [-\tfrac{l}{2}, \tfrac{l}{2}]^d$ with $a \in \R^d$ and $l > 0$, we denote the center point by $\bar{x}(Q) := a$ and the edge length by $l(Q) := l$. Moreover, for $\alpha > 0$, we define the scaled cube $\alpha Q := \bar{x}(Q) + \alpha(Q - \bar{x}(Q))$. We use the same definitions for open and half-open cubes. 
\end{definition}

The extension operator controls the distance to $\SO(d)$, $\R^{d\times d}_{\sym}$ and other sets simultaneously, as long as a geometric rigidity like statement holds on cubes. To unify this, we introduce the following notion:

\begin{definition}[$(\mathcal{A},p,q)$-rigidity] \label{Def:GeneralRigidity}
Let $Q \subset Q^+ \subset \R^d$ measurable, $p, q \in [1,\infty]$ and $\mathcal{A} \subset \R^{d\times d}$. We say $(\mathcal{A},p,q)$-rigidity holds on $Q$ w.r.t.\ $Q^+$, if the following statement holds. We find a constant $c > 0$, such that for all $u \in \SobW^{1,1}(Q^+, \R^d)$ and all decompositions $\dist(\dif{u},\mathcal{A}) = F_{\dist(\dif{u},\mathcal{A})} + G_{\dist(\dif{u},\mathcal{A})}$ in $\Leb^p + \Leb^q(Q^+)$, we find some matrix $M \in \R^{d\times d}$ and a decomposition $\dif{u} - M = F_{\dif{u}-M} + G_{\dif{u} - M}$ in $\Leb^p+\Leb^q(Q, \R^{d\times d})$, such that
\begin{equation} \label{Eq:ArbitraryRigidityEstimate}
\begin{aligned}
	\norm{F_{\dif{u} - M}}_{\Leb^p(Q)} &\leq c \norm{F_{\dist(\dif{u}, \mathcal{A})}}_{\Leb^p(Q^+)}, \\
	\norm{G_{\dif{u} - M}}_{\Leb^q(Q)} &\leq c \norm{G_{\dist(\dif{u}, \mathcal{A})}}_{\Leb^q(Q^+)}.
\end{aligned}
\end{equation}
We say $(\mathcal{A},p,q)$-rigidity holds on cubes, if $(\mathcal{A},p,q)$-rigidity holds on any cube $Q$ w.r.t.\ $Q^+ = \tfrac{33}{32}Q$.
\end{definition}

The choice $\tfrac{33}{32}$ for the scaling is technical and not important. The standard choices for $\mathcal{A}$ are $\R^{d\times d}_{\skewMat}$ which is Korn's inequality and $\SO(d)$ which is the geometric rigidity estimate, cf.\ \cref{Sec:GeometricRigidity}.

\begin{theorem} \label{Thm:Extenion_operator}
Let $U \subset \R^d$ be an open, bounded $(e, \delta)$-domain, $\rho := \min\set{\frac{1}{2}\operatorname{diam}(U), \delta}$. For $\gamma > 0$, define
\begin{equation*}
	U^+_\gamma := \class{x \in \R^d}{\dist(x, U) \leq \gamma}, \qquad
	U^-_\gamma := \class{x \in U}{\dist(x, \partial U) \geq \gamma},
\end{equation*}
and note that $U^-_\gamma \subset\subset U \subset\subset U^+_\gamma$.
There exist constants $0 < \alpha' < \alpha$ and $\alpha'' > 0$ (which we relate to the sets $U^-_{\rho\alpha''} \subset U \subset U^+_{\rho\alpha'} \subset U^+_{\rho\alpha}$) and a bounded, linear extension operator $E: \SobW^{1,1}(U, \R^d) \to \SobW^{1,1}_\mathrm{loc}(\R^d, \R^d)$ with
\begin{enumerate}[(a)] 
	\item $Eu = u$ a.e.\ in $U$ and
	\item $\supp Eu \subset U^+_{\rho\alpha}$,
\end{enumerate}
such that the following holds:
\medskip

Let $r \in [1,\infty]$, $1 \leq p \leq q \leq \infty$ and $\mathcal{A} \subset \R^{d\times d}$, such that $(\mathcal{A},p,q)$-rigidity holds on cubes. Then, there exists a constant $c > 0$, such that for all $u \in \SobW^{1,1}(U, \R^d)$ and all decompositions $\dist(\dif{u}, \mathcal{A}) = F_{\dist(\dif{u}, \mathcal{A})} + G_{\dist(\dif{u}, \mathcal{A})}$ in $\Leb^p+\Leb^q(U, \R^{d\times d})$, the following estimates hold.
\begin{enumerate}[(a)] \setcounter{enumi}{2}
	\item We find a decomposition $Eu = F_{Eu} + G_{Eu} + H_{Eu}$ in $\Leb^p+\Leb^q+\Leb^r(\R^d,\R^d)$, such that
	\begin{equation}
	\begin{aligned}
		\norm{F_{Eu}}_{\Leb^p(\R^d)} & \leq c \norm{F_{\dist(\dif{u}, \mathcal{A})}}_{\Leb^p(U)}, \\
		\norm{G_{Eu}}_{\Leb^q(\R^d)} & \leq c \norm{G_{\dist(\dif{u}, \mathcal{A})}}_{\Leb^q(U)}, \\
		\norm{H_{Eu}}_{\Leb^r(\R^d)} & \leq c \norm{u}_{\Leb^r(U)}.
	\end{aligned}
	\end{equation}
	Especially, for $p = q = r$,
	\begin{equation}
		\norm{Eu}_{\Leb^p(\R^d)} \leq c \left(\norm{u}_{\Leb^p(U)} + \norm{\dist(\dif{u}, \mathcal{A})}_{\Leb^p(U)}\right).
	\end{equation}
	\item We find a decomposition $\dist(\dif{Eu}, \mathcal{A}) = F_{\dist(\dif{Eu}, \mathcal{A})} + G_{\dist(\dif{Eu}, \mathcal{A})} + H_{\dist(\dif{Eu}, \mathcal{A})}$ in  \linebreak
	$\Leb^p+\Leb^q+\Leb^r(\R^d,\R^{d\times d})$ with $H_{\dist(\dif{Eu}, \mathcal{A})} = 0$ a.e.\ in $U^+_\mathrm{\rho\alpha'}$, such that
	\begin{equation}
	\begin{aligned}
		\norm{F_{\dist(\dif{Eu}, \mathcal{A})}}_{\Leb^p(\R^d)} & \leq c \norm{F_{\dist(\dif{u}, \mathcal{A})}}_{\Leb^p(U)}, \\
		\norm{G_{\dist(\dif{Eu}, \mathcal{A})}}_{\Leb^q(\R^d)} & \leq c \norm{G_{\dist(\dif{u}, \mathcal{A})}}_{\Leb^q(U)}, \\
		\norm{H_{\dist(\dif{Eu}, \mathcal{A})}}_{\Leb^r(\R^d)} & \leq c \left(\rho^{-1}\norm{u}_{\Leb^r(U^-_{\rho\alpha''})} + \norm{\dif{u}}_{\Leb^r(U^-_{\rho\alpha''})}\right).
	\end{aligned}
	\end{equation}
	Especially, for $p = q = r$,
		\begin{equation}
			\norm{\dist(\dif{Eu}, \mathcal{A})}_{\Leb^p(U^+_\mathrm{\rho\alpha'})} \leq c \norm{\dist(\dif{u}, \mathcal{A})}_{\Leb^p(U)}.
		\end{equation}
\end{enumerate}
Moreover, the constants $\alpha$, $\alpha'$, $\alpha''$ and $c$ depend on the domain $U$ only via its first Jones coefficient $e$.
\end{theorem}

We want to emphasize that the extension operator does not depend on the choice of $\mathcal{A}$ but is stable for \emph{all} admissible choices of $\mathcal{A}$, including Korn's inequality ($\mathcal{A} = \R^{d\times d}_{\sym}$) and the geometric rigidity estimate ($\mathcal{A} = \SO(d)$). Especially, (c) and (d) for the trivial choice $\mathcal{A} = \emptyset$ show that $E$ is a bounded operator w.r.t.\ $\SobW^{1,p}$ for any $p \in [1,\infty]$. We follow \cite{Jon81,DM04} for the construction of the extension operator and the proof of \cref{Thm:Extenion_operator}. The definition of the extension operator relies on Whitney decompositions of the domain $U$ and $\R^d \setminus U$ and a suitable reflection of the cubes from the outside to the inside of $U$ close to the boundary. For the readers convenience we recall here the relevant definitions and statements from \cite{Jon81}.

\begin{definition}[cf.\ \cite{Jon81}] \label{Def:WhitneyDecomposition}
Let $U \subset \R^d$ open. A Whitney decomposition of $U$ is a sequence of closed, dyadic cubes $Q_k$, $k \in \N$, such that $U = \bigcup_{k \in \N} Q_k$ and
\begin{enumerate}[(i)]
	\item $l(Q_k) \leq \dist(Q_k, \partial\Omega) \leq 4\sqrt{d}\, l(Q_k)$,
	\item $\operatorname{int}(Q_j) \cap \operatorname{int}(Q_k) = \emptyset$, whenever $j \neq k$,
	\item $\frac{1}{4} \leq \frac{l(Q_j)}{l(Q_k)} \leq 4$, whenever $Q_j \cap Q_k \neq \emptyset$.
\end{enumerate}
\end{definition}

\begin{lemma}[Reflection, {cf.\ \cite{Jon81}}]~ \label{Lem:Reflection}
\begin{enumerate}[(a)]
	\item Every open set $U \subset \R^d$ admits a Whitney decomposition, cf.\ \cite[Thm.~VI.1]{Ste71}.
	
	\item There exists a constant $c = c(d, e) > 0$, such that if $U$ is an $(e, \delta)$-domain, the following statements hold. Let $\rho := \min\set{\frac{1}{2}\operatorname{diam}(U), \delta}$, 
	\begin{itemize}
		\item $W_1 := (S_k)_{k \in \N}$ denote a Whitney decomposition of $U$,
		\item $W_2 := (Q_j)_{j \in \N}$ denote a Whitney decomposition of $\R^d \setminus U$ and
		\item $W_3 := \class{Q_j \in W_2}{l(Q_j) \leq \frac{\rho}{c}}$.
	\end{itemize}
	For every cube $Q_j \in W_3$, there exists a reflected cube $Q_j^*= S_k \in W_1$, such that
	\begin{equation*}
		1 \leq \tfrac{l(Q_j^*)}{l(Q_j)} \leq 4, \quad \dist(Q_j, Q_j^*) \leq c\,l(Q_j),
	\end{equation*}
	and if $Q_j, Q_k \in W_3$ with $Q_j \cap Q_k \neq \emptyset$, there exists a chain $F_{j,k} := \set{Q_j^*=S_1, \dots, S_m = Q_k^*} \subset W_1$, i.e.\ $S_j \cap S_{j+1} \neq \emptyset$, with chain length $m \leq c$. 
	
	\item There exists a constant $C = C(d) > 0$ and a partition of unity $(\varphi_j) \subset \Cont^\infty_c(\R^d, [0,1])$ subordinate to $W_3$, such that
	\begin{equation*}
		\supp \varphi_j \subset \tfrac{17}{16}Q_j, \quad \sum_{Q_j \in W_3} \varphi_j \equiv 1 \text{ on } \bigcup W_3, \quad |\nabla\varphi_j| \leq C\, l(Q_j)^{-1}.
	\end{equation*}
\end{enumerate}
\end{lemma}

Throughout this section let $U \subset \R^d$ an $(e,\delta)$-domain and $W_1$, $W_2$, $W_3$ and $(\varphi_j)$ as in \cref{Lem:Reflection}. Note that by property (i) of \cref{Def:WhitneyDecomposition}, $W_3$ consists of the cubes that are close to the boundary of $U$. In fact, we may choose $0 < \alpha' < \alpha$ only depending on $e$ and $d$, such that
\begin{equation*}
	U^+_{\rho\alpha'} \subset \bigcup W_3\cup\overline U \subset U^+_{\rho\alpha}.
\end{equation*}
For $u \in \SobW^{1,1}(U, \R^d)$, we define the extension of $u$ as,
\begin{equation} \label{Eq:Formula_extension}
	Eu(x) := \begin{cases} 
		u(x) & \text{if } x \in U, \\ 
		\sum_{Q_j \in W_3} P_{Q_j^*}[u](x) \varphi_j(x) & \text{if } x \in \operatorname{int}(\R^d \setminus U),
	\end{cases}
\end{equation}
where $P_Q[u]$ denotes the affine map
\begin{equation} \label{Eq:DefinitionPolynomialPQ}
	P_Q[u](x) := \bar{u}_Q + M(x - \bar{x}_Q), \qquad x \in \R^d,
\end{equation}
with
\begin{equation}
\newcommand{\smallfint}{\begingroup\textstyle\fint\endgroup}
	\bar{u}_Q := \fint_Q u, \quad M := \fint_Q \dif u, \quad \bar{x}_Q := \fint_Q x \,\dx.
\end{equation}
Jones showed in \cite[Lem.~2.3]{Jon81} that $\abs{\partial U} = 0$. Thus, the formula defines $Eu$ up to a null-set. We show later that indeed $Eu \in \SobW^{1,1}_\text{loc}(\R^d, \R^d)$. 

\begin{remark}
The main difference of the extension operators in \cite{Jon81,DM04} and in our work is the choice of $M$. We want to motivate our choice. By studying \cite{DM04} we identify the following key properties that $M$ needs to satisfy:
\begin{enumerate}[(a)]
	\item $M$ is linear w.r.t.\ $u$, such that $E$ is linear.
	\item $\abs{M}$ can be controlled by $\dif{u}$, such that $E$ is a bounded operator.
	\item We can estimate the difference $u - P_Q[u]$ by $\dist(\dif{u}, \mathcal{A})$ in $\SobW^{1,p}(Q)$ to be able to control $\dist(\dif{Eu}, \mathcal{A})$ (also in the mixed growth sense).
\end{enumerate}
The fact that $M = \fint_Q \dif{u}$ is a suitable choice is now due to the following simple observation. In \cref{Def:GeneralRigidity}, we can always choose the explicit matrix $M = \fint_Q \dif u$. Indeed, let $\tilde{M} \in \R^{d\times d}$ denote a matrix that satisfies the statement in the definition of $(\mathcal{A}, p, q)$-rigidity. Then, the inequality
\begin{equation*}
	\abs{M - \tilde{M}} = \abs{\fint_Q \dif{u} - \tilde{M}} \leq \fint_U \abs{\dif u - \tilde{M}} \leq \fint_Q \abs{F_{\dif{u} - \tilde{M}}} + \fint_Q \abs{G_{\dif{u} - \tilde{M}}}
\end{equation*}
implies this statement by using $\dif{u} - M = (\dif{u} - \tilde{M}) + (M - \tilde{M})$ and \cref{Rem:Properties_rigidity} (i) below. With this choice, \eqref{Eq:ArbitraryRigidityEstimate} reads
\begin{equation} \label{Eq:Estimate_PQ}
\begin{aligned}
	\norm{F_{\dif{u} - \dif{P_Q[u]}}}_{\Leb^p(Q)} &\leq c \norm{F_{\dist(\dif{u}, \mathcal{A})}}_{\Leb^p(Q^+)}, \\
	\norm{G_{\dif{u} - \dif{P_Q[u]}}}_{\Leb^q(Q)} &\leq c \norm{G_{\dist(\dif{u}, \mathcal{A})}}_{\Leb^q(Q^+)}.
\end{aligned}
\end{equation}
Moreover, since $\fint_Q u - P_Q[u] = 0$, a mixed growth version of the Poincaré-Wirtinger inequality, see \cref{Prop:Poincare_interpol}, yields a decomposition $u - P_Q[u] = F_{u - P_Q[u]} + G_{u - P_Q[u]}$ in $\Leb^p + \Leb^q(Q, \R^d)$ with
\begin{equation} \label{Eq:Estimate_DPQ}
\begin{aligned}
	\norm{F_{u - P_Q[u]}}_{\Leb^p(Q)} &\leq c \operatorname{diam}(Q) \norm{F_{\dist(\dif{u}, \mathcal{A})}}_{\Leb^p(Q^+)}, \\
	\norm{G_{u - P_Q[u]}}_{\Leb^q(Q)} &\leq c \operatorname{diam}(Q) \norm{G_{\dist(\dif{u}, \mathcal{A})}}_{\Leb^q(Q^+)}.
\end{aligned}
\end{equation}
These are exactly the estimates needed for (c). We note that using a similar argument we can also always use in \cref{Def:GeneralRigidity} the explicit choice $M \in \operatorname{Arg\,min}_{N \in \overline{\mathcal A}} \big|N - \fint_Q\dif{u}\big|  \subset \overline{\mathcal A}$.
\end{remark}

Before we proceed with the proof of \cref{Thm:Extenion_operator}, we provide some further remarks.

\begin{remark}~ \label{Rem:Properties_rigidity}
\begin{enumerate}[(i)]
	\item To conclude \eqref{Eq:Estimate_PQ} and \eqref{Eq:Estimate_DPQ} from the arguments above, we used the fact that an inequality of the form $\abs{v} \leq \tilde{f} + \tilde{g}$ for some $\tilde{f} \in \Leb^p(Q)$, $\tilde{g} \in \Leb^q(Q)$ with $\tilde{f},\tilde{g} \geq 0$ already implies the existence of a decomposition $v = f + g$ with $\norm{\tilde{f}}_{\Leb^p} \leq \norm{f}_{\Leb^p}$, $\norm{\tilde{g}}_{\Leb^q} \leq \norm{g}_{\Leb^q}$. This has already been pointed out in the notation section of \cite{CDM14}. We shall use this fact without further comments several times throughout the proofs below.
	
	\item The constant in \cref{Def:GeneralRigidity} is invariant under scaling and translation of the domains $Q$ and $Q^+$, which can be seen by a standard scaling argument. Moreover, this holds true with the explicit choice $M = \fint_Q \dif u$ from above. Especially, if $(\mathcal{A},p,q)$ rigidity holds on one cube, then it holds on all.
	
	\item If $(\mathcal{A}, p, q)$-rigidity holds on two open sets $Q_1, Q_2 \subset \R^d$ with $Q_1 \cap Q_2 \neq \emptyset$, then it is not hard to show that it also holds on $Q_1 \cup Q_2$. Especially, if it holds on cubes, then it holds on any finite union of intersecting cubes. We carry out the argument in the proof of \cref{Thm:RigidityEstimateJonesDomain}.
\end{enumerate}
\end{remark}

The proof of \cref{Thm:Extenion_operator} is analogous to \cite{Jon81,DM04}. Since one has to be careful to deal with the mixed growth estimates and the general $(\mathcal{A},p,q)$-rigidity, we provide the main arguments for the readers convenience. To that end, we fix $r \geq 1$, $1 \leq p \leq q \leq \infty$ and $\mathcal{A} \subset \R^{d\times d}$, such that $(\mathcal{A}, p, q)$-rigidity holds on cubes. Further, we consider a decomposition $\dist(\dif u, \mathcal{A}) = F_{\dist(\dif u, \mathcal{A})} + G_{\dist(\dif u, \mathcal{A})}$ in $\Leb^p+\Leb^q(U)$. 
\medskip

The idea of the proof relies on two observations. First, by construction we have $Eu \approx P_{Q_j^*}[u]$ on $Q_j$ and in view of \eqref{Eq:Estimate_PQ} and \eqref{Eq:Estimate_DPQ}, $u \approx P_{Q_j^*}[u]$ on $Q_j^*$, see \cref{Fig:ReflectionOfPolynomials}. Hence, to relate $Eu$ to $u$, the idea is to relate $P_{Q_j^*}[u]|_{Q_j}$ to $P_{Q_j^*}[u]|_{Q_j^*}$. This can be done, since $P_{Q_j^*}[u]$ is a polynomial and $Q_j$ and $Q_j^*$ have a similar size and controlled distance. This is made rigorous in the following lemma.

\begin{figure}[ht]
\begin{center}
	\includegraphics[page=1]{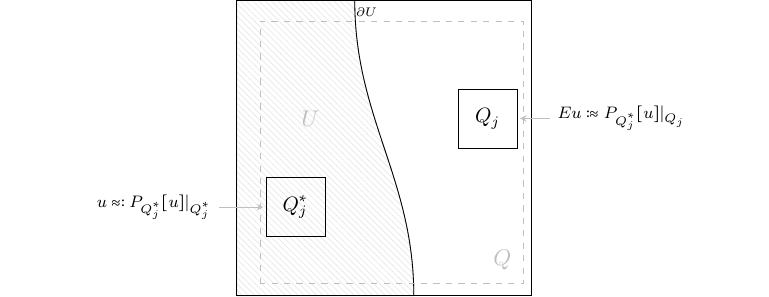}
\end{center}
\caption{\cref{Lem:ReflectionOfPolynomials} allows to relate $P_{Q_j^*}[u]|_{Q_j}$ to $P_{Q_j^*}[u]|_{Q_j^*}$ and thus $Eu|_{Q_j}$ to $u$.}
\label{Fig:ReflectionOfPolynomials}
\end{figure}

\begin{lemma}[{cf.\ \cite[Lem.~2.1]{Jon81}}] \label{Lem:ReflectionOfPolynomials}
Let $m \in \N$, $1 \leq p \leq q \leq \infty$, $Q \subset \R^d$ a cube and $E,F \subset Q$ measurable sets satisfying $\abs{E},\abs{F} \geq \gamma \abs{Q}$ for some $\gamma > 0$. There exists a constant $c=c(\gamma,m,d)$, such that for any polynomial $P$ of degree at most $m$ and decomposition $P = F_{P|_E} + G_{P|_E}$ in $\Leb^p+\Leb^q(E)$, we find a decomposition $P = F_{P|_F} + G_{P|_F}$ in $\Leb^p+\Leb^q(F)$ satisfying
\begin{equation}
\begin{aligned}
	\norm{F_{P|_F}}_{\Leb^p(F)} &\leq c \norm{F_{P|_E}}_{\Leb^p(E)}, \\
	\norm{G_{P|_F}}_{\Leb^q(F)} &\leq c \norm{G_{P|_E}}_{\Leb^p(E)}.
\end{aligned}
\end{equation}
(For the proof see \cref{Sec:mixed-growth-estimates}.)
\end{lemma}

The second observation is that due to the definition of the partition of unity, $Eu|_{Q_j}$ admits contributions $P_{Q_k^*}[u]$ only from neighboring cubes $Q_k$, i.e.\ only if $Q_k \cap Q_j \neq 0$, see \cref{Fig:Estimate_chain_W1}. Thus, we need to relate $P_{Q_j^*}[u]$ and $P_{Q_k^*}[u]$ on $Q_j$. To do so, together with \cref{Lem:ReflectionOfPolynomials} we can use that the reflections of neighboring cubes are not necessarily neighbors but at least connected by a controlled chain in view of \cref{Lem:Reflection}. The following lemma provides estimates for such a case.

\begin{figure}[ht]
\begin{center}
	\includegraphics[page=2]{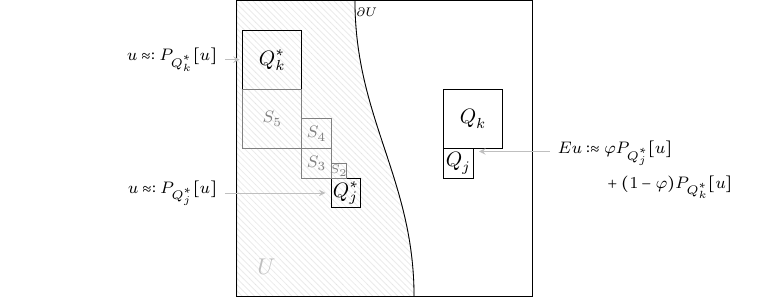}
\end{center}
\caption{\cref{Lem:Estimate_chain_W1} allows to relate $P_{Q_j^*}[u]$ and $P_{Q_k^*}[u]$ which is necessary to estimate $Eu$ close to $Q_j \cap Q_k$.}
\label{Fig:Estimate_chain_W1}
\end{figure}

\begin{lemma}[{cf.\ \cite[Lem.~3.1]{Jon81} and \cite[Lem.~2.3]{DM04}}] \label{Lem:Estimate_chain_W1}
Let $m \in \N$ and $\mathcal{F} := \set{S_1, \dots, S_m} \subset W_1$ a chain of cubes in $W_1$. Then, we find decompositions $P_{S_1}[u] - P_{S_m}[u] = F_{P_{S_1}[u] - P_{S_m}[u]} + G_{P_{S_1}[u] - P_{S_m}[u]}$ in $\Leb^p+\Leb^q(S_1, \R^d)$ and $\dif{P_{S_1}[u]} - \dif{P_{S_m}[u]} = F_{\dif{P_{S_1}[u]} - \dif{P_{S_m}[u]}} + G_{\dif{P_{S_1}[u]} - \dif{P_{S_m}[u]}}$ in $\Leb^p+\Leb^q(S_1, \R^{d\times d})$, such that
\begin{gather}
	\begin{aligned}
		\norm{F_{P_{S_1}[u] - P_{S_m}[u]}}_{\Leb^p(S_1)} &\leq c l(S_1) \norm{F_{\dist(\dif{u}, \mathcal{A})}}_{\Leb^p(\bigcup \frac{33}{32}\mathcal{F})}, \\
		\norm{G_{P_{S_1}[u] - P_{S_m}[u]}}_{\Leb^q(S_1)} &\leq c l(S_1) \norm{G_{\dist(\dif{u}, \mathcal{A})}}_{\Leb^q(\bigcup \frac{33}{32}\mathcal{F})},
	\end{aligned} \\[\baselineskip]
	\begin{aligned}
		\norm{F_{\dif{P_{S_1}[u]} - \dif{P_{S_m}[u]}}}_{\Leb^p(S_1)} &\leq c \norm{F_{\dist(\dif{u}, \mathcal{A})}}_{\Leb^p(\bigcup \frac{33}{32}\mathcal{F})}, \\
		\norm{G_{\dif{P_{S_1}[u]} - \dif{P_{S_m}[u]}}}_{\Leb^q(S_1)} &\leq c \norm{G_{\dist(\dif{u}, \mathcal{A})}}_{\Leb^q(\bigcup \frac{33}{32}\mathcal{F})},
	\end{aligned}
\end{gather}
for some constant $c = c(d,m,p,q) > 0$. Here, $\frac{33}{32}\mathcal{F} := \set{\frac{33}{32}S_1, \dots, \frac{33}{32}S_m}$.
\end{lemma}

\begin{proof}
Using a telescopic sum, we obtain the estimate
\begin{align*}
	\abs{P_{S_1}[u] - P_{S_m}[u]} &\leq \sum_{i=1}^{m-1} \abs{P_{S_i}[u] - P_{S_{i+1}}[u]}  \\
	&\leq \sum_{i=1}^{m-1} \abs{P_{S_i}[u] - P_{S_i \cup S_{i+1}}[u]} + \abs{P_{S_i \cup S_{i+1}}[u] - P_{S_{i+1}}[u]}, \qquad \text{on } S_1.
\end{align*}
In view of \cref{Lem:ReflectionOfPolynomials} it suffices to estimate $P_{S_i}[u] - P_{S_i \cup S_{i+1}}[u]$ and $P_{S_i \cup S_{i+1}}[u] - P_{S_{i+1}}[u]$ on $S_i$ (respectively on $S_{i+1}$) instead of $S_1$.  Here, we can add and subtract $u$ and then use \eqref{Eq:Estimate_PQ} and \eqref{Eq:Estimate_DPQ} with $Q = S_i$ and $Q^+ = \frac{33}{32}S_i$ (respectively with $Q = S_{i+1}$ and $Q = S_i \cup S_{i+1}$ and $Q^+$ analogously) to obtain the desired mixed growth estimates in terms of $\dist(\dif{u}, \mathcal{A})$. Note that in order to control $\gamma$ in \cref{Lem:ReflectionOfPolynomials}, we use that $S_1$ and $S_i$, $S_{i+1}$ are close in view of \cref{Lem:Reflection}. Moreover, the constant in \eqref{Eq:Estimate_PQ} and \eqref{Eq:Estimate_DPQ} can be chosen uniformly by scaling and translation invariance of the constant in $(\mathcal{A},p,q)$-rigidity.
\end{proof}

We use these lemmas to estimate $Eu$. We only provide a sketch of the proofs. For the details we refer to the related lemmas in \cite{Jon81} and \cite{DM04}.

\begin{lemma}[{cf.\ \cite[Lem.~3.2]{Jon81} and \cite[Lem.~2.4]{DM04}}] \label{Lem:Estimate_W3}
Let $Q_j \in W_3$. We define
\begin{equation}
	\mathcal{F}(Q_j) := \class{\tfrac{33}{32}S_i}{S_i \in F_{j,k}, Q_k \in W_3, Q_j \cap Q_k \neq \emptyset},
\end{equation}
where $F_{j,k}$ denote the chains connecting $Q_j^*$ and $Q_k^*$ in $W_1$ defined in \cref{Lem:Reflection} (b). We find decompositions $Eu = F_{Eu} + G_{Eu} + H_{Eu}$ in $\Leb^p+\Leb^q+\Leb^r(Q_j, \R^d)$ and $\dist(\dif{Eu}, \mathcal{A}) = F_{\dist(\dif{Eu}, \mathcal{A})} + G_{\dist(\dif{Eu}, \mathcal{A})}$ in $\Leb^p+\Leb^q(Q_j)$, such that
\begin{gather}
	\begin{aligned} \label{Eq:EstimateEuOnW3}
		\norm{F_{Eu}}_{\Leb^p(Q_j)} &\leq c l(Q_j) \norm{F_{\dist(\dif{u}, \mathcal{A})}}_{\Leb^p(\bigcup\mathcal{F}(Q_j))}, \\
		\norm{G_{Eu}}_{\Leb^q(Q_j)} &\leq c l(Q_j) \norm{G_{\dist(\dif{u}, \mathcal{A})}}_{\Leb^q(\bigcup\mathcal{F}(Q_j))}, \\
		\norm{H_{Eu}}_{\Leb^r(Q_j)} &\leq c \norm{u}_{\Leb^r(Q_j^*)},
	\end{aligned} \\[\baselineskip]
	\begin{aligned} \label{Eq:EstimateDEuOnW3}
		\norm{F_{\dist(\dif{Eu}, \mathcal{A})}}_{\Leb^p(Q_j)} &\leq C \norm{F_{\dist(\dif{u}, \mathcal{A})}}_{\Leb^p(\bigcup\mathcal{F}(Q_j))}, \\
		\norm{G_{\dist(\dif{Eu}, \mathcal{A})}}_{\Leb^q(Q_j)} &\leq C \norm{G_{\dist(\dif{u}, \mathcal{A})}}_{\Leb^q(\bigcup\mathcal{F}(Q_j))},
	\end{aligned}
\end{gather}
for some constants $c = c(d, e, \mathcal{A}, p, q, r) > 0$ and $C = C(d, e, \mathcal{A}, p, q) > 0$.
\end{lemma}

\begin{proof}
The partition of unity is constructed such that $\varphi_k|_{Q_j} \equiv 0$ whenever $Q_k \cap Q_j = \emptyset$ and $\sum_{Q_k \in W_3}\varphi_k \equiv 1$ on $Q_j$. Thus, we have the formula,
\begin{equation*}
	Eu|_{Q_j} = P_{Q_j^*}[u] + \sum_{\substack{Q_k \in W_3, \\ Q_j \cap Q_k \neq \emptyset}} (P_{Q_k^*}[u] - P_{Q_j^*}[u]) \varphi_k \qquad
	\text{on } Q_j.
\end{equation*}
Again, using \cref{Lem:ReflectionOfPolynomials}, it suffices to estimate $P_{Q_j^*}[u]$ and $P_{Q_k^*}[u] - P_{Q_j^*}[u]$ on $Q_j^*$ instead of $Q_j$. To show \eqref{Eq:EstimateEuOnW3}, we can now estimate on $Q_j^*$ the latter term using \cref{Lem:Estimate_chain_W1} and the former term using \eqref{Eq:Estimate_PQ} and the formula $P_{Q_j^*}[u] = (P_{Q_j^*}[u] - u) + u$. 
For the derivative, we obtain from the formula above,
\begin{equation*}
	(\dif{Eu} - \dif{P_{Q_j^*}[u]})|_{Q_j} = \sum_{\substack{Q_k \in W_3, \\ Q_j \cap Q_k \neq \emptyset}} (P_{Q_k^*}[u] - P_{Q_j^*}[u]) \nabla\varphi_k^T + \varphi_k (\dif{P_{Q_k^*}[u]} - \dif{P_{Q_j^*}[u]}) \qquad
	\text{on } Q_j,
\end{equation*}
We can estimate the right-hand side analogously to the procedure for $Eu|_{Q_j}$ and make the following observation from which we can infer \eqref{Eq:EstimateDEuOnW3}:
\begin{align*}
	\dist(\dif{Eu}, \mathcal{A}) &\leq \abs{\dif{Eu} - \dif{P_{Q_j^*}[u]}} + \dist(\dif{P_{Q_j^*}[u]}, \mathcal{A}) \\
	&\leq \abs{\dif{Eu} - \dif{P_{Q_j^*}[u]}} + \fint_{Q_j^*} \abs{\dif{u} - \dif{P_{Q_j^*}[u]}}  + \fint_{Q_j^*} \dist(\dif{u}, \mathcal{A}).
\end{align*}
\end{proof}

\begin{lemma}[{cf.\ \cite[Lem.~3.3]{Jon81} and \cite[Lem.~2.5]{DM04}}] \label{Lem:Estimate_W2_wo_W3}
Let $Q_j \in W_2 \setminus W_3$. We define
\begin{equation}
	\mathcal{F}(Q_j) := \class{\tfrac{33}{32}Q_k^*}{Q_k \in W_3, Q_j \cap Q_k \neq \emptyset}.
\end{equation}
We find decompositions $Eu = F_{Eu} + G_{Eu} + H_{Eu}$ in $\Leb^p+\Leb^q+\Leb^r(Q_j, \R^d)$ and $\dist(\dif{Eu}, \mathcal{A}) = F_{\dist(\dif{Eu}, \mathcal{A})} + G_{\dist(\dif{Eu}, \mathcal{A})} + H_{\dist(\dif{Eu}, \mathcal{A})}$ in $\Leb^p+\Leb^q+\Leb^r(Q_j)$, such that
\begin{gather}
	\begin{aligned} \label{Eq:EstimateEuOnW2}
		\norm{F_{Eu}}_{\Leb^p(Q_j)} &\leq c \rho \norm{F_{\dist(\dif{u}, \mathcal{A})}}_{\Leb^p(\bigcup\mathcal{F}(Q_j))}, \\
		\norm{G_{Eu}}_{\Leb^q(Q_j)} &\leq c \rho \norm{G_{\dist(\dif{u}, \mathcal{A})}}_{\Leb^q(\bigcup\mathcal{F}(Q_j))}, \\
		\norm{H_{Eu}}_{\Leb^r(Q_j)} &\leq c \norm{u}_{\Leb^r(\bigcup\mathcal{F}(Q_j))},
	\end{aligned}
\intertext{and}
	\begin{aligned} \label{Eq:EstimateDEuOnW2}
		\norm{F_{\dist(\dif{Eu}, \mathcal{A})}}_{\Leb^p(Q_j)} &\leq c \norm{F_{\dist(\dif{u}, \mathcal{A})}}_{\Leb^p(\bigcup\mathcal{F}(Q_j))}, \\
		\norm{G_{\dist(\dif{Eu}, \mathcal{A})}}_{\Leb^q(Q_j)} &\leq c \norm{G_{\dist(\dif{u}, \mathcal{A})}}_{\Leb^q(\bigcup\mathcal{F}(Q_j))}, \\
		\norm{H_{\dist(\dif{Eu}, \mathcal{A})}}_{\Leb^q(Q_j)} &\leq c \left(\rho^{-1}\norm{u}_{\Leb^r(\bigcup \mathcal{F}(Q_j))} + \norm{\dif{u}}_{\Leb^r(\bigcup \mathcal{F}(Q_j))}\right),
	\end{aligned}
\end{gather}
for some constant $c = c(d, e, \mathcal{A}, p, q, r) > 0$.
\end{lemma}

\begin{proof}
Note that $\mathcal{F}(Q_j)$ is empty only if $Eu|_{Q_j} \equiv 0$. Thus, we may assume $\mathcal{F}(Q_j) \neq \emptyset$. Then, \eqref{Eq:EstimateEuOnW2} follows from the formula
\begin{equation*}
	Eu|_{Q_j} = \sum_{\substack{Q_k \in W_3, \\ Q_j \cap Q_k \neq \emptyset}} P_{Q_k^*}[u] \varphi_k \qquad
	\text{on } Q_j,
\end{equation*}
where we estimate $P_{Q_k^*}[u]$ as in the previous lemma using $P_{Q_k^*}[u] = (P_{Q_k^*}[u] - u) + u$ on $Q_k^*$. Note that since $\mathcal{F}(Q_j) \neq \emptyset$, we have $\frac{4\rho}{c_1} \geq l(Q_j) \geq \frac{\rho}{c_1}$. Similarly, we obtain \eqref{Eq:EstimateDEuOnW2} from the formulas
\begin{align*}
	\dif{Eu} &= \sum_{\substack{Q_k \in W_3, \\ Q_j \cap Q_k \neq \emptyset}} P_{Q_k^*}[u] \nabla\varphi_k^T + \varphi_k \dif{P_{Q_k^*}[u]}, \quad\text{and} \\
	\dist(\dif{Eu}, \mathcal{A}) &\leq \abs{\dif{Eu}} + \dist(0, \mathcal{A}) \leq \abs{\dif{Eu}} + \fint_{Q_k^*} \dist(\dif{u}, \mathcal{A}) + \fint_{Q_k^*} \abs{\dif{u}},
\end{align*}
by estimating $\dif{P_{Q_k^*}[u]}$ using $\dif{P_{Q_k^*}[u]} = (\dif{P_{Q_k^*}[u]} - \dif{u}) + \dif{u}$ in $Q_k^*$.
\end{proof}

\begin{proof}[Proof of \cref{Thm:Extenion_operator}]
Let $u \in \SobW^{1,1}(U, \R^d)$. Recall formula \eqref{Eq:Formula_extension}, which states
\begin{equation*}
	Eu(x) := \begin{cases} 
		u(x) & \text{if } x \in U, \\ 
		\sum_{Q_j \in W_3} P_{Q_j^*}(u)(x) \varphi_j(x) & \text{if } x \in \operatorname{int}(U^c).
	\end{cases}
\end{equation*}
Note that $(\set{0},p,p)$-regularity is trivially satisfied from the observation
\begin{equation*}
	\norm{\dif{u} - 0}_{\Leb^p(Q)} = \norm{\dif{u}}_{\Leb^p(Q)} = \norm{\dist(\dif{u}, \set{0})}_{\Leb^p(Q)},
\end{equation*}
for any measurable set $Q \subset \R^d$ and any $p \in [1,\infty]$. Therefore, \cref{Lem:Estimate_W3,Lem:Estimate_W2_wo_W3}, applied for $\mathcal{A} = \set{0}$ and $r=p=q=1$ and $r=p=q=\infty$ respectively, yield the estimates
\begin{align*}
	\norm{Ev}_{\SobW^{1,1}(\operatorname{int}(\R^d \setminus U))} &\leq c_1 \rho^{-1} \norm{v}_{\SobW^{1,1}(U)}, && \text{for all } v \in \SobW^{1,1}(U, \R^d), \tag{$*_1$}\\
	\norm{Ev}_{\SobW^{1,\infty}(\operatorname{int}(\R^d \setminus U))} &\leq c_1 \rho^{-1} \norm{v}_{\SobW^{1,\infty}(U)}, && \text{for all } v \in \SobW^{1,\infty}(U, \R^d). \tag{$*_2$}
\end{align*}
Indeed, $(*_1)$ and $(*_2)$ follow by summing over the cubes $Q_j \in W_2$, where we note that in view of the controlled size and distances of cubes and their reflected cubes by \cref{Def:WhitneyDecomposition,Lem:Reflection},
\begin{equation*}
	\sum_{Q_j \in W_3} \mathds{1}_{Q_j^*} \leq c_2 \mathds{1}_U, \quad 
	\sum_{Q_j \in W_3} \mathds{1}_{\bigcup \mathcal{F}(Q_j)} \leq c_2 \mathds{1}_U, \quad
	\sum_{Q_j \in W_2 \setminus W_3} \mathds{1}_{\bigcup \mathcal{F}(Q_j)} \leq c_2 \mathds{1}_{U^-_{\rho\alpha''}}.
\end{equation*}
We have to show that the weak derivative of $Eu$ exists and thus, indeed, $Eu$ belongs to $\SobW^{1,1}(\R^d, \R^d)$. This can be seen as follows. By the approximation result of Jones in \cite[Sec.~4]{Jon81}, there exists a sequence $(u_k) \subset \Cont^\infty(\R^d, \R^d)$, which converges to $u$ in $\SobW^{1,1}(U, \R^d)$. It follows analogously to \cite[Lem.~3.5]{Jon81} from $(*_2)$ that $Eu_k \in \SobW^{1,\infty}(\R^d, \R^d)$ is a Lipschitz map. Applying $(*_1)$ yields that $(E u_k)$ defines a Cauchy sequence in $\SobW^{1,1}(\R^d, \R^d)$ and thus, converges to $Eu$ in $\SobW^{1,1}(\R^d, \R^d)$. Especially $Eu \in \SobW^{1,1}(\R^d, \R^d)$ holds. Finally, we obtain the estimates in \cref{Thm:Extenion_operator} from \cref{Lem:Estimate_W3,Lem:Estimate_W2_wo_W3} by summing over the cubes in $W_2$.
\end{proof}

\paragraph{Geometric rigidity estimate and Korn's inequality in Jones domains.}
In the remainder of this section we utilize the extension operator to obtain \cref{Thm:RigidityEstimateJonesDomain,Cor:KornWithLpNormOfu,Cor:KornWithBoundaryValue,Cor:KornPeriodic}.

\begin{proof}[Proof of \cref{Thm:RigidityEstimateJonesDomain}]
Recall the definitions $r := \operatorname{diam(U)}$ and $\rho := \min\set{\frac{r}{2},\delta}$. We only consider the the case of the rigidity estimate, since the argument for Korn's inequality is similar in view of $\abs{\sym F} = \dist(F, \R^{d\times d}_{\skewMat})$.
\medskip

\stepemph{Step 0 -- Rigidity on finite unions of cubes}: We cover $U$ by $m$ cubes $Q_1, \dots, Q_m \subset \R^d$ of size $l(Q_i) \sim \rho$. We can do so with $m \lesssim (\tfrac{r}{\rho})^d$ many cubes. To be precise we ask for the following properties from the covering:
\begin{align*}
	&U \subset \bigcup_{i=1}^m Q_i \subset \bigcup_{i=1}^m \tfrac{33}{32}Q_i \subset U^+_{\rho\alpha'}, &
	&m \leq c_1 (\tfrac{r}{\rho})^d, &
	&\frac{1}{c_1}\rho \leq l(Q_i) \leq c_1 \rho, & 
	&\sum_{i=1}^m \mathds{1}_{Q_i} \leq c_1, \text{ and} \\
	&\mathrlap{\text{for each } i=1,\dots,m, \text{ there exists } j < i \text{ with } \abs{Q_i \cap Q_j} \geq \tfrac{1}{2}\abs{Q_i}.}
\end{align*}
Such a covering can be constructed using $d + 1$ shifted copies of the lattice 
\begin{equation*}
	\class{Q_k = \tfrac{\alpha'\rho}{2\sqrt{d}}(k + [0,1)^d)}{k \in \Z^d, Q_k\cap U \neq \emptyset}.
\end{equation*}
Recall $(\mathcal{A}, p, q)$-rigidity, cf.\ \cref{Def:GeneralRigidity}. Let us first show that $(\SO(d),p,q)$-rigidity holds on $\bigcup_{i=1}^m Q_i$ w.r.t.\ $Q^+ = \bigcup_{i=1}^m \frac{33}{32} Q_i$ with a constant $c m$ where $c=c(d,p,q) > 0$. Consider $v \in \SobW^{1,1}(\bigcup_{i=1}^m Q_i)$ and a decomposition $\dist(\dif{v}, \SO(d)) = F_{\dist(\dif{v}, \SO(d))} + G_{\dist(\dif{v}, \SO(d))}$ in $\Leb^p+\Leb^q(\bigcup_{i=1}^m Q_i)$. Due to \cite[Prop.~3.4]{FJM02} and \cite[Thm.~1.1]{CDM14} $(\SO(d),p,q)$-rigidity holds on cubes. Thus, we find $R_i \in \SO(d)$ and decompositions $\dif{v} - R_i = F_{\dif{v} - R_i} + G_{\dif{v} - R_i}$ in $\Leb^p+\Leb^q(Q_i, \R^{d\times d})$, $i=1,\dots,m$, satisfying
\begin{align*}
	\norm{F_{\dif{v} - R_i}}_{\Leb^p(Q_i)} &\leq c_2 \norm{F_{\dist(\dif{v}, \SO(d))}}_{\Leb^p(\tfrac{33}{32}Q_i)}, \\
	\norm{G_{\dif{v} - R_i}}_{\Leb^q(Q_i)} &\leq c_2 \norm{G_{\dist(\dif{v},\SO(d))}}_{\Leb^q(\tfrac{33}{32}Q_i)}.
\end{align*}
Note that by scaling and translation invariance of \cref{Def:GeneralRigidity}, we can choose the constant $c_2$ independent of $i$. We show that we can replace $R_i$ by $R_1$ by estimating the difference. For each $i$, we find a chain $Q_1 = Q_{i_1}, \dots Q_{i_k} = Q_i$ with $k \leq m$ and $\big|Q_{i_j} \cap Q_{i_{j+1}}\big| \geq c_3 \min \set{\big|Q_{i_j}\big|, \big|Q_{i_{j+1}}\big|}$. Thus,
\begin{align*}
	\abs{R_i - R_1} &\leq \sum_{j=1}^{k-1} \big|R_{i_{j+1}} - R_{i_j}\big| 
	\leq \sum_{j=1}^{k-1} \fint_{Q_{i_j} \cap Q_{\mathrlap{i_{j+1}}}} \big|\dif{v} - R_{i_j}\big| + \big|\dif{v} - R_{i_{j+1}}\big| \\
	&\leq \sum_{j=1}^{k-1} \fint_{Q_{i_j} \cap Q_{\mathrlap{i_{j+1}}}} \big|F_{\dif{v} - R_{i_j}}\big| + \big|G_{\dif{v} - R_{i_j}}\big| + \big|F_{\dif{Eu} - R_{i_{j+1}}}\big| + \big|G_{\dif{v} - R_{i_{j+1}}}\big|.
\end{align*}
We estimate the right-hand side as follows,
\begin{equation*}
	\norm{\fint_{Q_{i_j} \cap Q_{\mathrlap{i_{j+1}}}} \big|F_{\dif{v} - R_{i_j}}\big|}_{\Leb^p(Q_i)}
	\leq \tfrac{\abs{Q_i}^{1/p}}{\abs{Q_{i_j} \cap Q_{i_{j+1}}}^{1/p}} \big\Vert F_{\dif{v} - R_{i_j}}\big\Vert_{\Leb^p(Q_{i_j} \cap Q_{i_{j+1}})} \leq c_4 \norm{F_{\dist(\dif{v}, \SO(d))}}_{\Leb^p(\tfrac{33}{32}Q_{i_j})},
\end{equation*}
and obtain analogous results for the other terms. Note that the constant $c_4$ only depends on $d$ and $p$ (respectively on $q$). Hence, \cref{Rem:Properties_rigidity} (i) shows that $R_i - R_1$ admits a decomposition $R_i - R_1 = F_{R_i - R_1} + G_{R_i - R_1}$ in $\Leb^p+\Leb^q(\Omega, \R^{d\times d})$ with
\begin{align*}
	\norm{F_{R_i - R_1}}_{\Leb^p(Q_i)} &\leq k c_4 \norm{F_{\dist(\dif{v}, \SO(d))}}_{\Leb^p(\tfrac{33}{32}Q_i)}, \\
	\norm{G_{R_i - R_1}}_{\Leb^q(Q_i)} &\leq k c_4 \norm{G_{\dist(\dif{v},\SO(d))}}_{\Leb^q(\tfrac{33}{32}Q_i)}.
\end{align*}
Set $A_i := Q_i \setminus \bigcup_{j=i}^{i-1} Q_j$. Thus, choosing $R := R_1$, $F_{\dif{v} - R} := \sum_{i=1}^{m} (F_{\dif{v} - R_i} + F_{R_i - R_1}) \mathds{1}_{A_i}$ and $G_{\dif{v} - R} := \sum_{i=1}^{m} (G_{\dif{v} - R_i} + G_{R_i - R_1}) \mathds{1}_{A_i}$, we obtain $\dif{v} - R = F_{\dif{v} - R} + G_{\dif{v} - R}$ a.e.\ in $\bigcup_{i=1}^m Q_i$ and
\begin{align*}
	\norm{F_{\dif{v} - R}}_{\Leb^p(\bigcup_{i=1}^m Q_i)} &\leq m c_5 \norm{F_{\dist(\dif{v}, \SO(d))}}_{\Leb^p(\bigcup_{i=1}^m\frac{33}{32}Q_i)}, \\
	\norm{G_{\dif{v} - R}}_{\Leb^q(\bigcup_{i=1}^m Q_i)} &\leq m c_5 \norm{G_{\dist(\dif{v},\SO(d))}}_{\Leb^q(\bigcup_{i=1}^m\frac{33}{32}Q_i)}.
\end{align*}

\stepemph{Step 1 -- $A \equiv I$}: Let us first restrict to the case $A \equiv I$. Given $u \in \SobW^{1,1}(U, \R^d)$ and a decomposition $\dist(\dif{u}, \SO(d)) = F_{\dist(\dif{u}, \SO(d))} + G_{\dist(\dif{u}, \SO(d))}$ in $\Leb^p+\Leb^q(U)$, by \cref{Thm:Extenion_operator} we find a decomposition $\dist(\dif{Eu}, \SO(d)) = F_{\dist(\dif{Eu}, \SO(d))} + G_{\dist(\dif{Eu}, \SO(d))}$ in $\Leb^p+\Leb^q(\bigcup_{i=1}^m Q_i)$, such that
\begin{align*}
	\norm{F_{\dist(\dif{Eu}, \SO(d))}}_{\Leb^p(\bigcup_{i=1}^m Q_i)} &\leq c_6 \norm{F_{\dist(\dif{u}, \SO(d))}}_{\Leb^p(U)}, \\
	\norm{G_{\dist(\dif{Eu}, \SO(d))}}_{\Leb^q(\bigcup_{i=1}^m Q_i)} &\leq c_6 \norm{G_{\dist(\dif{u}, \SO(d))}}_{\Leb^q(U)}.
\end{align*}
Then, by Step 0, we find $R \in \SO(d)$ and a decomposition $\dif{Eu} - R = F_{\dif{Eu} - R} + G_{\dif{Eu} - R}$ in $\Leb^p+\Leb^q(\bigcup_{i=1}^m Q_i, \R^{d\times d})$, such that
\begin{align*}
	\norm{F_{\dif{Eu} - R}}_{\Leb^p(\bigcup_{i=1}^m Q_i)} &\leq c_7 (\tfrac{r}{\rho})^d \norm{F_{\dist(\dif{u}, \SO(d))}}_{\Leb^p(U)}, \\
	\norm{G_{\dif{Eu} - R}}_{\Leb^q(\bigcup_{i=1}^m Q_i)} &\leq c_7 (\tfrac{r}{\rho})^d \norm{G_{\dist(\dif{u}, \SO(d))}}_{\Leb^q(U)}.
\end{align*}
The claim follows by combining the two results.
\medskip

\stepemph{Step 2 -- Bilipschitz potentials}:
Now, if $A$ admits a Bilipschitz potential $a$ \cref{Thm:RigidityEstimateJonesDomain} is an easy consequence of the arguments above for $\tilde{u} := u \circ a^{-1}$ on $\tilde{U} = a(U)$ and using the transformation rule. This is possible, since $\tilde{U}$ is a Jones domain and $\tilde{u} \in \SobW^{1,1}(\tilde{U}, \R^d)$. Note that $e$, $\delta$ and $\tfrac{r}{\rho}$ are controlled when transforming the domain by a map in $\operatorname{Bil}_L(U,\R^d)$.
\medskip

\stepemph{Step 3 -- Arbitrary stress-free joints}:
Now, let us consider an arbitrary stress-free joint $A \in \operatorname{SFJ}(U)$ with potential $a$. By \ref{Property:SFJisPiecewiseBilipschitz}, we find disjoint Lipschitz domains $U_1,\dots,U_n$ with $\overline{U} = \bigcup_{i=1}^n \overline{U_i}$, such that $a$ is Bilipschitz on $U_i$ for all $i=1,\dots,n$. Thus, we may apply Step 2 on $U_i$ and obtain $R_i \in \SO(d)$ and decompositions $\dif{u}A(\placeholder)^{-1} -R_i = F^i_{\dif{u}A(\placeholder)^{-1} - R_i} + G^i_{\dif{u}A(\placeholder)^{-1} - R_i}$ in $\Leb^p+\Leb^q(U_i, \R^{d\times d})$, such that
\begin{align*}
	\norm{F^i_{\dif{u}A(\placeholder)^{-1} - R_i}}_{\Leb^p(U_i)} &\leq c_8 \norm{F_{\dist(\dif{u}A(\placeholder)^{-1}, \SO(d))}}_{\Leb^p(U_i)}, \\
	\norm{G^i_{\dif{u}A(\placeholder)^{-1} - R_i}}_{\Leb^q(U_i)} &\leq c_8 \norm{G_{\dist(\dif{u}A(\placeholder)^{-1}, \SO(d))}}_{\Leb^q(U_i)}.
\end{align*}
It remains to show that we can choose the same rotation for each $i=1,\dots,n$. Indeed, we can do so, by estimating the difference between the rotations of neighboring domains. Therefore, let $i\neq j$ with $\mathcal{H}^{d-1}(\partial U_i \cap \partial U_j) > 0$. Set $\Gamma_{ij} := \partial U_i \cap \partial U_j$ and define $\tilde{u}_i := u \circ a^{-1} \in \SobW^{1,p}(a(U_i),\R^d)$ and $\tilde{u}_j := u \circ a^{-1} \in \SobW^{1,p}(a(U_j),\R^d)$. Since $\tilde{u}_i = \tilde{u}_j$ on $a(\Gamma_{ij})$, we obtain by \cref{Cor:MatrixNorm}, continuity of the trace operator, the Poincaré-Wirtinger inequality and the transformation rule, for suitable $\xi_i, \xi_j \in \R^d$,
\begin{align*}
	\abs{R_i - R_j}^p &= \abs{R_j R_i^T - I}^p 
	\leq c_9 \int_{a(\Gamma_{ij})} \abs{R_i z - R_j z - \xi_i + \xi_j}^p \,\dx[\mathcal{H}^{d-1}(z)] \\
	&= c_9 \int_{a(\Gamma_{ij})} \abs{(R_i z - \tilde{u}(z) - \xi_i) - (R_j z - \tilde{u}(z) - \xi_j)}^p \,\dx[\mathcal{H}^{d-1}(z)], \\
	&\leq c_{10} \left(\int_{a(U_i)} \abs{R_i - \dif{\tilde{u}}}^p + \int_{a(U_j)} \abs{R_j - \dif{\tilde{u}}}^p\right) \\
	&\leq c_{11} \left(\int_{U_i} \abs{R_i - \dif{u}\dif{a}(\placeholder)^{-1}}^p + \int_{U_j} \abs{R_j - \dif{u}\dif{a}(\placeholder)^{-1}}^p\right) \\
	&\leq c_{12} \int_{U_i \cup U_j} \dist^p\big(\dif{u}\dif{a}(\placeholder)^{-1}, \SO(d)\big).
\end{align*}
Since $\dist(\dif{u}\dif{a}(\placeholder)^{-1}, \SO(d)) = F_{\dist(\dif{u}A(\placeholder)^{-1}, \SO(d))} + G_{\dist(\dif{u}A(\placeholder)^{-1}, \SO(d))}$, we obtain
\begin{equation*}
	\abs{R_i - R_j} \leq c_{13} \left(\big\Vert F_{\dist(\dif{u}A(\placeholder)^{-1}, \SO(d))}\big\Vert_{\Leb^p(U)} + \big\Vert G_{\dist(\dif{u}A(\placeholder)^{-1}, \SO(d))}\big\Vert_{\Leb^q(U)} \right).
\end{equation*}
Hence, since $U$ is connected, by an induction argument it follows that we may choose $R_i := R_1$ for all $i=1,\dots,n$, which finishes the proof. Note that $R_1$ is chosen arbitrarily among the $R_i$.
\end{proof}

\begin{proof}[Proof of \cref{Cor:KornWithLpNormOfu,Cor:KornWithBoundaryValue}.]
Let $u \in \SobW^{1,p}(U, \R^d)$. By \cref{Thm:RigidityEstimateJonesDomain}, we find $S \in \R^{d\times d}_{\skewMat}$, such that
\begin{equation*}
	\norm{\dif{u}}_{\Leb^p(U)} 
	\leq c_1 \left( \norm{\dif{u}A(\placeholder)^{-1} - S}_{\Leb^p(U)} + \abs{S} \right) 
	\leq c_2 \left( \norm{\sym(\dif{u}A(\placeholder)^{-1})}_{\Leb^p(U)} + \abs{S} \right).
\end{equation*}
Hence, it remains to bound $\abs{S}$ by the right-hand sides of \eqref{Eq:KornWithBoundaryValue} and \eqref{Eq:KornWithLpNormOfu} respectively. By \ref{Property:SFJisPiecewiseBilipschitz} we find a Lipschitz domain $\tilde{U} \subset U$ such that the potential $a$ of $A$ is Bilipschitz on $\tilde{U}$. In the case of \cref{Cor:KornWithBoundaryValue}, we can choose $\tilde{U}$ such that $\tilde{\Gamma} := \Gamma \cap \partial\tilde{U}$ satisfies $\mathcal{H}^{d-1}(\tilde{\Gamma}) > 0$. Consider $\tilde{u} := u \circ a^{-1} \in \SobW^{1,p}(a(\tilde{U}), \R^d)$. Then,
\begin{equation*}
	\dif{\tilde{u}}(a(x)) = \dif{u}(x)\dif{a}(x)^{-1} \qquad \text{for a.e.\ } x \in \tilde{U}.
\end{equation*}
Let $\xi := \fint_{a(\tilde{U})} Sz - \tilde{u}(z) \,\dx[z]$. For \cref{Cor:KornWithLpNormOfu}, we estimate the modulus of $S$ using \cref{Cor:MatrixNormOmega}, the change of variables rule and the Poincaré-Wirtinger inequality,
\begin{alignat*}{3}
	\abs{S}^p &\leq c_3 \int_{a(\tilde{U})} \abs{Sz - \xi}^p \,\dx[z] \leq \mathrlap{c_4 \left(\int_{a(\tilde{U})} \abs{Sz - \tilde{u}(z) - \xi}^p \,\dx[z] + \norm{\tilde{u}}_{\Leb^p(a(\tilde{U}))}^p\right)} \\
	&\leq c_5\left( \int_{a(\tilde{U})} \abs{S - \dif{\tilde{u}}(z)}^p \,\dx[z] + \norm{\tilde{u}}_{\Leb^p(a(\tilde{U}))}^p\right) 
	&&\leq c_6\left( \int_{\tilde{U}} \abs{S - \dif{u}(x)\dif{a}(x)^{-1}}^p \,\dx + \norm{u}_{\Leb^p(\tilde{U})}^p\right) \\ 
	&&&\leq c_7 \left(\norm{\sym\left( \dif{u}\dif{a}(\placeholder)^{-1}\right)}_{\Leb^p(U)}^p + \norm{u}_{\Leb^p(\tilde{U})}^p\right).
\end{alignat*}
Similarly if $u \in \SobW^{1,p}_{\Gamma,0}(U,\R^d)$, i.e.\ $\tilde{u} = 0$ on $a(\tilde{\Gamma})$, we estimate using \cref{Cor:MatrixNorm},
\begin{align*}
	\abs{S}^p &\leq c_8 \int_{a(\tilde{\Gamma})} \abs{Sz - \xi}^p \,\dx[z] = c_6 \int_{a(\tilde{\Gamma})} \abs{Sz - \tilde{u}(z) - \xi}^p \,\dx[z]
	\leq c_9 \int_{a(\tilde{U})} \abs{S - \dif{\tilde{u}}(z)}^p \,\dx[z] \\
	&\leq c_{10} \int_{\tilde{U}} \abs{S - \dif{u}(x)\dif{a}(x)^{-1}}^p \,\dx 
	\leq c_{11} \norm{\sym\left( \dif{u}\dif{a}(\placeholder)^{-1}\right)}_{\Leb^p(U)}^p.
\end{align*}
Hence, \cref{Cor:KornWithBoundaryValue} follows by applying the Poincaré-Friedrich inequality. Note that, if $a$ is Bilipschitz on $U$, we can choose $\tilde{U} = U$ and $\tilde{\Gamma} = \Gamma$ and then the constants can be chosen uniformly for potentials with controlled Bilipschitz constant.
\end{proof}

\begin{proof}[Proof of \cref{Cor:KornPeriodic}]
We argue similarly to \cite[Thm 2.3]{Pom03}. Let $\varphi \in \SobW^{1,p}(\Omega, \R^{d\times d})$ and let $a$ denote the Bilipschitz potential of $A$ with $a(0)=0$, cf.\ \cref{Prop:SFJBilipschitz}.
\medskip

\stepemph{Step 1}: First, we show that $\sym(\dif{\varphi}\dif{a}(\placeholder)^{-1}) = 0$ implies $\varphi = 0$. This can be carried out as follows. Let $\tilde{\varphi} := \varphi \circ a^{-1}$. If $\varphi$ satisfies $\sym(\dif{\varphi}\dif{a}(\placeholder)^{-1}) = 0$, then $\tilde{\varphi}$ satisfies $\sym \dif{\tilde{\varphi}} = 0$. As stated in the proof of \cite[Thm 6.3-4]{Cia88}, then $\tilde{\varphi}_i$ must be an affine map, $i=1,\dots,d$. From the periodicity of $\varphi$ and $\dif{a}$, we obtain (cf.~\cref{Lem:periodic_derivative})
\begin{equation*}
	\tilde{\varphi}_i(k a(e_j)) = \tilde{\varphi}_i(a(k e_j)) = \varphi_i(k e_j) = \varphi_i(e_j) = \tilde{\varphi}_i(a(e_j)), \quad k \in \Z.
\end{equation*}
Since the vectors $a(e_j) = a(0) + \bar{A}e_j = \bar{A}e_j$, $j=1,\dots,d$ span the matrix $\bar{A} := \fint_Y \dif{a}(y)\,\dx[y]$ with $\det \bar{A} \neq 0$, see \cref{Lem:periodic_derivative}, we infer that $\tilde{\varphi}$ and thus also $\varphi$ are in fact constant. Finally, because $\int_Y \varphi = 0$, $\varphi = 0$.
\medskip

\stepemph{Step 2}: We show that Step 1 implies the claim. Suppose the claim does not hold. Then we find a sequence $(\varphi_k) \subset \SobW^{1,p}_{\mathrm{per},0}(Y,\R^d)$ with $\norm{\varphi_k}_{\SobW^{1,p}(Y)} = 1$ and
\begin{equation*}
	\norm{\sym (\dif{\varphi_k}\dif{a}(\placeholder)^{-1})}_{\Leb^p(Y)} \leq \tfrac{1}{k}.
\end{equation*}
Since $(\varphi_k)$ is bounded in $\SobW^{1,p}(Y,\R^d)$, it is not restrictive to assume that $\varphi_k \wto \varphi \in \SobW^{1,p}_{\mathrm{per},0}(Y,\R^d)$ weakly in $\SobW^{1,p}(Y,\R^d)$. \cref{Cor:KornWithLpNormOfu} yields
\begin{equation*}
	\norm{\varphi_k - \varphi_l}_{\SobW^{1,p}(Y)} \leq c_1 \big(\norm{\varphi_k - \varphi_l}_{\Leb^p(Y)} + \norm{\sym (\dif{\varphi_k}\dif{a}(\placeholder)^{-1})}_{\Leb^p(Y)}
	+ \norm{\sym (\dif{\varphi_l}\dif{a}(\placeholder)^{-1})}_{\Leb^p(Y)} \big).
\end{equation*}
Hence, $(\varphi_k)$ is a Cauchy sequence in $\SobW^{1,p}(Y,\R^d)$ and thus actually converges strongly to $\varphi$. By continuity we obtain $\sym (\dif{\varphi}\dif{a}(\placeholder)^{-1}) = 0$ and thus by Step 1, $\varphi = 0$. But this is a contradiction to $\norm{\varphi}_{\SobW^{1,p}(Y)} = 1$.
\end{proof}

\subsection{Proofs of Lemmas \ref{Lem:homandprestrain} and \ref{L:coercivity} and introduction of auxiliary integrands \texorpdfstring{$\widetilde{W}^h_{\hom}$ and $\widetilde{Q}_{\hom}$}{}}
\label{Sec:Proof:AsyExpansionWhom1}

For the proofs of \cref{Thm:Expansion_Whom,Cor:ExpansionMinimizer,Prop:Bhom_algorithmic} it is natural to work with the transformed energy densities
\begin{align}
	\widetilde{W}^h_{\hom}(F) &:= \inf_{k \in \N} \inf_{\varphi \in \SobW^{1,\infty}_\mathrm{per}(kY, \R^d)} \fint_{kY} W\left(y, (F + \dif{\varphi}(y)A(y)^{-1})(I-hB_h(y))\right) \,\dx[y], \\
	\widetilde{Q}_{\hom}(G) &:= \min_{\varphi \in \SobH^1_\mathrm{per}(Y, \R^d)} \int_Y Q\left(y, G + \dif{\varphi}(y)A(y)^{-1}\right) \,\dx[y].
\end{align}
The relation to $[W^h]_{\hom}$ and $[Q^A]_{\hom}$ is the following: For all $F,G\in\R^{d\times d}$ it holds
\begin{equation}\label{eq:st1111}
  \begin{aligned}
  \widetilde{W}^h_{\hom}(F) &= \Big[(y,F) \mapsto W^h(y,FA(y))\Big]_{\hom}(F) = [W^h]_{\hom}(F\bar{A}), \\
  \widetilde{Q}_{\hom}(G) &= \Big[(y,G) \mapsto Q^A(y,GA(y))\Big]_{\hom}(G) = [Q^A]_{\hom}(G\bar{A}).
\end{aligned}
\end{equation}
Indeed, this can be easily seen by realizing that for $A\in\operatorname{SFJ}_{\rm per}$ with potential $a$, and for any $F\in\R^{d\times d}$, Lemma~\ref{Lem:periodic_derivative} implies that
\begin{equation}\label{eq:fromAtobarA}
  \varphi_F = F(\bar{A}^{-1}\placeholder - a^{-1}) \circ a \in \SobW^{1,\infty}_\mathrm{per}(Y, \R^d),\qquad FA(\placeholder)^{-1}=F\bar A^{-1}-D\varphi_FA(\placeholder)^{-1}.
\end{equation}
With this observation at hand, the identities \eqref{eq:st1111} follow from the definition of $[\cdot ]_{\hom}$. Note that in the definition of $\widetilde{Q}_{\hom}$ it suffices to minimize w.r.t.\ a single periodicity cell and correctors in $H^1$. This is due to the fact that $Q$ satisfies a quadratic growth condition and is convex. With similar argumentation, we can prove \cref{Lem:homandprestrain}.

\begin{proof}[Proof of \cref{Lem:homandprestrain}]
Let $G \in \R^{d \times d}$ and $a$ denote the unique potential of $A$ with $a(0) = 0$. \cref{Lem:periodic_derivative} shows that the map $\varphi_A := \bar{A}\placeholder - a$ satisfies $\varphi_A \in \SobW^{1,\infty}_\mathrm{per}(Y, \R^d)$. Hence,
\begin{equation*}
	W^h(y, A(y) + G) = W^h(y, \bar{A} + G - \dif{\varphi_A}(y)).
\end{equation*}
From this \cref{Lem:homandprestrain} follows immediately, since $\dif{\varphi_A}(y)$ is not seen when applying $[\placeholder]_{\hom}$ to the right-hand side.
\end{proof}

We note that we shall establish most properties of $[W^h]_{\hom}$ and $[Q^A]_{\hom}$ by proving equivalent assertions for $\widetilde{W}^h_{\hom}$ and $\widetilde{Q}_{\hom}$. An example is the following:
\begin{proof}[Proof of \cref{L:coercivity}]
In view of \eqref{eq:st1111} and the definition of $\sym_{\bar A}$ it suffices to show 
\begin{equation*}
  \forall G\in\R^{d\times d}\,:\,\frac{1}{c_1}\abs{(\sym G)\bar A}^2\leq \widetilde{Q}_{\hom}(G)\leq c_1\abs{(\sym G)\bar A}^2.
\end{equation*}
For the upper bound note that by \eqref{Eq:UniformBounds_Q} and the invertibility of $\bar A$, we have 
\begin{equation*}
  \widetilde{Q}_{\hom}(G)\leq \int_Y Q\big(y,G\big) \,\dx[y] \leq \beta_{\rm el}\abs{\sym G}^2
  \leq \beta_{\rm el} \abs{\bar{A}^{-1}}^2\,\abs{(\sym G)\bar A}^2.
\end{equation*}
For the lower bound, let $\varphi \in \SobH^1_\mathrm{per}(Y, \R^d)$ and $a$ the potential of $A$ with $a(0) = 0$. We use that $\varphi \circ a \in \SobH^1_\text{per}(\bar{A}Y, \R^d)$ and thus its derivative is orthogonal to constants in $\Leb^2(a(Y), \R^{d\times d})$, cf.\ \cref{Lem:periodic_derivative}. We observe
\begin{align*}
	\abs{\sym G}^2 &\leq \abs{\sym G}^2 + \fint_{a(Y)} \abs{\sym\dif{(\varphi \circ a)}}^2
	= \fint_{a(Y)} \abs{\sym(G + \dif{(\varphi \circ a)})}^2 \\
	&\leq \tfrac{\norm{\det A}_{\Leb^\infty}}{\abs{a(Y)}} \int_Y \abs{\sym (G + \dif{\varphi}A(\placeholder)^{-1})}^2
	\leq \tfrac{\norm{\det A}_{\Leb^\infty}}{\alpha_{\rm el}\det \bar A} \int_Y Q\big(y, G + \dif{\varphi}(y)A(y)^{-1}\big) \,\dx[y].
\end{align*}
Hence, we observe the lower bound by minimizing the inequality over $\varphi$.
\end{proof}

\subsection{Representation formulas for the homogenized energy and perturbation. Proofs of Lemmas \ref{lemma:homogenizationaftelinearization} and \ref{Thm:FormulaBhom} and Proposition~\ref{Prop:Bhom_algorithmic}}
\label{Sec:ProofHomogenPerturbation}
\begin{proof}[Proof of \cref{Thm:FormulaBhom}]  
We claim that $P_{\mathcal O}$ is onto and injective. Indeed, by definition we have $\mathcal O\subset(\mathcal S+\R^{d\times d}_{\sym})$ and $\mathcal{O}\subset \mathcal{S}^\perp$ and thus,
\begin{equation*}
  \mathcal{O} = \operatorname{P}_\mathcal{O}(\mathcal{S}+\R^{d\times d}_{\sym}) = \operatorname{P}_\mathcal{O}(\mathcal{S}) + \operatorname{P}_\mathcal{O}(\R^{d\times d}_{\sym}) = \operatorname{P}_\mathcal{O}(\R^{d\times d}_{\sym}),
\end{equation*}
which yields surjectivity. For injectivity we only need to show $P_{\mathcal O}(\sym G) = 0$ implies $\sym G = 0$. For the argument, first note that
\begin{equation}\label{eq:P123333}
\begin{aligned}
  \operatorname{P}_\mathcal{O}(\sym G) = 0 \quad&\Leftrightarrow\quad \operatorname{P}_{\mathcal{S}}(\sym G) = \sym G \\
  &\Leftrightarrow\quad \sym G \in \mathcal{S} \quad\Leftrightarrow\quad \exists \varphi \in \SobH^1_\mathrm{per}(Y,\R^d):\; \sym\left[ G - \dif{\varphi}A(\placeholder)^{-1}\right] = 0.
\end{aligned}
\end{equation}
We claim: If $\operatorname{P}_{\mathcal O}(\sym G) = 0$, then there exists $\varphi \in \SobH^1_\mathrm{per}(Y,\R^d)$ such that
\begin{equation*}
  \sym G = \sym \dif(\varphi\circ a^{-1})\qquad\text{a.e.\ in } \R^d,
\end{equation*}
where $a$ denotes the potential of $A$ as in \ref{Property:SFJisDerivative} with $a(0) = 0$. Indeed, this follows since by \eqref{eq:P123333} there exists $\varphi \in \SobH^1_\mathrm{per}(Y,\R^d)$ such that $\sym G = \sym (\dif{\varphi}A(\placeholder)^{-1})$, and by the chain rule we have $\dif{\varphi}(y)A^{-1}(y)=\dif\varphi(y)\dif{a^{-1}}(z)=\dif{(\varphi\circ a^{-1})}(z)$ for a.e.\ $z=a(y) \in \R^d$. Now the assertion $\sym G=0$ follows by integrating \eqref{eq:P123333} over $a(Y)$ and by appealing to \cref{Lem:periodic_derivative}.
\end{proof}
\begin{proof}[Proof of \cref{lemma:homogenizationaftelinearization}]
\stepemph{Step 1 -- Proof of \eqref{lemma:homogenizationaftelinearization:2}}: 
By definition of $W^h$ we have
\begin{equation*}
  \tfrac{1}{h^2}W^h(y,A(y)+hG)=\tfrac{1}{h^2}W\big(y,I+h(GA(y)^{-1}-B_h(y))-h^2GA(y)^{-1}B_h(y)\big).
\end{equation*}
Since up to a subsequence we have $B_h \to B$ a.e.\ and in view of \ref{Property:W_Expansion}, we deduce that the right-hand side converges to $Q(y,GA(y)^{-1}-B(y))=Q^A(y,G-B(y)A(y))$ as claimed.
\medskip

\stepemph{Step 2 -- Proof of \eqref{lemma:homogenizationaftelinearization:1}}:
Thanks to the definition of $Q^A$, $[\cdot]_{\hom}$, and $\mathcal S$, we have
\begin{align*}
	&\text{LHS of \eqref{lemma:homogenizationaftelinearization:1}} 
	= \inf_{\varphi \in \SobH^1_{\rm per}(Y, \R^d)}\int_Y Q\big(y,GA(y)^{-1} + \dif{\varphi}(y)A(y)^{-1} - B(y)\big)\,\dx[y]. \\
	&\qquad = \inf_{\Psi \in \mathcal{S}^\perp} \norm{GA(\placeholder)^{-1} + \Psi - B}_Q^2
	= \norm{\operatorname{P}_{\mathcal S^\perp}(\sym(GA(\placeholder)^{-1}-B))}_Q^2
	= \norm{\operatorname{P}_{\mathcal S^\perp}(\sym(G\bar{A}^{-1} - B))}_Q^2,
\end{align*}
where in the last identity we used that $\sym(GA(\placeholder)^{-1})-\sym(G\bar A^{-1})\in\mathcal S$, which follows from \eqref{eq:fromAtobarA}. Since $\mathcal S+\mathcal O+(\mathcal S+\R^{d\times d}_{\sym})^\perp$ is an orthogonal sum, we get
\begin{equation*}
  \text{LHS of \eqref{lemma:homogenizationaftelinearization:1}} 
  =\norm{\operatorname{P}_{\mathcal O}(\sym(G\bar{A}^{-1} - B))}_Q^2 + \norm{\operatorname{P}_{(\mathcal S + \R^{d\times d}_{\sym})^\perp}(\sym B)}_Q^2.
\end{equation*}
In view of the definition of $B_{\hom}$ and $R^A(B)$ we arrive at
\begin{equation}\label{eq:ss1123ss}
  \text{LHS of \eqref{lemma:homogenizationaftelinearization:1}} 
  =\norm{\operatorname{P}_{\mathcal O}(\sym(G\bar{A}^{-1}) - B_{\hom})}_Q^2 + R^A(B).
\end{equation}
Note that in the special case $B=0$, this identity reduces to  $Q^A_{\hom}(G)=\norm{\operatorname{P}_{\mathcal O}(\sym(G\bar A^{-1}))}_Q^2$. In particular, $\norm{\operatorname{P}_{\mathcal O}(\sym(G\bar A^{-1}) - B_{\hom})}_Q^2=Q^A_{\hom}(G-B_{\hom}\bar A)$ and \eqref{lemma:homogenizationaftelinearization:1} follows.
\end{proof}

We finish this section with the proof of the representation formulas \cref{Prop:Bhom_algorithmic}.

\begin{proof}[Proof of \cref{Prop:Bhom_algorithmic}]
Thanks to the periodic Korn inequality (cf.~\cref{Cor:KornPeriodic}) and the non-degeneracy of $Q$ (cf.~\eqref{Eq:UniformBounds_Q}), we see that \eqref{Eq:correctors} is a coercive, strictly convex minimization problem. We conclude that the corrector $\varphi_G$ indeed exists and is unique. Furthermore, since the associated Euler-Lagrange equation is linear, we deduce that
\begin{equation*}
  \widetilde{Q}_{\hom}(G)=\xi\cdot\Qhom\;\xi,\qquad \xi=\operatorname{emb}^{-1}(G),\qquad\text{for all } G\in\R^{d\times d}_{\sym}.
\end{equation*}
Combined with \eqref{eq:st1111} the second identity in \eqref{Eq:char_Bhom} follows.
It remains to prove the representation of $B_{\hom}$. To this end, we first note that
\begin{equation*}
  \operatorname{P}_{\mathcal O}(\sym G) =\sym G + \sym(\dif{\varphi_G}A(y)^{-1}).
\end{equation*}
Indeed, this follows from  $\sym G\in \mathcal S+\mathcal O$ and the definition of $\varphi_G$. We also conclude that
\begin{equation*}
  b_i = \left( P_{\mathcal O}(G_i), \sym B\right)_Q.
\end{equation*}
Thanks to the surjectivity of $P_{\mathcal O}$ we have $\mathcal O=P_{\mathcal O}(\operatorname{emb}(\R^s))$, and thus, we obtain the following characterization of $B_{\hom}$:
\begin{align*}
&\operatorname{P}_{\mathcal O}(B_{\hom}) = \operatorname{P}_\mathcal{O}(\sym B) \quad\Leftrightarrow\quad\forall\,\chi \in \mathcal{O}\,:\,\left(\chi, \operatorname{P}_{\mathcal O}(B_{\hom})\right)_Q = \left(\chi, \sym B\right)_Q\\
&\quad\Leftrightarrow\quad  \forall\,\xi \in \R^s\,:\,\left(\operatorname{P}_{\mathcal O}(\operatorname{emb}(\xi)), \operatorname{P}_{\mathcal O}(B_{\hom})\right)_Q = \left(\operatorname{P}_{\mathcal O}(\operatorname{emb}(\xi)), \sym B \right)_Q \\
&\quad\Leftrightarrow\quad \forall\,\xi \in \R^s\,:\,\xi \cdot \Qhom\; \operatorname{emb}^{-1}(B_{\hom}) = \xi \cdot b, 
\end{align*}
and thus $B_{\hom} = \operatorname{emb}\big((\Qhom)^{-1}b\big)$ as claimed.
\end{proof}

\subsection{Asymptotic expansion of the homogenized elastic energy density; Proofs of Theorem~\ref{Thm:Expansion_Whom} and Corollary~\ref{Cor:ExpansionMinimizer}}
\label{Sec:Proof:AsyExpansionWhom2}

We note that \cref{Thm:Expansion_Whom} (a) easily follows from frame-indifference of $W$ by substituting $\varphi$ with $R\varphi$ in the definition of $W^h_{\hom}$. 
In view of \eqref{eq:st1111}, for \cref{Thm:Expansion_Whom} (b) and (c) it suffices to prove the following two propositions:

\begin{proposition}[Non-degeneracy] \label{Prop:NonDegeneracy}
There exists some $\alpha > 0$ and $h_0 > 0$ such that for all $F \in \R^{d \times d}$ and $0 < h \leq h_0$
\begin{equation}
	\widetilde{W}^h_{\hom}(F) \geq \tfrac{1}{\alpha} \dist^2(F, \SO(d)) - \alpha h^2 \norm{B_h}_{\Leb^2(Y)}^2.
\end{equation}
\end{proposition}

\begin{proposition}[Asymptotic expansion] \label{Prop:AsyExpansion}
There exists a continuous, increasing map $\rho: [0,\infty) \to [0,\infty]$ with $\rho(0) = 0$, such that for all $h > 0$ and $G \in \R^{d \times d}$,
\begin{equation}
	\abs{\tfrac{1}{h^2} \widetilde{W}^h_{\hom}(I + hG) - \left(\widetilde{Q}_{\hom}(G - B_{\hom}) + R^A(B)\right)} \leq (1 + \abs{G}^2) \rho(h + \abs{hG}).
\end{equation}
\end{proposition}

We start with the argument for the non-degeneracy property:

\begin{proof}[Proof of \cref{Prop:NonDegeneracy}]
We adapt the argument of \cite[Thm. 1.1]{MN11}. By definition of $\widetilde{W}^h_{\hom}$ we find for all $\eta > 0$ some $k_\eta \in \N$ and $\varphi_\eta \in \SobW^{1,\infty}_\mathrm{per}(k_\eta Y, \R^d)$ such that
\begin{align*}
	\widetilde{W}^h_{\hom}(F) &\geq  \fint_{k_\eta Y} W\Big(y, (F + \dif{\varphi_\eta}A(y)^{-1})(I - h B_h)\Big) \,\dx[y] - \eta \\
	&\geq \alpha_\mathrm{el} \fint_{k_\eta Y} \dist^2\Big((F + \dif{\varphi_\eta}A(y)^{-1})(I - h B_h), \SO(d)\Big) \,\dx[y] - \eta.
\end{align*}
Here and throughout this section we often omit the explicit dependence on the argument of certain quantities when integrating for easier reading. However, we do display the dependence for $A(\placeholder)^{-1}$ to distinguish between the function and matrix inverse.
Furthermore, the triangle inequality yields for all $\tilde{F} \in \R^{d\times d}$,
\begin{align*}
	\dist(\tilde{F}, \SO(d)) &\leq \dist(\tilde{F}(I - hB_h(y)), \SO(d)) + h\abs{B_h(y)}\abs{\tilde{F}} \\
	&\leq \dist(\tilde{F}(I - hB_h(y)), \SO(d)) + h\abs{B_h(y)}(\dist(\tilde{F}, \SO(d)) + \abs{I}).
\end{align*}
Since $\limsup_{h \to 0} h\norm{B_h}_{\Leb^\infty(\R^d)} \to 0$, we can absorb for small $h \ll 1$ the term $\dist(\tilde{F}, \SO(d))$ into the left-hand side and obtain
\begin{equation}
	\dist^2(\tilde{F}(I - hB_h(y)), \SO(d)) \geq \tfrac{1}{c_1} \dist^2(\tilde{F}, \SO(d)) - c_1 h^2 \abs{B_h(y)}^2.
\end{equation}
Thus, with $\tilde{F} = F + \dif{\varphi_\eta}(y)A(y)^{-1}$ and thanks to $Y$-periodicity of $B_h$, we get
\begin{equation*}
	\widetilde{W}^h_{\hom}(F) \geq \tfrac{1}{c_2} \fint_{k_\eta Y} \dist^2\left(F + \dif{\varphi_\eta}A(y)^{-1}, \SO(d)\right) \,\dx[y] - c_2h^2\norm{B_h}_{\Leb^2(Y)}^2 -\eta.
\end{equation*}
Note that $\norm{B_h}_{\Leb^2(Y)}$ is bounded. Hence, by the transformation rule and $A = \dif{a}$ with $\frac{1}{c} \leq \det \dif{a} \leq c$ a.e., we obtain
\begin{equation*}
	\widetilde{W}^h_{\hom}(F)
	\geq \tfrac{1}{c_3} \fint_{a(k_\eta Y)}\dist^2\big(F + \dif{(\varphi_\eta \circ a^{-1})}, \SO(d)\big)\,\dx[z] - c_3 h^2\norm{B_h}_{\Leb^2(Y)}^2 - \eta.
\end{equation*}
Finally, by quasiconvexity of $\operatorname{Qdist}^2(\placeholder, \SO(d))$, we get
\begin{equation*}
	\widetilde{W}^h_{\hom}(F)
	\geq \tfrac{1}{c_3} \operatorname{Qdist}^2(F, \SO(d)) - c_3 h^2\norm{B_h}_{\Leb^2(Y)}^2 - \eta,
\end{equation*}
since $\varphi_\eta \circ a^{-1} \in \SobH^1_\text{per}(k_\eta\bar{A}Y, \R^d)$ and $a(k_\eta Y)$ is a periodicity cell for this kind of periodicity, cf.\ \cref{Lem:periodic_derivative} and \cite[Section 5.1.1]{Bra06}. Since this holds for arbitrary $\eta > 0$ with constants independent of $\eta$ and Zhang showed in \cite{Zha97} that $\operatorname{Qdist}^2(\placeholder, \SO(d))$ can again be controlled by $\dist^2(\placeholder, \SO(d))$, we conclude the claim.
\end{proof}

Before we proceed with the proof of the asymptotic expansion, we show the following technical lemma which we use for the linearization.

\begin{lemma} \label{Lem:ExpansionPeriodicityCell}
For $h > 0$, let $k_h \in \N$ and $\delta_h \in (0,\infty)$ with $\limsup_{h \to 0} \delta_h < \infty$ and $\lim_{h \to 0} h^2\delta_h^{-1} = 0$. Moreover, let $\Phi_h \in \Leb^2(k_hY, \R^{d\times d})$ with
\begin{equation*}
	\limsup_{h \to 0} \delta_h \fint_{k_hY} \abs{\Phi_h(y)}^2 \,\dx[y] < \infty,
\end{equation*}
and $\Psi_h, \Psi \in \Leb^2_\mathrm{per}(Y,\R^{d\times d})$ with $\delta_h\norm{\Psi_h - \Psi}^2_{\Leb^2(Y)} \to 0$. Then, there exist subsets $Y_h \subset k_hY$ with $\frac{1}{k_h^d} \abs{k_hY \setminus Y_h} \to 0$, such that
\begin{equation}
	\delta_h \abs{\fint_{k_h Y} \mathds{1}_{Y_h}(y) \Big( \tfrac{1}{h^2} W\big(y, I + h\Phi_h(y) + h\Psi_h(y)\big) - Q\big(y, \Phi_h(y) + \Psi(y)\big) \,\dx[y] \Big)} \to 0.
\end{equation}
Moreover, if $\lim_{h \to 0} \norm{h\Psi_h}_{\Leb^\infty(\R^d)} = 0$ and $\Phi_h$ satisfies the uniform bound
\begin{equation*}
	\limsup_{h \to 0} \delta_h \norm{\Phi_h}_{\Leb^\infty(k_hY)}^2 < \infty,
\end{equation*}
then we may choose $Y_h = k_hY$.
\end{lemma}
\begin{proof}
Let $G_h := \Psi_h + \Phi_h$. For some $Y_h \subset k_h Y$ (to be specified below), \ref{Property:W_Expansion} yields that 
\begin{equation*}
	\delta_h \abs{\fint_{k_h Y} \mathds{1}_{Y_h}(y)\Big( \tfrac{1}{h^2} W\big(y, I + h\Phi_h + h\Psi_h\big) -  Q\big(y, \Phi_h + \Psi\big)\Big) \,\dx[y]}
	\leq \text{(I)}_h + \text{(II)}_h,
\end{equation*}
where
\begin{align*}
	\text{(I)}_h &:= \delta_h \fint_{k_h Y} \mathds{1}_{Y_h} \abs{G_h}^2 r(y,\abs{hG_h}) \,\dx[y], &
	\text{(II)}_h &:= \delta_h \fint_{k_h Y} \abs{Q\big(y, \Phi_h + \Psi_h\big) - Q\big(y, \Phi_h + \Psi\big)} \,\dx[y].
\end{align*}
Since $Q$ is a quadratic form and satisfies \eqref{Eq:UniformBounds_Q}, we obtain $\text{(II)}_h \to 0$ from the $\Leb^2$-bound of $\Phi_h$ and the convergence and $Y$-periodicity of $\Psi_h$.
We proceed to show $\text{(I)}_h \to 0$ for suitable sets $Y_h$. Consider $\bar{r}_h(y) := r(y, h^{1/2}\delta_h^{-1/4})$ and $\rho_h := (\int_Y \bar{r}_h)^{1/2}$. 
By the second identity in \ref{Property:W_NonDegeneracy} we have $\bar{r}_h \leq 2\beta_{\rm el}$ a.e.\ for $h \ll 1$ and by \ref{Property:W_Expansion}, $\bar{r}_h$ converges to $0$ a.e.\ as $h \to 0$. Thus, by dominated convergence, $\rho_h \to 0$. Now set,
\begin{equation*}
	Y_h := \class{y \in k_hY}{\abs{G_h(y)} \leq h^{-1/2}\delta_h^{-1/4}} \cap \class{y \in k_hY}{\bar{r}_h(y) \leq \rho_h}.
\end{equation*}
Then, Markov's inequality, the $Y$-periodicity of $\bar{r}_h$ and the $\Leb^2$-boundedness of $\Phi_h$ and $\Psi_h$ yield that $\frac{1}{k_h^d} \abs{k_hY \setminus Y_h} \to 0$. With this definition, we obtain
\begin{equation*}
	\text{(I)}_h \leq \delta_h \fint_{k_h Y} \mathds{1}_{Y_h} \abs{G_h}^2 \bar{r}_h \,\dx[y] \leq \rho_h \fint_{k_h Y} \delta_h \abs{G_h}^2 \,\dx[y].
\end{equation*}
Since the integral on the right-hand side is bounded, we indeed obtain $\text{(I)}_h \to 0$. Finally, note that if $\Phi_h$ and $\Psi_h$ satisfy the uniform bounds, we can estimate $\text{(I)}_h$ differently as
\begin{align*}
	\text{(I)}_h &\leq c_1 \fint_{k_h Y} \delta_h(\norm{\Phi_h}_{\Leb^\infty}^2 + \abs{\Psi_h}^2) r(y, \norm{hG_h}_{\Leb^\infty}) \,\dx[y] \\
	&\leq c_2 \int_Y (1 + \delta_h\abs{\Psi_h - \Psi}^2 + \abs{\Psi}^2) r(y, \norm{hG_h}_{\Leb^\infty}) \,\dx[y].
\end{align*}
Since $\norm{hG_h}_{\Leb^\infty} \to 0$ and $r(y,\norm{hG_h}_{\Leb^\infty}) \leq 2 \beta_\mathrm{el}$ a.e.\ for $\norm{hG_h}_{\Leb^\infty} \leq \rho_\mathrm{el}$, dominated convergence (with strongly converging dominating sequence) shows that the right-hand side converges to $0$. Hence, in this case we may choose $Y_h = k_h Y$.
\end{proof}

\begin{proof}[Proof~\cref{Prop:AsyExpansion}]
\stepemph{Step 1 -- Reduction}: In order to treat the cases $G_h = B_{\hom}$ and $G_h \neq B_{\hom}$ simultaneously, we let throughout this proof, $\alpha_h := \abs{G_h - B_{\hom}}$ if $G_h \neq B_{\hom}$ and $\alpha_h := 1$ otherwise. As can be easily seen by a contradiction argument, it suffices to prove that
\begin{equation*}
	\frac{\abs{\widetilde{W}^h_{\hom}(I + hG_h) - h^2\left(\widetilde{Q}_{\hom}(G_h - B_{\hom}) + R^A(B)\right)}}{\abs{hG_h}^2 + h^2} \to 0, \qquad \text{as } h \to 0,
\end{equation*}
for an arbitrary sequence $(G_h) \subset \R^{d\times d}$ satisfying $hG_h \to 0$. Moreover, we may without loss assume that $H_h := \frac{1}{\alpha_h}(G_h - B_{\hom}) \to G$ for some $G \in \R^{d\times d}$ by the subsequence principle and compactness. We prove separately the upper and lower bounds
\begin{align} 
	\limsup_{h \to 0} \frac{\widetilde{W}^h_{\hom}(I + hG_h) - h^2\left(\widetilde{Q}_{\hom}(G_h - B_{\hom}) + R^A(B)\right)}{\abs{hG_h}^2 + h^2} \leq 0, \label{Eq:ProofAsyExpansion:Limsup}\\
	\liminf_{h \to 0} \frac{\widetilde{W}^h_{\hom}(I + hG_h) - h^2\left(\widetilde{Q}_{\hom}(G_h - B_{\hom}) + R^A(B)\right)}{\abs{hG_h}^2 + h^2} \geq 0. \label{Eq:ProofAsyExpansion:Liminf}
\end{align}
For the proof we adapt the argument of \cite[Thm 1.1]{MN11}.
\medskip

\stepemph{Step 2 -- Upper bound}: By definition of $\widetilde{W}^h_{\hom}$ we have for all $\sigma_h \in \SobW^{1,\infty}_\mathrm{per}(Y, \R^d)$ that
\begin{equation*}
	\widetilde{W}^h_{\hom}(I + hG_h) \leq \int_Y W\left(y, (I + hG_h + h\dif{\sigma_h}A(y)^{-1})(I - hB_h) \right) \,\dx[y].
\end{equation*}
Suppose that $\sigma_h$ satisfies
\begin{equation} \label{Eq:ProofAsyExpansion:sigmah}
	\limsup_{h \to 0} \tfrac{1}{\abs{G_h}^2 + 1} \norm{\dif{\sigma_h}}_{\Leb^\infty(Y)}^2 < \infty.
\end{equation}
Then, since $\limsup_{h \to 0} \norm{hB_h}_{\Leb^\infty(\R^d)} = 0$, the assumptions of \cref{Lem:ExpansionPeriodicityCell} including the uniform bound hold with
\begin{equation*}
	\delta_h := \tfrac{1}{\abs{G_h}^2 + 1}, \quad
	\Phi_h := G_h + \dif{\sigma_h}A(\placeholder)^{-1}, \quad 
	\Psi_h := -B_h - hG_h B_h - h\dif{\sigma_h}A(\placeholder)^{-1}B_h, \quad
	\Psi := -B.
\end{equation*}
Thus, we obtain
\begin{equation*}
	\widetilde{W}^h_{\hom}(I + hG_h) \leq \int_Y Q\big(y, G_h + \dif{\sigma_h}A(y)^{-1} - B\big) \,\dx[y] + (\abs{G_h}^2 + 1)\operatorname{rest}_1(h),
\end{equation*}
with $\operatorname{rest}_1(h) \to 0$ as $h \to 0$. Our goal is to construct a suitable map $\sigma_h$ satisfying \eqref{Eq:ProofAsyExpansion:sigmah}. This construction is done by estimating the latter term in \eqref{Eq:ProofAsyExpansion:Limsup} in two steps. We fix $\eta > 0$. First, by an approximation argument we find some $\varphi \in \SobW^{1,\infty}_\mathrm{per}(Y, \R^d)$ such that
\begin{equation*}
	\widetilde{Q}_{\hom}(G) \geq \int_Y Q\left(y, G + \dif{\varphi}A(y)^{-1}\right) \,\dx[y] - \eta.
\end{equation*}
Second, recall the construction of $B_{\hom}$ and $R^A(B)$ in \cref{Sec:Characterization_Bhom}. Note that $B_{\hom} - \sym B + \operatorname{P}_{(\mathcal{S} + \R^{d\times d}_{\sym})^\perp}(\sym B) = \operatorname{P}_S(B_{\hom} - \sym B) \in \mathcal{S}$, and thus, we find $\psi \in \SobH^1_\mathrm{per}(Y,\R^d)$ such that,
\begin{equation} \label{Eq:ProofAsyExpansion:BhomFormula}
	\sym B - \sym(\dif{\psi}A(\placeholder)^{-1})
	= B_{\hom} + \operatorname{P}_{(\mathcal{S} + \R^{d\times d}_{\sym})^\perp}(\sym B).
\end{equation}
Hence, again by an approximation argument, we find $\tilde{\psi} \in \SobW^{1,\infty}_\mathrm{per}(Y, \R^d)$ independent of $h$, such that with $\sigma_h := \alpha_h\varphi + \tilde{\psi}$ (recall $H_h = \frac{1}{\alpha_h}(G_h - B_{\hom})$),
\begin{align*}
	&\int_Y Q\left(y, G_h + \dif{\sigma_h}A(y)^{-1} - B\right) \,\dx[y] \\
	&\qquad= \int_Y Q\left(y, G_h + (\alpha_h\dif{\varphi} + \dif{\psi})A(y)^{-1} - B\right) \,\dx[y] \\
	&\qquad\leq \int_Y Q\left(y, G_h + (\alpha_h\dif{\varphi} + \dif{\tilde{\psi}})A(y)^{-1} - B\right) \,\dx[y] + (\abs{G_h}^2 + 1)\eta \\
	&\qquad= \int_Y Q\left(y, G_h - B_{\hom} + \alpha_h\dif{\varphi}A(y)^{-1}\right) \,\dx[y] + R^A(B) + (\abs{G_h}^2 + 1)\eta \\
	&\qquad= \alpha_h^2 \int_Y Q\left(y, H_h + \dif{\varphi}A(y)^{-1}\right) \,\dx[y] + R^A(B) + (\abs{G_h}^2 + 1)\eta.
\end{align*}
Note that $\sigma_h$ satisfies \eqref{Eq:ProofAsyExpansion:sigmah}. Combining these results, we conclude
\begin{align*}
	\text{LHS of \eqref{Eq:ProofAsyExpansion:Limsup}}
	&\leq \limsup_{h \to 0} \frac{\int_Y Q(y, G_h + \dif{\sigma_h}A(y)^{-1} - B) \,\dx[y] - \widetilde{Q}_{\hom}(G_h - B_{\hom}) - R^A(B)}{\abs{G_h}^2 + 1} \\
	&\leq \limsup_{h \to 0} \left(\int_Y Q(y, H_h + \dif{\varphi}A(y)^{-1}) \,\dx[y] - \widetilde{Q}_{\hom}(H_h) \right) \tfrac{\alpha_h^2}{\abs{G_h}^2 + 1} + \eta \\
	&= \left(\int_Y Q(y, G + \dif{\varphi}A(y)^{-1}) \,\dx[y] - \widetilde{Q}_{\hom}(G) \right) \limsup_{h \to 0} \tfrac{\alpha_h^2}{\abs{G_h}^2 + 1} + \eta 
	\leq c_1 \eta.
\end{align*}
Since $\eta > 0$ was arbitrary and the constant $c_1$ is independent of $\eta$, we conclude \eqref{Eq:ProofAsyExpansion:Limsup}.
\medskip

\stepemph{Step 3 -- Lower bound}:
We observe,
\begin{equation*}
	\text{LHS of \eqref{Eq:ProofAsyExpansion:Liminf}} \geq \liminf_{h \to 0} \frac{\widetilde{W}^h_{\hom}(I + hG_h)}{\abs{hG_h}^2 + h^2} - \limsup_{h \to 0} \frac{\widetilde{Q}_{\hom}(G_h - B_{\hom}) + R^A(B)}{\abs{G_h}^2 + 1}.
\end{equation*}
Thus, without loss of generality we may assume that $\liminf_{h \to 0} \frac{\widetilde{W}^h_{\hom}(I + hG_h)}{\abs{hG_h}^2 + h^2}$ is finite and restrict ourselves to a subsequence (not relabeled) where the $\liminf$ is achieved as a limit.
By definition of $\widetilde{W}^h_{\hom}$, for all $h > 0$ we find $k_h \in \N$ and $\sigma_h \in \SobW^{1,\infty}_\mathrm{per}(k_hY, \R^d)$, such that
\begin{equation*}
	\widetilde{W}^h_{\hom}(I + hG_h)\geq \fint_{k_hY} W\left(y, (I + hG_h + h\dif{\sigma_h}A(y)^{-1}) (I - hB_h) \right) \,\dx[y] - h(\abs{hG_h}^2 + h^2).
\end{equation*}
In order to apply \cref{Lem:ExpansionPeriodicityCell}, we want to show
\begin{equation} \label{Eq:ProofAsyExpansion:sigmah_lower}
	\limsup_{h \to 0} \tfrac{1}{\abs{G_h}^2 + 1} \fint_{k_h Y} \abs{\dif{\sigma}_h}^2 \,\dx[y] < \infty.
\end{equation}
Indeed, by the rigidity estimate \cref{Thm:RigidityEstimateJonesDomain}, we find a constant rotation $R_h \in \SO(d)$ and a constant $c_2 > 0$ independent of $k_h$ such that
\begin{equation*}
	 \fint_{k_h Y} \abs{I + hG_h + h\dif{\sigma_h}A(y)^{-1} - R_h}^2 \,\dx[y]
	 \leq c_2 \fint_{k_h Y} \dist^2\big(I + hG_h + h\dif{\sigma_h}A(y)^{-1}, \SO(d)\big) \,\dx[y].
\end{equation*}
Arguing as in \cref{Prop:NonDegeneracy}, using the definition of $\sigma_h$, we may estimate the right-hand side to find
\begin{equation*}
	\fint_{k_h Y} \abs{I + hG_h + h\dif{\sigma_h}A(y)^{-1} - R_h}^2 \,\dx[y] 
	\leq c_3 \left(\widetilde{W}^h_{\hom}(I+hG_h) + h(\abs{hG_h}^2 + h^2) + h^2 \right).
\end{equation*}
Using $\fint_{a(k_hY)} \dif{(\sigma_h \circ a^{-1})} = 0$, see \cref{Lem:periodic_derivative}, and Pythagoras, we now obtain,
\begin{align*}
	\tfrac{1}{\abs{G_h}^2 + 1}\fint_{k_hY} \abs{\dif{\sigma_h}}^2 \,\dx[y] 
	&\leq \tfrac{c_4}{\abs{G_h}^2 + 1} \fint_{k_hY} \abs{\dif{\sigma_h}\dif{a}(y)^{-1}}^2 \det \dif{a} \,\dx[y] \\ 
	&\leq \tfrac{c_5}{\abs{G_h}^2 + 1}\left(\fint_{a(k_hY)} \abs{\dif{(\sigma_h \circ a^{-1})}}^2 \,\dx[z] + \abs{\frac{I + hG_h - R_h}{h}}^2\right) \\
	&= \tfrac{c_5}{\abs{G_h}^2 + 1} \fint_{a(k_hY)} \abs{\frac{I + hG_h - R_h}{h} + \dif{(\sigma_h \circ a^{-1})}}^2 \,\dx[z] \\
	&\leq \tfrac{c_6}{\abs{hG_h}^2 + h^2} \fint_{k_hY} \abs{I + hG_h + h\dif{\sigma_h}\dif{a}(y)^{-1} - R_h}^2 \,\dx[y] \\
	&\leq c_7 \left(\frac{\widetilde{W}^h_{\hom}(I+hG_h)}{\abs{hG_h}^2 + h^2} + 1\right).
\end{align*}
Thus, we may apply \cref{Lem:ExpansionPeriodicityCell} with
\begin{equation*}
	\delta_h := \tfrac{1}{\abs{G_h}^2 + 1}, \quad
	\Phi_h := G_h + \dif{\sigma_h}A(\placeholder)^{-1}, \quad 
	\Psi_h := -B_h - hG_h B_h - h\dif{\sigma_h}A(\placeholder)^{-1}B_h, \quad
	\Psi := -B,
\end{equation*}
and find sets $Y_h \subset k_h Y$ with $\frac{1}{k_h^d}\abs{k_h Y \setminus Y_h} \to 0$, such that
\begin{align*}
	\widetilde{W}^h_{\hom}(I + hG_h)
	&\geq \fint_{k_hY} \mathds{1}_{Y_h} W\left(y, (I + hG_h + h\dif{\sigma_h}A(y)^{-1}) (I - hB_h) \right) \,\dx[y] - h(\abs{hG_h}^2 + h^2) \\
	&\geq h^2\fint_{k_hY} \mathds{1}_{Y_h} Q\left(y,G_h + \dif{\sigma_h}A(y)^{-1} - B\right)\,\dx[y] + (\abs{hG_h}^2 + h^2)\operatorname{rest}_2(h),
\end{align*}
with $\operatorname{rest}_2(h) \to 0$. Let $\psi \in \SobH^1_\mathrm{per}(Y,\R^d)$ be as in \eqref{Eq:ProofAsyExpansion:BhomFormula}, set $\Psi := \operatorname{P}_{(\mathcal{S} + \R^{d\times d}_{\sym})^\perp}(\sym B)$ and define $\tilde{\varphi}_h \in \SobH^1_\mathrm{per}(k_hY, \R^d)$ by $\sigma_h = \psi + \alpha_h \tilde{\varphi}_h$. We observe as in Step 2,
\begin{equation*}
	Q\left(y, G_h + \dif{\sigma_h}(y)A(y)^{-1} - B(y)\right)
	= \alpha_h^2 Q\left(y,H_h + \dif{\tilde{\varphi}_h}(y)A(y)^{-1} - \alpha_h^{-1}\Psi(y) \right).
\end{equation*}
Moreover, let $\varphi_G \in \SobH^1_\mathrm{per}(Y,\R^d)$ the corrector for $\widetilde{Q}_{\hom}(G)$, cf.\ \cref{Prop:Bhom_algorithmic}. Then, using orthogonality and $Y$-periodicity, we observe
\begin{equation*}
	\widetilde{Q}_{\hom}(G) + \alpha_h^{-2} R^A(B) = \fint_{k_hY} Q\left(y, G + \dif{\varphi_G}A(y)^{-1} - \alpha_h^{-1}\Psi\right) \,\dx[y].
\end{equation*}
Concluding the results of this step so far, we obtain
\begin{align*}
	&\text{LHS of \eqref{Eq:ProofAsyExpansion:Liminf}} \\
	&\quad\geq \liminf_{h \to 0} \tfrac{\alpha_h^2}{\abs{G_h}^2 + 1}
		\begin{aligned}[t]
			\Bigg( \fint_{k_hY} \mathds{1}_{Y_h} Q\left(y,H_h + \dif{\tilde{\varphi}_h}A(y)^{-1} - \alpha_h^{-1}\Psi \right) \,\dx[y]& \\
			 - \fint_{k_hY} Q\left(y, G + \dif{\varphi_G}A(y)^{-1} - \alpha_h^{-1}\Psi\right) \,\dx[y]& 
			 + \widetilde{Q}_{\hom}(G) - \widetilde{Q}_{\hom}(H_h) \Bigg)
		\end{aligned} \\
	&\quad= \liminf_{h \to 0} \tfrac{\alpha_h^2}{\abs{G_h}^2 + 1}
		\begin{aligned}[t]
			\Bigg( \fint_{k_hY} \mathds{1}_{Y_h} Q\left(y,H_h + \dif{\tilde{\varphi}_h}A(y)^{-1} - \alpha_h^{-1}\Psi \right) \,\dx[y]& \\
		 - \fint_{k_hY} Q\left(y, G + \dif{\varphi_G}A(y)^{-1} - \alpha_h^{-1}\Psi\right) \,\dx[y]& \Bigg).
		\end{aligned} 
\end{align*}
Since $Q$ is a quadratic form, there holds the formula $Q(y,G_1) - Q(y,G_2) \geq 2(G_1 - G_2):\mathbb{L}_Q(y)G_2$ for all $G_1, G_2 \in \R^{d\times d}$. Thus, we can treat the latter term as follows,
\begin{gather*}
	\frac{1}{2} \fint_{k_hY} \mathds{1}_{Y_h} Q\left(y,H_h + \dif{\tilde{\varphi}_h}A(y)^{-1} - \alpha_h^{-1}\Psi \right)
	- Q\left(y, G + \dif{\varphi_G}A(y)^{-1} - \alpha_h^{-1}\Psi\right) \,\dx[y] \\
	\geq \text{(I)}_h + \text{(II)}_h + \text{(III)}_h,
\intertext{where}
\begin{aligned}
	\text{(I)}_h &:= \fint_{k_hY} \Big[ H_h + \dif{\tilde{\varphi}_h}A(y)^{-1} - \alpha_h^{-1}\Psi \Big] : \mathbb{L}_Q \Big[ (\mathds{1}_{Y_h} - 1)\big(G + \dif{\varphi_G}A(y)^{-1} - \alpha_h^{-1}\Psi\big) \Big] \,\dx[y], \\
	\text{(II)}_h &:= \fint_{k_hY} \Big[ H_h - G \Big] : \mathbb{L}_Q \Big[ G + \dif{\varphi_G}A(y)^{-1} - \alpha_h^{-1}\Psi \Big] \,\dx[y], \\
	\text{(III)}_h &:= \fint_{k_hY} \Big[ (\dif{\tilde{\varphi}_h} - \dif{\varphi_G})A(y)^{-1} \Big] : \mathbb{L}_Q \Big[ G + \dif{\varphi_G}A(y)^{-1} - \alpha_h^{-1}\Psi \Big] \,\dx[y].
\end{aligned}
\end{gather*}
We show that each term, potentially multiplied by $\frac{\alpha_h^2}{\abs{G_h}^2 + 1}$, converges to $0$. Define $\chi_h(y) := \frac{1}{k_h^d} \sum_{\xi \in \Z^d, \xi + Y \subset k_hY} (1-\mathds{1}_{Y_h}(y + \xi))$, $y \in Y$. Then $\frac{\alpha_h^2}{\abs{G_h}^2 + 1}\text{(I)}_h$ converges to $0$ by the Cauchy-Schwartz inequality, since 
\begin{equation*}
	\tfrac{\alpha_h^2}{\abs{G_h}^2 + 1} \fint_{k_h Y} \abs{H_h + \dif{\tilde{\varphi}_h}A(y)^{-1} + \alpha_h^{-1}\Psi}^2 \,\dx[y] 
\end{equation*}
is bounded due to \eqref{Eq:ProofAsyExpansion:sigmah_lower} and by periodicity,
\begin{equation*}
	\fint_{k_hY}\mathds{1}_{k_hY \setminus Y_h}\abs{G + \dif{\varphi_G}A(y)^{-1} + \alpha_h^{-1}\Psi}^2 \,\dx[y]
	= \int_Y \chi_h \abs{G + \dif{\varphi_G}A(y)^{-1} + \alpha_h^{-1}\Psi}^2 \,\dx[y],
\end{equation*}
which converges to $0$ when multiplied with $\tfrac{\alpha_h^2}{\abs{G_h}^2 + 1}$ by dominated convergence, since $\norm{\chi_h}_{\Leb^\infty(Y)} \leq 1$ and $\norm{\chi_h}_{\Leb^1(Y)} = \frac{1}{k_h^d} \abs{k_h Y \setminus Y_h}$. Similarly, using the Cauchy-Schwartz inequality, we observe $\frac{\alpha_h^2}{\abs{G_h}^2 + 1}\text{(II)}_h \to 0$, since $H_h \to G$. To treat $\text{(III)}_h$, we note that according to \cite[Thm.~2.1]{Mar78} $\varphi_G$ also minimizes
\begin{equation*}
	\psi \in \SobH^1_\mathrm{per}(k_hY, \R^d) \mapsto \fint_{k_h Y} Q\big(y, G + \dif{\psi}A(y)^{-1}\big) \,\dx[y].
\end{equation*}
Referring to the associated Euler-Lagrange equation, we obtain
\begin{equation*}
	\text{(III)}_h = \alpha_h^{-1}\fint_{k_hY} \Big[ \big(\dif{\tilde{\varphi}_h} - \dif{\varphi_G}\big)A(y)^{-1} \Big] : \mathbb{L}_Q \Psi \,\dx[y].
\end{equation*} 
Thus, we obtain $\text{(III)}_h = 0$, since similarly $\Psi = \operatorname{P}_{(\mathcal{S} + \R^{d\times d}_{\sym})^\perp}(\sym B)$ satisfies the Euler-Lagrange equation
\begin{equation*}
	\fint_{k_hY} \Big[\hat{G} + \dif{\hat{\varphi}}A(y)^{-1}\Big] : \mathbb{L}_Q \Psi \,\dx[y] = 0, \qquad
	\text{for all } \hat{G} \in \R^{d\times d}, \hat{\varphi} \in \SobH^1_{\rm per}(k_hY, \R^d).
\end{equation*}
This completes the proof.
\end{proof}

Finally, we prove \cref{Cor:ExpansionMinimizer} in the following equivalent form:
\begin{corollary}
Let $(F_h^*) \subset \R^{d\times d}$ denote a sequence of almost minimizers for $(\widetilde{W}^h_{\hom})$ in the sense that
\begin{equation*}
	\limsup_{h \to 0} \frac{1}{h^2} \abs{\widetilde{W}^h_{\hom}(F_h^*) - \inf_{F \in \R^{d\times d}} \widetilde{W}^h_{\hom}(F)} = 0.
\end{equation*}
Then there exist rotations $R_h \in \SO(d)$, such that
\begin{equation}
	F_h^* = R_h(I + h B_{\hom}) + \oclass(h).
\end{equation}
\end{corollary}

\begin{proof}
Let $\mathcal{F}_h, \mathcal{F}: \R^{d\times d} \to \R$ with $\mathcal{F}_h(G) := \frac{1}{h^2}\widetilde{W}^h_{\hom}(I+hG)$ and $\mathcal{F}_h(G) := \widetilde{Q}_{\hom}(G - B_{\hom}) + R^A(B)$. Then, \cref{Prop:AsyExpansion} implies $\Gamma$-convergence of $\mathcal{F}_h$ to $\mathcal{F}$. Thus, any subsequence of any bounded sequence $(G_h) \subset \R^{d\times d}$ with $\mathcal{F}_h(G_h) - \inf \mathcal{F}_h \to 0$, admits a further subsequence which converges to a minimizer of $\mathcal{F}$. We construct such a sequence. \cref{Prop:AsyExpansion} clearly implies $\widetilde{W}^h_{\hom}(F_h^*) \to 0$ as $h \to 0$. Hence, non-degeneracy \cref{Prop:NonDegeneracy} implies that $\det F_h^* > 0$ if $h \ll 1$. Thus, by polar decomposition we obtain a rotation $R_h \in \SO(d)$ such that
\begin{equation*}
	F_h^* = R_h\sqrt{{F_h^*}^T F_h^*}.
\end{equation*}
Let $G_h^* := \frac{1}{h}(\sqrt{{F_h^*}^T F_h^*} -I)$. Then, by frame indifference of $\widetilde{W}^h_{\hom}$, $G_h^*$ satisfies $\mathcal{F}_h(G_h^*) - \inf \mathcal{F}_h \to 0$. Moreover, by \cref{Prop:NonDegeneracy},
\begin{equation*}
	\abs{G_h^*}^2 = \tfrac{1}{h^2}\abs{\sqrt{{F_h^*}^T F_h^*} -I}^2 \leq \tfrac{1}{h^2}\dist^2(F_h^*, \SO(d)) \leq c_1\left(\frac{\widetilde{W}^h_{\hom}(F_h^*)}{h^2} + 1 \right).
\end{equation*}
Thus, $(G_h^*) \subset \R^{d\times d}_{\sym}$ is bounded and any subsequence admits a further subsequence which converges to a minimizer of $\mathcal{F}$. Since $B_{\hom}$ is the unique minimizer of $\mathcal{F}$ in $\R^{d\times d}_{\sym}$, the whole sequence $G_h^*$ must converge to $B_{\hom}$. Hence, the claim follows.
\end{proof}

\subsection{Linearization; Proof of Theorem \ref{Thm:gamma_convergence} (1) and (4), (\ref{Eq:equi_coercivity_Ihhom}) and Proposition \ref{Prop:strong_conv} (a) and (b)}
\label{Sec:ProofLinearization}

In this section we show the directions (1) and (4) in \cref{Thm:gamma_convergence} as well as an associated equi-coercivity statement which especially includes \eqref{Eq:equi_coercivity_Ihhom} as a consequence. We show here a stronger statement which includes (1) and (4) simultaneously and does not require any periodicity assumptions. Indeed, note that $W^h$, $A$, $B_h$ and $B$ as well as $W^h_{\hom}$, $\bar{A}$ and $B_{\hom}$ satisfy the assumptions of the following theorem if $h$ is small enough. For $W^h_{\hom}$ this is a consequence of \cref{Thm:Expansion_Whom} and for $W^h$ this follows from similar but easier considerations.

\begin{theorem} \label{Thm:Linearization}
Let $\alpha > 0$, $\hat{W}^h: \Omega \times \R^{d\times d} \to [0,\infty]$ a sequence of Carathéodory functions continuous in the second component, $\hat{A} \in \operatorname{SFJ(\Omega)}$ and $\hat{B}_h, \hat{B} \in \Leb^2(\Omega, \R^{d\times d})$ such that for a.e.\ $x \in \Omega$ and all $h > 0$, the following properties hold.
\begin{enumerate}[(i)]
	\item (Frame indifference): $\hat{W}^h(x,RF) = \hat{W}^h(F)$ for all $F \in \R^{d\times d}$, $R \in \SO(d)$.
	
	\item (Non-degeneracy): $\hat{W}^h(x,F) \geq \tfrac{1}{\alpha} \dist^2\big(F\hat{A}(x)^{-1}, \SO(d)\big) - \alpha h^2 \big(\abs{\hat{B}_h(x)}^2 + 1\big)$ for all $F \in \R^{d\times d}$.
	
	\item (Asymptotic expansion): There exists a quadratic form $\tilde{Q}: \Omega \times \R^{d\times d} \to \R$, a map $\hat{R} \in \Leb^2(\Omega)$ and a remainder $\hat{r}: \Omega \times [0,\infty) \to [0,\infty]$ (which is measurable in the first, continuous and increasing in the second component and satisfies $\limsup_{\delta \to 0} \operatorname{ess\,sup}_{x \in \Omega} \hat{r}(x,\delta) < \infty$ and $\lim_{\delta \to 0}\hat{r}(x,\delta) = 0$), such that for all $G \in \R^{d\times d}$,
	\begin{gather*}
		\abs{\tfrac{1}{h^2} \hat{W}^h(x, A(x) + hG) - \big(\tilde{Q}(x,G - \hat{B}(x)\hat{A}(x)) + \hat{R}(x)\big)} \\
		\leq (1 + \abs{\hat{B}_h(x)}^2 + \abs{G}^2) \hat{r}(x, h + \abs{hG})
	\end{gather*}
	and with $\sym_{\hat{A}(x)}$ defined as in \cref{L:coercivity},
	\begin{equation*}
		\tfrac{1}{\alpha} \big|\sym_{\hat{A}(x)} G\big|^2 \leq \tilde{Q}(x,G) \leq \alpha \big|\sym_{\hat{A}(x)} G\big|^2.
	\end{equation*}
	
	\item $\hat{B}_h \to \hat{B}$ in $\Leb^2(\Omega, \R^{d\times d})$ and $\limsup_{h \to 0} \norm{h \hat{B}_h}_{\Leb^\infty(\Omega)} = 0$.
\end{enumerate}
Then, the energy functionals $\hat{\mathcal{I}}^h(u) := \int_\Omega \hat{W}^h(x, \hat{A}(x) + h\dif{u}(x)) \,\dx$ $\Gamma$-converge w.r.t.\ weak convergence in $\SobH^1_{\Gamma, g}(\Omega, \R^d)$ to $\hat{\mathcal{I}}^\mathrm{lin}(u) := \int_\Omega \hat{Q}(x, \dif{u}(x) - \hat{A}(x)\hat{B}(x))  + \hat{R}(x)\,\dx$.
\end{theorem}

Before we move on to the proof, we provide the following lemma whose proof follows analogously to \cref{Lem:ExpansionPeriodicityCell}.

\begin{lemma} \label{Lem:ExpansionLinearization}
Consider the situation of \cref{Thm:Linearization}. Let $(\Phi_h) \subset \Leb^2(\Omega, \R^{d\times d})$ bounded. Then, there exist subsets $\Omega_h \subset \Omega$ with $\abs{\Omega \setminus \Omega_h} \to 0$, such that
\begin{equation*}
	\Bigg| \frac{1}{h^2} \int_{\Omega_h} \hat{W}^h\left(x, \hat{A}(x) + h\Phi_h(x)\right) \,\dx
	- \int_{\Omega_h} \hat{Q}\left(x, \Phi_h(x) - \hat{B}(x)\hat{A}(x)\right) \,\dx - \int_\Omega \hat{R} \Bigg| \to 0.
\end{equation*}
Moreover, if $(\Phi_h)$ is bounded in $\Leb^\infty(\Omega, \R^{d\times d})$ then we may choose $\Omega_h = \Omega$.
\end{lemma}

\begin{proof}[Proof of \cref{Thm:Linearization}]

\stepemph{Step 1 -- Recovery sequence}: Let $u \in \SobH^1_{\Gamma,g}(\Omega, \R^d)$. By \cref{Def:boundary_conditions} there is a sequence $(u_\delta) \subset \SobW^{1,\infty}_{\Gamma, g}(\Omega, \R^d)$ converging to $u$ strongly in $\SobH^1$ as $\delta \to 0$. It is sufficient to show 
\begin{equation*}
	\hat{\mathcal{I}}^h(u_\delta) \xrightarrow[\text{(I)}]{h \to 0} \hat{\mathcal{I}}^\mathrm{lin}(u_\delta) \xrightarrow[\text{(II)}]{\delta \to 0} \hat{\mathcal{I}}^\mathrm{lin}(u),
\end{equation*}
because then we may extract a diagonal sequence $\delta(h)$ with $\hat{\mathcal{I}}^h(u_{\delta(h)}) \to \hat{\mathcal{I}}^\mathrm{lin}(u)$ using Attouch's diagonalization lemma, see \cite[Lemma 1.15]{Att84}. It remains for us to show (I), since (II) is a direct consequence of the continuity of $\hat{\mathcal{I}}^\mathrm{lin}$ w.r.t.\ strong convergence in $\SobH^1(\Omega,\R^d)$. But (I) follows by applying \cref{Lem:ExpansionLinearization} with the constant sequence $\Phi_h = \dif{u_\delta}$.
\medskip

\stepemph{Step 2 -- Lower bound}: Let again $u \in \SobH^1_{\Gamma,g}(\Omega, \R^d)$ and $(u_h) \subset \SobH^1_{\Gamma,g}(\Omega, \R^d)$ be an arbitrary sequence converging weakly to $u$ in $\SobH^1(\Omega)$. We show
\begin{equation*}
	\liminf_{h \to 0} \hat{\mathcal{I}}^h(u_h) \geq \hat{\mathcal{I}}^\mathrm{lin}(u).
\end{equation*}
Since $(\dif{u_h}) \subset \Leb^2(\Omega, \R^{d\times d})$ is bounded, we may apply \cref{Lem:ExpansionLinearization} with $\Phi_h = \dif{u_h}$. Thus, since $\hat{W}^h$ is non-negative, we obtain
\begin{equation*}
	\liminf_{h \to 0} \hat{\mathcal{I}}^h(u_h) \geq \liminf_{h \to 0} \int_{\Omega}  \hat{Q}\left(x, \mathds{1}_{\Omega_h}(x)(\dif{u_h}(x) - \hat{B}(x)\hat{A}(x))\right) + \hat{R}(x) \,\dx.
\end{equation*}
Finally, since $\mathds{1}_{\Omega_h}$ is bounded in $\Leb^\infty(\Omega)$ and converges to $1$ strongly in $\Leb^2(\Omega)$, we obtain that $\mathds{1}_{\Omega_h}(\dif{u_h} - \hat{B}\hat{A})$ converges weakly to $\dif{u} - \hat{B}\hat{A}$ in $\Leb^2(\Omega, \R^{d\times d})$. Hence, weak lower semi-continuity of the quadratic integral functional, yields 
\begin{equation*}
	\liminf_{h \to 0} \int_{\Omega}  \hat{Q}\left(x, \mathds{1}_{\Omega_h}(x)(\dif{u_h}(x) - \hat{B}(x)\hat{A}(x))\right) + \hat{R}(x) \,\dx \geq \hat{\mathcal{I}}^\mathrm{lin}(u). \qedhere
\end{equation*}
\end{proof}

Moreover, we prove the following equi-coercivity statement. The approach to prove this statement mimics \cite[§5.3.1]{NeuPHD} and is influenced by \cite[§5.1]{BNS20} and \cite[§3]{DNP02}.

\begin{theorem} \label{Thm:Equi-coercivityLinearization}
Consider the situation of \cref{Thm:Linearization}. Then, we find a constant $c > 0$, such that for all $h > 0$,
\begin{equation}
	\norm{u}_{\SobH^1(\Omega)} \leq c \left(\hat{\mathcal{I}}^h(u) + 1\right).
\end{equation}
\end{theorem}

\begin{proof}
We show that there is a constant $c_1$ depending only on $\Omega$, $\Gamma$ and $\hat{A}$ such that
\begin{equation} \label{Eq:ProofCoercivityLinearization:Apriori}
	\int_\Omega \abs{\dif{u}(x)}^2 \,\dx \leq c_1 \left(\frac{1}{h^2} \int_\Omega \dist^2\big(I + h\dif{u}(x)\hat{A}(x)^{-1}, \SO(d)\big) \,\dx + \norm{g}_{\Leb^2(\Gamma, \mathcal{H}^{d-1})}^2\right).
\end{equation}
Then the claim follows immediately from the non-degeneracy of $\hat{W}^h$ and the Poincaré-Friedrich inequality. \eqref{Eq:ProofCoercivityLinearization:Apriori} can be shown in the following way. By the geometric rigidity estimate \cref{Thm:RigidityEstimateJonesDomain}, we find a rotation $R \in \SO(d)$, such that
\begin{equation} \label{Eq:ProofCoercivityLinearization:Rigidity}
	\int_{\Omega} \abs{I + h\dif{u}(x)\hat{A}(x)^{-1} - R}^2 \,\dx \leq c_2 \int_\Omega \dist^2\big(I + h\dif{u}(x)\hat{A}(x)^{-1}, \SO(d)\big) \,\dx.
\end{equation}
Note that $I + h\dif{u}\hat{A}(\placeholder)^{-1} = \dif{\phi}\hat{A}(\placeholder)^{-1}$, where $\phi = a + hu$ and $a: \Omega \to \R^d$ is the potential of $\hat{A}$, i.e.\ $\hat{A} = \dif{a}$. We seek to estimate the difference $I - R$ suitably. By \ref{Property:SFJisPiecewiseBilipschitz} we find some Lipschitz domain $\Omega_i \subset \Omega$, such that $a$ is Bilipschitz on $\Omega_i$ and $\mathcal{H}^{d-1}(\Gamma \cap \partial \Omega_i) > 0$. Then, $u - g = 0$ on $\Gamma$ implies $u \circ a^{-1} - g \circ a^{-1} = 0$ on $a(\Gamma \cap \partial\Omega_i)$ in the sense of traces. Thus, by choosing a suitable $\xi_h \in \R^d$, \cref{Cor:MatrixNorm}, continuity of the trace operator and the Poincaré-Wirtinger inequality yield
\begin{align*}
	\abs{I - R}^2 
	&\leq c_3  \abs{I - R}_{a(\Gamma \cap \partial\Omega_i), 2}^2 
	\leq c_4 \int_{a(\Gamma \cap \partial\Omega_i)} \abs{z - Rz + hu(a^{-1}(z)) - hg(a^{-1}(z)) - \xi_h}^2 \,\dx[\mathcal{H}^{d-1}(z)] \\
	&\leq c_5 \int_{\partial a(\Omega_i)} \abs{z - Rz + hu(a^{-1}(z)) - \xi_h}^2 \,\dx[\mathcal{H}^{d-1}(z)] + c_5 h^2 \norm{g \circ a^{-1}}_{\Leb^2(a(\Gamma \cap \partial\Omega_i), \mathcal{H}^{d-1})}^2 \\
	&\leq c_6 \int_{a(\Omega_i)} \abs{I + h\dif{(u \circ a^{-1})}(z) - R}^2\,\dx[\mathcal{H}^{d-1}(z)] + c_6 h^2 \norm{g}_{\Leb^2(\Gamma, \mathcal{H}^{d-1})}^2. \\
	&\leq c_7 \int_{\Omega} \abs{I + h\dif{u}(x)\dif{a}(x)^{-1} - R}^2\,\dx[\mathcal{H}^{d-1}(x)] + c_7 h^2 \norm{g}_{\Leb^2(\Gamma, \mathcal{H}^{d-1})}^2.
\end{align*}
Hence, we conclude \eqref{Eq:ProofCoercivityLinearization:Apriori} by combining this estimate with \eqref{Eq:ProofCoercivityLinearization:Rigidity}.
\end{proof}

Finally, we state the following strong convergence result in the spirit of \cref{Prop:strong_conv}. In fact, this result admits \cref{Prop:strong_conv} (a) and (b) as corollaries.

\begin{proposition} \label{Prop:strong-convergenceLinearization}
Consider the situation of \cref{Thm:Linearization}. Let $u_h \wto u$ weakly in $\SobH^1(\Omega, \R^d)$ and assume $\hat{\mathcal{I}}^h(u_h) \to \hat{\mathcal{I}}^\mathrm{lin}(u)$. Then, $\dif{u_h} \to \dif{u}$ strongly in $\Leb^2(\Omega, \R^{d\times d})$.
\end{proposition}

At the heart of upgrading to strong convergence lies a formula for quadratic forms, stating
\begin{equation} \label{Eq:QuadraticTrick}
	\hat{Q}(x,F - \hat{F}) = \hat{Q}(x,F) - \hat{Q}(x,\hat{F}) + 2 (\hat{F} - F) : \mathbb{L}_{\hat{Q}}(x) \hat{F}, \qquad x \in \Omega,\, F,\hat{F} \in \R^{d\times d}.
\end{equation}
This allows to upgrade a convergence of energies and weak convergence in $\Leb^2(\Omega, \R^{d\times d})$ to strong convergence, at least for the symmetric part using $\abs{\sym(F - \hat{F})}^2 \leq \alpha \hat{Q}(x,F - \hat{F})$. Using this trick, strong convergence in $\SobW^{1,p}$ for $1 \leq p < 2$ for the linearization in elasticity was already established in \cite{DNP02}. The argument was later refined in \cite{ADD12} to also include $p = 2$. The essential ingredient for this was the observation that uniform integrability of the distance to the set of rotations implies uniform integrability of the distance to a single rotation. For our purposes here, we need the following variant of this statement.

\begin{proposition} \label{Prop:uniform_integrability}
Let $1 < p < \infty$, $\mathcal{F} \subset W^{1,1}(\Omega, \R^d)$ and $\mathcal{A} \subset \operatorname{SFJ}(\Omega)$ such that \cref{Thm:RigidityEstimateJonesDomain} holds with a uniform constant in $\mathcal{A}$ and $(\eta_\phi^A)_{\phi \in \mathcal{F}}^{A \in \mathcal{A}} \subset (0,\infty)$. Assume
\begin{equation*}
	\class{\eta_\phi^A \dist^p(\dif{\phi}A(\placeholder)^{-1}, \SO(d))}{\phi \in \mathcal{F}, A \in \mathcal{A}}
\end{equation*}
to be uniformly integrable. Then, there exist rotations $(R_\phi^A)_{\phi \in \mathcal{F}}^{A \in \mathcal{A}} \subset \SO(d)$, such that
\begin{equation*}
	\class{\eta_\phi^A \abs{\dif{\phi}A(\placeholder)^{-1} - R_\phi^A}^p}{\phi \in \mathcal{F}, A \in \mathcal{A}}
\end{equation*}
is uniformly integrable.
\end{proposition}

This proposition is a consequence of the mixed growth version of the rigidity estimate, see \cref{Thm:RigidityEstimateJonesDomain}. The proof exactly follows \cite[Cor.~4.2]{CDM14} but for the readers convenience we sketch the main arguments of the proof.

\begin{proof}
Let $d_\phi^A := (\eta_\phi^A)^{1/p} \dist(\dif{\phi}A(\placeholder)^{-1}, \SO(d))$. Due to uniform integrability, we find $T_\eps > 0$, such that
\begin{equation*}
	\int_{\set{d_\phi^A > T_\eps}} \abs{d_\phi^A}^p \leq \eps.
\end{equation*}
Consider the decompositions $\dist(\dif{\phi}A(\placeholder)^{-1}, \SO(d)) = F_\phi^A + G_\phi^A$ in $\Leb^p + \Leb^\infty(\Omega)$, where
\begin{equation*}
	F_\phi^A := \dist(\dif{\phi}A(\placeholder)^{-1}, \SO(d)) \mathds{1}_{\set{d_\phi^A > T_\eps}}, \qquad
	G_\phi^A := \dist(\dif{\phi}A(\placeholder)^{-1}, \SO(d)) \mathds{1}_{\set{d_\phi^A \leq T_\eps}}.
\end{equation*}
Choose $q \in (p,\infty)$. Then, by construction we observe,
\begin{equation*}
	\norm{(\eta_\phi^A)^{1/p} F_\phi^A}_{\Leb^p(\Omega)}^p \leq \eps, \qquad
	\norm{(\eta_\phi^A)^{1/p} G_\phi^A}_{\Leb^q(\Omega)}^q \leq c_1 T_\eps^{q-p}.
\end{equation*}
By \cref{Thm:RigidityEstimateJonesDomain}, we obtain rotations $R_\phi^A \in \SO(3)$ and decompositions $\dif{\phi}A(\placeholder)^{-1} - R_\phi^A = \hat{F}_\phi^A + \hat{G}_\phi^A$ in $\Leb^p+\Leb^q(\Omega, \R^{d\times d})$, such that
\begin{align*}
	\norm{(\eta_\phi^A)^{1/p}\hat{F}_\phi^A}_{\Leb^p(\Omega)}^p &\leq c_2 \norm{(\eta_\phi^A)^{1/p}F_\phi^A}_{\Leb^p(\Omega)}^p \leq c_2 \eps, \\
	\norm{(\eta_\phi^A)^{1/p}\hat{G}_\phi^A}_{\Leb^q(\Omega)}^q &\leq c_2 \norm{(\eta_\phi^A)^{1/p}G_\phi^A}_{\Leb^q(\Omega)}^q \leq c_3 T_\eps^{q - p}.
\end{align*}
By assumption the constant $c_2$ is independent of $\phi$ and $A$. Note that $(\eta_\phi^A)^{1/p}\left(\dif{\phi}A(\placeholder)^{-1} - R_\phi^A\right) = (\eta_\phi^A)^{1/p}\hat{F}_\phi^A + (\eta_\phi^A)^{1/p}\hat{G}_\phi^A$. It is not hard to show that these estimates imply uniform integrability of $\eta_\phi^A \abs{\dif{\phi}A(\placeholder)^{-1} - R_\phi^A}^p$, see \cite[Lem.~5.1]{ADD12} for a proof.
\end{proof}

Using \cref{Prop:uniform_integrability} and Korn's inequality \cref{Cor:KornWithLpNormOfu}, the proof of \cref{Prop:strong-convergenceLinearization} can be obtained analogously to \cite[Sect.~5]{ADD12}. Thus, we omit this proof here. Note that later we do carry out the proof of the strong convergences \cref{Prop:strong_conv} (c) and (d), since there we have the additional difficulty to deal with two-scale convergence. One may also follow and alter the proof there to establish a proof for \cref{Prop:strong-convergenceLinearization}.

\subsection{Homogenization; Proof of Theorem \ref{Thm:gamma_convergence} (2) and (5), (\ref{Eq:equi_coercivity_Ihe}), (\ref{Eq:equi_coercivity_Iline}) and Proposition \ref{Prop:strong_conv} (c) and (d)}
\label{Sec:ProofSimultaneous}

Our approach to obtain the homogenization results relies on the notion of \emph{two-scale convergence}, which was developed by Gabriel Ngutseng in \cite{Ngu89} and first used by Grégoire Allaire in \cite{All92} to treat such homogenization problems. Following \cite{Vis06,Vis07,MT07,CDG02} we appeal to a characterization of two-scale convergence via the \emph{periodic unfolding operator} $\mathcal{T}_\eps: \R^d \times Y \to \R^d$, given by
\begin{equation}\label{Eq:UnfoldingOperator}
	\mathcal{T}_\eps(x,y) := \eps\lfloor\tfrac{x}{\eps}\rfloor + \eps y,
\end{equation}
where $\eps > 0$ and $\lfloor z\rfloor \in \Z^d$ denotes the (vectorial) integer part of $z \in \R^d$, given by $\lfloor z\rfloor_i := \max \class{\xi \in \Z}{\xi_i \leq z_i}$, $i=1,\dots,d$. Moreover, we let $\mathcal{T}_\eps$ act point-wise on maps, i.e., for $u: \R^d \to \R^n$ we set $\mathcal{T}_\eps u := u \circ \mathcal{T}_\eps$. In combination with the unfolding operator, we extend maps defined on subsets of $\R^d$ to $\R^d$ by $0$.

\begin{definition}
Let $1 \leq p \leq \infty$. We say $(u_\eps) \subset \Leb^p(\Omega)$ weakly (resp.\ strongly) two-scale converges to $u \in \Leb^p(\Omega \times Y)$ and write $u_\eps \xrightharpoonup{2} u$ (resp. $u_\eps \xrightarrow{2} u$), if $(u_\eps)$ is bounded in $\Leb^p(\Omega)$ and $\mathcal{T}_\eps u_\eps \wto u$ (resp.\ $\mathcal{T}_\eps u_\eps \to u$) in $\Leb^p(\Omega \times Y)$.
\end{definition}

\begin{remark}~
\begin{enumerate}[(i)]
	\item Since $\mathcal{T}_\eps^{-1}(\Omega) \subset \class{(x,y) \in \R^d \times Y}{\dist(x,\Omega) \leq \eps\sqrt{d}}$ and $\abs{\partial\Omega} = 0$, boundedness of $(u_\eps)$ in $\Leb^p(\Omega)$ ensures that
	\begin{equation} \label{Eq:UnfoldingRestTermVanishes}
		\iint_{\R^d \setminus \Omega \times Y} \abs{\mathcal{T}_\eps u_\eps}^p \to 0 \qquad \text{as } \eps \to 0.
	\end{equation}
	Especially, from the isometry from $\Leb^1(\R^d)$ to $\Leb^1(\R^d \times Y)$ induced by the unfolding operator, see \cite[Lem.~1.1]{Vis06}, we obtain
	\begin{equation} \label{Eq:IsometryUnfolding}
		\lim_{\eps \to 0} \left(\int_\Omega u_\eps - \iint_{\Omega \times Y} \mathcal{T}_\eps u_\eps\right) = 0.
	\end{equation}

	\item Below, we also speak of two-scale convergence of a sequence $(u_h) \subset \Leb^p(\Omega)$, where we mean convergence of $\mathcal{T}_{\eps(h)}u_h$ for a subsequence $\eps: (0,\infty) \to (0,\infty)$. The precise meaning will be clear from the context.
\end{enumerate}
\end{remark}

With this technique, the proof of \cref{Thm:gamma_convergence} (2) is standard, see e.g.\ \cite{All92}. We sketch the proof here, since we shall need elements of it for the simultaneous limit. Equation \eqref{Eq:IsometryUnfolding} and the characterization of two-scale limits of derivatives, see \cite[Prop.~1.14]{All92}, motivates the definition of the \emph{two-scale $\Gamma$-limit} $\EnergyIlints: \SobH^1(\Omega, \R^d) \times \Leb^2(\Omega, \SobH^1_\mathrm{per}(Y,\R^d)) \to \R$ of $\EnergyIline{\eps}$, given by
\begin{equation} \label{Eq:Two-scale-Gamma-limit}
	\EnergyIlints(u,\varphi) := \iint_{\Omega \times Y} Q \left( y, (\dif{u}(x) + \dif{_y \varphi}(x,y))A(y)^{-1} - B(y) \right) \,\dx[y]\,\dx.
\end{equation}
Indeed, strong continuity and weak lower semi-continuity w.r.t.\ $\Leb^2$ of the quadratic energy functional, implies that we can understand $\EnergyIlints$ rigorously as a two-scale $\Gamma$-limit of $\EnergyIline{\eps}$ in the sense that given $(u,\varphi) \in \SobH^1_{\Gamma, g}(\Omega, \R^d) \times \Leb^2(\Omega, \SobH^1_\mathrm{per}(Y,\R^d))$, the following statements hold.
\begin{itemize}
	\item Suppose $u_\eps \wto u$ weakly in $\SobH^1_{\Gamma, g}(\Omega, \R^d)$ and $\dif{u_\eps} \xrightharpoonup{2} \dif{u} + \dif{_y \varphi}$ weakly two-scale in $\Leb^2(\Omega \times Y, \R^{d\times d})$. Then, $\liminf_{\eps \to 0} \EnergyIline{\eps}(u_\eps) \geq \EnergyIlinhom(u,\varphi)$.
	
	\item We find a sequence $(u_\eps) \subset \SobH^1_{\Gamma, g}(\Omega, \R^d)$ converging to $(u,\varphi)$ in the sense as above with $\lim_{\eps \to 0} \EnergyIline{\eps}(u_\eps) = \EnergyIlinhom(u,\varphi)$.
\end{itemize}
A construction for the recovery sequence can be found in \cite[Thm. 3.3.1 (3)]{NeuPHD}. In fact, we also use this construction in the proof of \cref{Thm:gamma_convergence} (5). Given this statement it is not hard to infer that $\EnergyIline{\eps}$ $\Gamma$-converges w.r.t.\ weak convergence in $\SobH^1_{\Gamma, g}(\Omega, \R^d)$ with the limit given by
\begin{equation}\label{Eq:homogenized_limit}
	\EnergyIlinhom[*](u) := \min\class{\EnergyIlints(u,\varphi)}{\varphi \in \Leb^2(\Omega, \SobH^1_{\mathrm{per}}(Y,\R^d))} , \quad u \in \SobH^1(\Omega, \R^d).
\end{equation}
Indeed, one can use the direct method to show that a minimizer exists. The argument uses the version of Korn's inequality given in \cref{Cor:KornPeriodic} to obtain a equi-coercivity statement and the weak lower semi-continuity of quadratic functionals. From the existence of a minimizer, the $\Gamma$-convergence follows immediately. Finally, testing the Euler-Lagrange equation for the minimizer with test functions of the form $v(x)n(y)$ with $v \in \Leb^2(\Omega)$ and $n \in \SobH^1_{\mathrm{per},0}(Y,\R^d)$, we obtain $\EnergyIlinhom[*] = \EnergyIlinhom$. Before we proceed with the proof of \cref{Thm:gamma_convergence} (5), we show the following lemma which is in the spirit of \cref{Lem:ExpansionPeriodicityCell,Lem:ExpansionLinearization}.

\begin{lemma} \label{Lem:ExpansionTwoScale}
Let $\eps: (0,\infty) \to (0,\infty)$ with $\lim_{h \to 0} \eps(h) = 0$. Let $(\Phi_h),(\Psi_h) \subset \Leb^2(\Omega, \R^{d\times d})$ bounded and $\Psi \in \Leb^2(\Omega, \Leb^2_\text{per}(Y,\R^{d\times d}))$, such that $\Psi_h \xrightarrow{2} \Psi$ in $\Leb^2(\Omega\times Y, \R^{d\times d})$. Then, there exist subsets $\Omega_h \subset \Omega$ with $\abs{\Omega \setminus \Omega_h} \to 0$, such that
\begin{equation*}
	\Bigg| \frac{1}{h^2} \int_{\Omega_h} W\left(\tfrac{x}{\eps(h)}, I + h\Phi_h(x) + h\Psi_h(x)\right) \,\dx
	- \int_{\Omega_h} Q\left(\tfrac{x}{\eps(h)}, \Phi_h(x) + \Psi(x, \tfrac{x}{\eps(h)})\right) \,\dx \Bigg| \to 0.
\end{equation*}
Moreover, if $(\Phi_h)$ is bounded in $\Leb^\infty(\Omega, \R^{d\times d})$ and $\lim_{h \to 0} \norm{h\Psi_h}_{\Leb^\infty(\Omega)} = 0$, then we may choose $\Omega_h = \Omega$.
\end{lemma}

\begin{proof}
Let $G_h := \Psi_h + \Phi_h$. As in \cref{Lem:ExpansionPeriodicityCell} we may estimate the term in question from above by a sum $\text{(I)}_h + \text{(II)}_h$, where
\begin{align*}
	\text{(I)}_h &:= \int_{\Omega_h} \abs{G_h(x)}^2 r\left(\tfrac{x}{\eps(h)},\abs{hG_h(x)}\right) \,\dx, \\
	\text{(II)}_h &:= \int_{\Omega_h} Q\left(\tfrac{x}{\eps(h)}, \Phi_h(x) + \Psi_h(x)\right) - Q\left(y, \Phi_h(x) + \Psi(x, \tfrac{x}{\eps(h)})\right) \,\dx.
\end{align*}
We estimate $\text{(II)}_h$ first. Since, $\Phi_h$ and $\Psi_h$ are bounded in $\Leb^2(\Omega, \R^{d\times d})$ and $Q$ satisfies \eqref{Eq:UniformBounds_Q}, we find from \eqref{Eq:IsometryUnfolding}, that
\begin{equation*}
	\abs{\text{(II)}_h - \iint_{\Omega \times Y} Q\left(y, \mathcal{T}_{\eps(h)}\big[\mathds{1}_{\Omega_h}(\Phi_h + \Psi_h)\big](x,y)\right) - Q\left(y, \mathcal{T}_{\eps(h)}\big[\mathds{1}_{\Omega_h}(\Phi_h + \Psi)\big](x,y)\right) \,\dx[y]\dx} \to 0.
\end{equation*}
Thus, since $Q$ is a quadratic form, we obtain $\text{(II)}_h \to 0$ for an arbitrary sequence $(\Omega_h)$ from the $\Leb^2$-boundedness of $\Phi_h$ and the two-scale convergence $(\Psi_h - \Psi) \xrightarrow{2} 0$ in $\Leb^2(\Omega \times Y, \R^{d\times d})$. It remains to treat the term $\text{(I)}_h$. As in \cref{Lem:ExpansionPeriodicityCell}, we set $\bar{r}_h(y) := r(y, h^{1/2})$ and $\rho_h := (\int_Y \bar{r}_h)^{1/2}$. By dominated convergence we observe $\rho_h \to 0$. We set
\begin{equation*}
	\Omega_h := \class{x \in \Omega}{\abs{G_h(x)} \leq h^{-1/2}} \cap \mathcal{T}_{\eps(h)} \Big( \class{(x,y) \in \R^d \times Y}{\bar{r}_h(y) \leq \rho_h} \Big).
\end{equation*}
Then, Markov's inequality and \eqref{Eq:IsometryUnfolding} yield $\abs{\Omega \setminus \Omega_h} \to 0$. The set $\Omega_h$ is defined, such that the $Y$-periodicity of $r$ implies $\mathds{1}_{\Omega_h}(x) r(\tfrac{x}{\eps(h)},\abs{hG_h(x)}) \leq \rho_h$. Hence, $\text{(I)}_h \leq \norm{G_h}_{\Leb^2(\Omega)}^2 \rho_h$ converges to $0$ as $h \to 0$. If $\Phi_h$ and $\Psi_h$ satisfy the uniform bounds, then again by \eqref{Eq:IsometryUnfolding},
\begin{equation*}
	\limsup_{h \to 0} \text{(I)}_h \leq c_1 \limsup_{h \to 0} \iint_{\Omega \times Y} (\norm{\Phi_h}_{\Leb^\infty(\Omega)}^2 + \abs{\mathcal{T}_{\eps(h)}\Psi_h(x,y)}^2) r(y, \norm{hG_h}_{\Leb^\infty(\Omega)}) \,\dx[y]\dx,
\end{equation*}
independently of the definition of $\Omega_h$. The right-hand side equals $0$ by dominated convergence (with strongly converging dominating sequence). Hence, we may choose $\Omega_h = \Omega$.
\end{proof}

\begin{proof}[Proof of \cref{Thm:gamma_convergence} (5)]
Let $\eps: (0,\infty) \to (0,\infty)$ with $\lim_{h \to 0} \eps(h) = 0$. We show that $\EnergyIhe{\eps(h)}$ $\Gamma$-converges w.r.t.\ weak convergence in $\SobH^1_{\Gamma, g}(\Omega, \R^d)$ to $\EnergyIlinhom$.
\medskip

\stepemph{Step 1 -- Recovery sequence}: Let $u \in \SobH^1_{\Gamma, g}(\Omega, \R^d)$. We find some $\varphi \in \Leb^2(\Omega, \SobH^1_{\mathrm{per},0}(Y,\R^d))$, such that $\EnergyIlinhom(u) = \EnergyIlinhom[*](u) = \EnergyIlints(u,\varphi)$. Via a density argument, we find sequences $(u_\delta) \subset \SobW^{1,\infty}(\Omega, \R^d) \cap \SobH^1_{\Gamma, g}(\Omega, \R^d)$ and $(\varphi_\delta) \subset \Cont^\infty_c(\Omega, \Cont^\infty_\mathrm{per}(Y,\R^d))$ such that
\begin{alignat*}{2}
		u_{\delta} &\to u &&\qquad \text{strongly in } \SobH^1(\Omega, \R^d), \\
		\varphi_\delta &\to \varphi &&\qquad\text{strongly in } \Leb^2(\Omega, \SobH^1_{\mathrm{per}}(Y,\R^d)).
\end{alignat*}
Define $u_{\delta, h}(x) := u_\delta(x) + \eps(h) \varphi_\delta(x,\tfrac{x}{\eps(h)})$. Then, for fixed $\delta > 0$, the sequence $(\dif{u_{\delta,h}})$ is bounded in $\Leb^\infty(\Omega, \R^{d\times d})$. Thus, the assumptions of \cref{Lem:ExpansionTwoScale}, including the uniform bounds, are satisfied with
\begin{equation*}
	\Phi_h(x) := \dif{u_{\delta,h}}(x)A(\tfrac{x}{\eps(h)})^{-1}, \quad
	\Psi_h(x) := -B_h(\tfrac{x}{\eps(h)}) - h\Phi_h(x)B_h(\tfrac{x}{\eps(h)}), \quad
	\Psi(x,y) := -B(y),
\end{equation*}
and we obtain $\EnergyIhe{\eps(h)}(u_{\delta, h}) - \EnergyIline{\eps(h)}(u_{\delta,h}) \to 0$. Now, continuity of the quadratic energy functional w.r.t.\ strong convergence in $\Leb^2$ applied twice along with \eqref{Eq:IsometryUnfolding}, implies
\begin{equation*}
	\lim_{\delta \to 0}\lim_{h \to 0} \EnergyIhe{\eps(h)}(u_{\delta, h}) = \lim_{\delta \to 0}\lim_{h \to 0} \EnergyIline{\eps(h)}(u_{\delta, h}) = \lim_{\delta \to 0} \EnergyIlints(u_\delta, \varphi_\delta) = \EnergyIlints(u, \varphi) = \EnergyIlinhom(u).
\end{equation*}
Especially, for fixed $\delta$, the sequence $(u_{\delta,h})$ is a recovery sequence for the two-scale $\Gamma$-convergence of $\EnergyIline{\eps(h)}$ introduced above. Thus, by Attouch's diagonalization lemma \cite[Lem.~1.15]{Att84}, we find a diagonal sequence $u_h := u_{h,\delta(h)}$ satisfying $\EnergyIhe{\eps(h)}(u_h) \to \EnergyIlinhom(u)$ and $u_h \to u$ in $\Leb^2(\Omega,R^{d\times d})$. Since $(u_h)$ is bounded in $\SobH^1(\Omega, \R^{d\times d})$, we conclude $u_h \wto u$.
\medskip

\stepemph{Step 2 -- Lower bound}: Let $u \in \SobH^1_{\Gamma, g}(\Omega, \R^d)$ and $(u_h) \subset \SobH^1_{\Gamma, g}(\Omega, \R^d)$ converging weakly to $u$ in $\SobH^1(\Omega, \R^d)$. By \cite[Prop.~1.14]{All92} we find some $\varphi \in \Leb^2(\Omega, \SobH^1_{\mathrm{per},0}(\Omega, \R^d))$ and a subsequence (not relabeled) such that
\begin{equation*}
	\liminf_{h \to 0} \EnergyIhe{\eps(h)}(u_h) = \lim_{h \to 0} \EnergyIhe{\eps(h)}(u_h), \quad 
	\dif{u_h} \xrightharpoonup{2} \dif{u} + \dif{_y \varphi} \text{ in } \Leb^2(\Omega \times Y, \R^{d\times d}).
\end{equation*}
Since $(\dif{u_h})$ is bounded in $\Leb^2(\Omega,\R^{d\times d})$, we may apply \cref{Lem:ExpansionTwoScale} with
\begin{equation*}
	\Phi_h(x) := \dif{u_h}(x)A(\tfrac{x}{\eps(h)})^{-1}, \quad
	\Psi_h(x) := -B_h(\tfrac{x}{\eps(h)}) - h\Phi_h(x)B_h(\tfrac{x}{\eps(h)}), \quad
	\Psi(x,y) := -B(y).
\end{equation*}
Thus, we find subsets $\Omega_h \subset \Omega$ with $\abs{\Omega \setminus \Omega_h} \to 0$, such that
\begin{align*}
	\liminf_{h \to 0} \EnergyIhe{\eps(h)}(u_h) 
	&\geq \liminf_{h \to 0} \frac{1}{h^2} \int_{\Omega_h} W\left(\tfrac{x}{\eps(h)}, (I + h\dif{u_h}(x)A(\tfrac{x}{\eps(h)})^{-1})(I - h B_h\big(\tfrac{x}{\eps(h)}\big))\right) \,\dx \\
	&= \liminf_{h \to 0} \int_\Omega Q\left(\tfrac{x}{\eps(h)}, \mathds{1}_{\Omega_h}(x)\big(\dif{u_h}(x)A(\tfrac{x}{\eps(h)})^{-1} - B(\tfrac{x}{\eps(h)})\big)\right) \,\dx \\
	&= \liminf_{h \to 0} \iint_{\Omega \times Y} Q\left(y, \mathcal{T}_{\eps(h)}\Big[\mathds{1}_{\Omega_h}\big(\dif{u_h}A(\tfrac{\placeholder}{\eps(h)})^{-1} - B(\tfrac{\placeholder}{\eps(h)})\big)\Big](x,y)\right) \,\dx[y]\dx.
\end{align*}
The inequality above is trivial, since $W$ is non-negative and the last equality follows from \eqref{Eq:IsometryUnfolding}. Then, since $\mathds{1}_{\Omega_h}$ is bounded in $\Leb^\infty(\Omega)$ and converges in $\Leb^2(\Omega)$ to $1$, we find that $\mathds{1}_{\Omega_h}(\dif{u_h}A(\tfrac{\placeholder}{\eps(h)})^{-1} - B(\tfrac{\placeholder}{\eps(h)}))$ converges weakly two-scale to $(\dif{u} + \dif{_y \varphi})A(\placeholder)^{-1} - B$ in $\Leb^2(\Omega \times Y, \R^{d\times d})$. Hence, lower semi-continuity of the quadratic functional w.r.t.\ weak convergence in $\Leb^2(\Omega\times Y)$ yields
\begin{equation*}
	\liminf_{h \to 0} \EnergyIhe{\eps(h)}(u_h) \geq \EnergyIlints(u,\varphi) \geq \EnergyIlinhom[*](u) = \EnergyIlinhom(u). \qedhere
\end{equation*}
\end{proof}

We proceed by proving the associated equi-coercivity statement \eqref{Eq:equi_coercivity_Ihe}. Note that this also implies \eqref{Eq:equi_coercivity_Iline}, since $\Gamma$-convergence of $\EnergyIhe{\eps}$ to $\EnergyIline{\eps}$ implies
\begin{equation*}
	\EnergyIline{\eps}(u) = \inf \class{\liminf_{h \to 0} \EnergyIhe{\eps}(u_h)}{u_h \wto u \text{ weakly in } \SobH^1_{\Gamma,g}(\Omega,\R^d)}.
\end{equation*}
Indeed, given \eqref{Eq:equi_coercivity_Ihe}, from this characterization one obtains \eqref{Eq:equi_coercivity_Iline} immediately from the weak lower semi-continuity of the norm in $\SobH^1(\Omega,\R^d)$.

\begin{proof}[Proof of \eqref{Eq:equi_coercivity_Ihe}]
Let $u \in \SobH^1_{\Gamma,g}(\Omega, \R^d)$ and $\eps,h > 0$ small enough. As in the proof of \cref{Thm:Equi-coercivityLinearization}, in view of \cref{Thm:RigidityEstimateJonesDomain}, we find a constant rotation $R \in \SO(d)$, such that
\begin{equation*}
	\int_{\Omega} \abs{I + h\dif{u}(x)A(\tfrac{x}{\eps})^{-1} - R}^2 \,\dx \leq c_1 \int_\Omega \dist^2(I + h\dif{u}(x)A(\tfrac{x}{\eps})^{-1}, \SO(d)) \,\dx.
\end{equation*}
Let $a$ denote the potential of $A$ as given in \ref{Property:SFJisDerivative}. The constant $c_1$ is independent of $\eps$, since the potential $a_\eps := \eps a(\frac{\placeholder}{\eps})$ of $A(\frac{\placeholder}{\eps})$ is a Bilipschitz map by \cref{Prop:SFJBilipschitz} with Bilipschitz constant independent of $\eps$. We conclude the claim by estimating $I - R$ as in \cref{Thm:Equi-coercivityLinearization}. Indeed, repeating the argument there but with $a$ replaced by $a_\eps$ and applied on the whole part of the boundary $\Gamma$ instead of $\Gamma \cap \partial \Omega_i$, we obtain
\begin{equation*}
	\abs{I - R}^2
	\leq c_2 \int_{\Omega} \abs{I + h\dif{u}(x)\dif{a_\eps}(x)^{-1} - R}^2\,\dx + c_2 h^2 \norm{g}_{\Leb^2(\Gamma, \mathcal{H}^{d-1})}^2.
\end{equation*}
Note that also $c_2$ is independent of $\eps$ by the uniformity of the estimate \cref{Cor:MatrixNorm} w.r.t.\ the Bilipschitz maps $a_\eps$. Now, combining these two estimates, we obtain
\begin{equation*}
	\int_{\Omega} \abs{\dif{u}(x)}^2 \,\dx \leq c_3 \int_\Omega \dist^2(I + h\dif{u}(x)A(\tfrac{x}{\eps})^{-1}, \SO(d)) \,\dx + c_3 \norm{g}_{\Leb^2(\Gamma, \mathcal{H}^{d-1})}^2.
\end{equation*}
From this we may conclude the claim by appealing to the Poincaré-Friedrich inequality and non-degeneracy \ref{Property:W_NonDegeneracy}. Indeed, for the non-degeneracy we can argue as in \cref{Prop:NonDegeneracy} to estimate the right-hand side by $c_4 (\EnergyIhe{\eps}(u) + 1)$.
\end{proof}

We finish by proving \cref{Prop:strong_conv} (c) and (d) which establishes strong convergence. In comparison to \cref{Prop:strong-convergenceLinearization} we can only expect to obtain strong two-scale convergence of the sequences, since even the recovery sequences constructed for \cref{Thm:gamma_convergence} (2) and (5) are only two-scale converging. An upgrade to strong two-scale convergence has been obtained in \cite[Sec.~7.5]{NeuPHD} in the context of homogenization for planar rods. We follow the procedure presented there and incorporate the arguments from \cite{ADD12} to show \cref{Prop:strong_conv} (d).

\begin{proof}[Proof of \cref{Prop:strong_conv} (c)]
Using the construction in the proof of \cref{Thm:gamma_convergence} (5), we find a recovery sequence $(u_\eps^*)$, satisfying $\EnergyIline{\eps}(u_\eps^*) \to \EnergyIlints(u,\varphi) = \EnergyIlinhom(u)$ and
\begin{align*}
	&u_\eps^* \to u &&\text{strongly in } \Leb^2(\Omega, \R^d), \\
	&\dif{u_\eps^*} \xrightarrow{2} \dif{u} + \dif{_y \varphi} &&\text{strongly two-scale in } \Leb^2(\Omega \times Y, \R^d).
\end{align*}
Since $u_\eps \wto u$ weakly in $\SobH^1(\Omega,\R^d)$, we infer from \cite[Prop. 1.14]{All92} that each subsequence admits a further subsequence $(u_{\eps''})$ with $\dif{u_{\eps''}} \xrightharpoonup{2} \dif{u} + \dif{_y \varphi''}$ in $\Leb^2(\Omega \times Y, \R^{d\times d})$, where $\varphi'' \in \Leb^2(\Omega, \SobH^1_{\mathrm{per}, 0}(Y, \R^d))$ a priori depends on the subsequence. But then the lower bound statement for the two-scale $\Gamma$-convergence of $\EnergyIline{\eps}$ implies
\begin{equation*}
	\EnergyIlints(u,\varphi) = \EnergyIlinhom(u) = \lim_{\eps \to 0} \EnergyIline{\eps}(u_\eps) = \liminf_{\eps'' \to 0} \EnergyIline{\eps''}(u_{\eps''}) \geq \EnergyIlints(u,\varphi'') \geq \EnergyIlints(u,\varphi).
\end{equation*}
Thus, $\varphi'' = \varphi$ by uniqueness of the minimizer of \eqref{Eq:homogenized_limit} and so the whole sequence $(u_\eps)$ satisfies $\dif{u_\eps} \xrightharpoonup{2} \dif{u} + \dif{_y \varphi}$ in $\Leb^2(\Omega \times Y, \R^{d\times d})$. Now, since $\EnergyIline{\eps}(u_\eps) - \EnergyIline{\eps}(u_\eps^*) \to 0$, we may refer to \eqref{Eq:QuadraticTrick} to show
\begin{multline*}
	\int_\Omega \abs{\sym(\dif{u_\eps}(x)A(\tfrac{x}{\eps})^{-1} - \dif{u_\eps^*}(x)A(\tfrac{x}{\eps})^{-1})}^2 \,\dx \leq c_1 \Big( \EnergyIline{\eps}(u_\eps) - \EnergyIline{\eps}(u_\eps^*)\\
	+ 2\int_\Omega (\dif{u_\eps^*}(x) - \dif{u_\eps}(x))A(\tfrac{x}{\eps})^{-1} : \mathbb{L}_Q(\tfrac{x}{\eps}) (\dif{u_\eps^*}(x)A(\tfrac{x}{\eps})^{-1} - B(\tfrac{x}{\eps}))  \,\dx \Big) \to 0.
\end{multline*}
Indeed, the last term converges to $0$ as it is a product of a weakly and a strongly two-scale converging sequence, see \cite[Prop.~1.4]{Vis06}. Now, Korn's inequality \cref{Cor:KornWithLpNormOfu} shows $\dif{u_\eps} - \dif{u_\eps^*} \to 0$ in $\Leb^2(\Omega, \R^{d\times d})$ and then the claim follows from the strong two-scale convergence $\dif{u_\eps^*} \xrightarrow{2} \dif{u} + \dif{_y \varphi}$.
\end{proof}

\begin{proof}[Proof of \cref{Prop:strong_conv} (d)]
Fix $\eps: (0, \infty) \to (0,\infty)$ with $\lim_{h \to 0} \eps(h) = 0$. As above it suffices to show that $\dif{u_h} - \dif{u_{\eps(h)}^*} \to 0$ in $\Leb^2(\Omega, \R^{d\times d})$, where $u_h := u_{\eps(h),h}$ and $u_\eps^*$ is as above.
\medskip

\stepemph{Step 1 -- Two-scale convergence}: As above we can show that $(u_h)$ satisfies $\dif{u_h} \xrightharpoonup{2} \dif{u} + \dif{_y \varphi}$ in $\Leb^2(\Omega \times Y, \R^{d \times d})$. Indeed, each subsequence of $(u_h)$ admits a further subsequence $(u_{h''})$ such that $\dif{u_{h''}} \xrightharpoonup{2} \dif{u} + \dif{_y \varphi''}$ and $\varphi'' = \varphi$ because of
\begin{equation*}
	\EnergyIlints(u,\varphi) = \EnergyIlinhom(u) = \lim_{h \to 0} \EnergyIhe{\eps(h)}(u_h) = \liminf_{h'' \to 0} \EnergyIhe[h'']{\eps(h'')}(u_{h''}) \geq \EnergyIlints(u,\varphi'').
\end{equation*}
The last inequality follows as shown in the proof of \cref{Thm:gamma_convergence} (5).
\medskip

\stepemph{Step 2 -- Quadratic trick}: Since $\dif{u_h}$ is bounded in $\Leb^2(\Omega, \R^{d\times d})$, we may apply \cref{Lem:ExpansionTwoScale} and find sets $\Omega_h \subset \Omega$ with $\abs{\Omega \setminus \Omega_h} \to 0$, such that $\EnergyIhe{\eps(h)}(u_h)$ admits a decomposition
\begin{align*}
	\EnergyIhe{\eps(h)}(u_h) &= \text{(I)}_h + \text{(II)}_h + \text{(III)}_h, \qquad \text{where,}\\
	\text{(I)}_h &:= \int_\Omega Q\left(\tfrac{x}{\eps(h)}, \mathds{1}_{\Omega_h}(x)( \dif{u_h}(x)A(\tfrac{x}{\eps(h)})^{-1} - B(\tfrac{x}{\eps(h)}))\right)\,\dx, \\
	\text{(II)}_h &:= \frac{1}{h^2}\int_{\Omega \setminus\Omega_h} W\left(\tfrac{x}{\eps(h)}, (I + h\dif{u_h}(x)A(\tfrac{x}{\eps(h)})^{-1})(I - h B_h(\tfrac{x}{\eps(h)}))\right)\,\dx,
\end{align*}
and the remainder $\text{(III)}_h$ converges to $0$. Without loss of generality, we may assume that $\{x \in \Omega \mid |\dif{u_h}(x)| \leq h^{-1/2}\} \subset \Omega_h$. Furthermore, we have $\liminf_{h \to 0} \text{(I)}_h \geq \EnergyIlinhom(u)$ as we have shown in Step 2 of the proof of \cref{Thm:gamma_convergence} (5). We show that $\text{(II)}_h \to 0$. Indeed, since $\text{(II)}_h$ is non-negative, this follows from
\begin{equation*}
	\EnergyIlinhom(u) = \lim_{h \to 0} \EnergyIhe{\eps(h)}(u_h) \geq \liminf_{h \to 0} \text{(I)}_h + \limsup_{h \to 0} \text{(II)}_h \geq \EnergyIlinhom(u) + \limsup_{h \to 0} \text{(II)}_h.
\end{equation*}
Hence, we obtain $\lim_{h \to 0} \text{(I)}_h = \lim_{h \to 0} \EnergyIhe{\eps(h)}(u_h) = \EnergyIlinhom(u)$ and we may proceed as for \cref{Prop:strong_conv} (c) to show that this implies
\begin{equation*}
	\sym \left((\mathds{1}_{\Omega_h}\dif{u_h} - \dif{u_{\eps(h)}^*})A(\tfrac{\placeholder}{\eps(h)})^{-1}\right) \to 0 \qquad \text{strongly in } \Leb^2(\Omega, \R^{d\times d}).
\end{equation*}

\stepemph{Step 3 -- Convergence of the rest}: Let $1 \leq p < 2$. By Hölder's inequality, we have
\begin{equation*}
	\int_{\Omega} \mathds{1}_{\Omega \setminus \Omega_h}(x)\abs{\dif{u_h}(x)}^p \,\dx \leq \abs{\Omega \setminus \Omega_h}^{\frac{2-p}{2}} \left(\int_\Omega \abs{\dif{u_h}(x)}^2 \,\dx\right)^{\frac{p}{2}} \to 0.
\end{equation*}
Combined with Step 2 and Korn's inequality \cref{Cor:KornWithLpNormOfu} this yields $\dif{u_h} - \dif{u_{\eps(h)}^*} \to 0$ in $\Leb^p(\Omega, \R^{d\times d})$. Especially, $\dif{u_h} - \dif{u_{\eps(h)}^*}$ converges to $0$ in measure.
\medskip

\stepemph{Step 4 -- Strong convergence in $\Leb^2$}: We seek to apply Vitali's convergence theorem to the sequence $(\dif{u_h} - \dif u_{\eps(h)}^*)$. Since we already showed convergence in measure, it remains to establish uniform integrability. Since $(\dif u_{\eps(h)}^*)$ is strongly two-scale converging in $\Leb^2(\Omega \times Y, \R^{d\times d})$, the sequence $(|\dif u_{\eps(h)}^*|^2)$ is uniformly integrable by \eqref{Eq:IsometryUnfolding}. Thus, it suffices to show $(\dif u_h)$ is uniformly integrable. For this, we want to use \cref{Prop:uniform_integrability}. First, arguing as in \cref{Prop:NonDegeneracy}, we get
\begin{equation*}
	\frac{1}{h^2}\int_{\Omega} \mathds{1}_{\Omega \setminus\Omega_h} \dist^2\big(I + h\dif{u_h}(x)A(\tfrac{x}{\eps(h)})^{-1}, \SO(d)\big) \,\dx \leq c_1 \left( \text{(II)}_h + \int_{\Omega \setminus \Omega_h} \abs{B_h(\tfrac{x}{\eps(h)})}^2 \,\dx \right).
\end{equation*}
The right hand-side converges to $0$, since $\mathds{1}_{\Omega \setminus \Omega_h}B_h(\tfrac{\placeholder}{\eps(h)})$ converges strongly two-scale to $0$ in $\Leb^2(\Omega \times Y, \R^{d\times d})$. Moreover, since $\abs{h\dif{u_h}} \leq \sqrt{h}$ on $\Omega_h$, the Taylor expansion $\dist(I+G, \SO(d)) = \abs{\sym G} + \Oclass(\abs{G}^2)$ yields,
\begin{equation*}
	\tfrac{1}{h^2} \mathds{1}_{\Omega_h}\dist^2\big(I + h\dif{u_h}A(\tfrac{\placeholder}{\eps(h)})^{-1}, \SO(d)\big) 
	\leq c_1 \left( \mathds{1}_{\Omega_h} \abs{\sym(\dif{u_h}A(\tfrac{\placeholder}{\eps(h)})^{-1})}^2 + h \abs{\dif{u_h}}^2 \right).
\end{equation*}
Since by Step 2 also here the right-hand side is strongly two-scale converging, the sequence $(\frac{1}{h^2}\dist^2(I + h\dif{u_h}A(\tfrac{\placeholder}{\eps(h)})^{-1}, \SO(d)))$ is uniformly integrable. Hence, by \cref{Prop:uniform_integrability}, we find rotations $(R_h) \subset \SO(d)$, such that $\frac{1}{h^2} \abs{I + h\dif{u_h}A(\tfrac{\placeholder}{\eps(h)})^{-1} - R_h}^2$ is uniformly integrable. Finally,
\begin{equation*}
	\abs{\frac{I - R_h}{h}}^2 \leq 2 \left(\frac{1}{h^2}\fint_\Omega \abs{I + h\dif{u_h}A(\tfrac{\placeholder}{\eps(h)})^{-1} - R_h}^2 + \fint_\Omega \abs{\dif{u_h}A(\tfrac{\placeholder}{\eps(h)})^{-1}}^2 \right)
\end{equation*}
is bounded. Hence, we conclude that $(\abs{\dif{u_h}}^2)$ is uniformly integrable and an application of Vitali's convergence theorem yields the claim.
\end{proof}

\appendix
\section{Formulas of the stress-free joints}\label{Sec:formulas-of-the-stress-free-joints}

In this section we provide the explicit formulas used in \cref{Fig:StressFreeJoints}. \cref{Fig:StressFreeJoints:laminate} is a laminate given by
\begin{equation*}
	A_{\rm a}(y) = \begin{cases} A_{\rm a}^1 & \text{if } y_1 \in [0,\tfrac{1}{2}), \\
	A_{\rm a}^2 & \text{if } y_1 \in [\tfrac{1}{2}, 1). \end{cases}
\end{equation*}
To ensure that $A_{\rm a}$ is a stress-free joint, by \eqref{Eq:Rank_one_compatiblity} it is necessary and sufficient to satisfy the rank-one compatibility condition $A_{\rm a}^2 = A_{\rm a}^1 + c \otimes e_1$ for some $c \in \R^3$ and have $\det A_{\rm a}^1, \det A_{\rm a}^2 > 0$. In \cref{Fig:StressFreeJoints:laminate} we used
\begin{equation*}
	A_{\rm a}^1 = \begin{pmatrix} 1 & 1 & 0 \\ 0 & 1 & 0 \\ 0 & 0 & 1 \end{pmatrix}, \qquad
	A_{\rm a}^2 = \begin{pmatrix} 2 & 1 & 0 \\ 0 & 1 & 0 \\ 0 & 0 & 1 \end{pmatrix}.
\end{equation*}

The stress-free joint in \cref{Fig:StressFreeJoints:2Dlaminate} is given by
\begin{equation*}
	A_{\rm b}(y) = 
	\begin{cases} 
		A_{\rm b}^1 & \text{if } (y_1,y_2) \in [0,\tfrac{1}{2})^2, \\
		A_{\rm b}^2 & \text{if } (y_1,y_2) \in [\tfrac{1}{2},1)\times [0,\tfrac{1}{2}), \\
		A_{\rm b}^3 & \text{if } (y_1,y_2) \in [\tfrac{1}{2},1)^2, \\
		A_{\rm b}^4 & \text{if } (y_1,y_2) \in [0,\tfrac{1}{2}) \times [\tfrac{1}{2},1).
	\end{cases}
\end{equation*}
We choose $A_{\rm b}^1$, $A_{\rm b}^2$ and $A_{\rm b}^3$ with positive determinant, such that they satisfy the appropriate rank-one compatibility conditions. From this $A_{\rm b}^4$ is uniquely determined. In \cref{Fig:StressFreeJoints:2Dlaminate} we used
\begin{align*}
	A_{\rm b}^1 &= \begin{pmatrix} 1 & \tfrac{1}{2} & 0 \\ 0 & 1 & 0 \\ 0 & 0 & 1 \end{pmatrix}, &
	A_{\rm b}^2 &= \begin{pmatrix} 1 & \tfrac{1}{2} & 0 \\ \tfrac{1}{2} & 1 & 0 \\ 0 & 0 & 1 \end{pmatrix}, &
	A_{\rm b}^3 &= \begin{pmatrix} 1 & -\tfrac{1}{2} & 0 \\ \tfrac{1}{2} & 1 & 0 \\ 0 & 0 & 1 \end{pmatrix}, &
	A_{\rm b}^4 &= \begin{pmatrix} 1 & -\tfrac{1}{2} & 0 \\ 0 & 1 & 0 \\ 0 & 0 & 1 \end{pmatrix}.
\end{align*}

In \cref{Fig:StressFreeJoints:smoothSFJ}, by \cref{Prop:SFJBilipschitz} and \cref{Lem:periodic_derivative} it suffices to choose $\bar{A} \in \R^{3\times 3}$ with $\det \bar{A} > 0$, such that $A_{\rm d} := \bar{A} + \dif{\bar{a}_{\rm d}}$ satisfies $\det A_{\rm d} > 0$ in $Y$. We used $\bar{A} := I$ and
\begin{equation*}
	\bar{a}_{\rm d}(y) = \tfrac{1}{20} \sum_{i=1}^3 \sin(2\pi y_i) \left(\begin{smallmatrix} 1 \\ 1 \\ 1 \end{smallmatrix}\right), \quad y \in Y.
\end{equation*}

The most interesting case is \cref{Fig:StressFreeJoints:periodicSFJ}. For the construction we consider the domains
\begin{align*}
	Y_1 &:= \class{y \in Y}{y_1 \leq \abs{y_2 - \tfrac{1}{2}}}, &
	Y_2 &:= \class{y \in Y}{\tfrac{1}{2} \geq y_2 > \abs{y_1 - \tfrac{1}{2}}}, \\
	Y_3 &:= \class{y \in Y}{1 - \abs{y_1 - \tfrac{1}{2}} > y_2 > \tfrac{1}{2}}, &
	Y_4 &:= \class{y \in Y}{1 - \abs{y_2 - \tfrac{1}{2}} \leq y_1},
\end{align*}
with normals $n_i$ at $\partial Y_1 \cap \partial Y_i$ (periodically continued) given by,
\begin{equation*}
	n_2 := \left(\tfrac{1}{\sqrt{2}}, \tfrac{1}{\sqrt{2}}, 0\right)^T, \qquad
	n_3 := \left(\tfrac{1}{\sqrt{2}}, -\tfrac{1}{\sqrt{2}},0\right)^T, \qquad
	n_4 := e_1.
\end{equation*}
For $\alpha \in [0,1]$ and $c \in \R^3$ with $c_2^{-2} = c_1^{-2}\alpha^2 + c_3^{-2}(1 - \alpha^2)$, consider the orthonormal basis
\begin{equation*}
	b_1 := \alpha e_1 + \sqrt{1 - \alpha^2} e_3, \qquad
	b_2 := e_2, \qquad
	b_3 := \alpha e_3 - \sqrt{1 - \alpha^2} e_1.
\end{equation*}
and the matrices
\begin{align*}
	A_1 &:= \sum_{i=1}^3 c_i b_i \otimes b_i, &
	A_2 &:= A_1\left(I +  \frac{2\alpha(c_2^2 c_1^{-2} - 1)}{\sqrt{2(1 - \alpha^2)}}e_3 \otimes n_2\right), \\
	A_3 &:= A_1\left(I + \frac{2\alpha(c_2^2 c_1^{-2} - 1)}{\sqrt{2(1 - \alpha^2)}}e_3 \otimes n_3\right), &
	A_4 &:= A_1\left(I + \frac{2\alpha(c_2^2 c_1^{-2} - 1)}{\sqrt{(1 - \alpha^2)}}e_3 \otimes n_4\right), \\
\end{align*}
We set $A_{\text{(c)}}(y) := A_i$, if $y \in Y_i$, $i=1,\dots,4$. One can check that with these definitions the matrices have positive determinant and all necessary rank-one compatibility conditions are satisfied. Moreover, the matrices satisfy $A_i = R_i A_1 Q_i$, for rotations $R_i, Q_i \in \SO(3)$, $i=2,\dots,4$. Thus, the stress-free joint depicts the joint of multiple bodies of the same material in different orientations, cf.\ \cite[Sect.~2]{Eri83}. This stress-free joint was found using the theory of \cite{Eri83}. In \cref{Fig:StressFreeJoints:periodicSFJ} we used the parameters $\alpha = \frac{1}{\sqrt{2}}$, $c = (\sqrt{\frac{3}{4}},1,\sqrt{\frac{3}{2}})^T$.

\section{Correctors for isotropic laminates}\label{Sec:correctors-for-isotropic-laminates}

\begin{proof}[Proof of \cref{Prop:CorrectorsIsotropicLaminate}]
We state the general procedure, how one can find the corrector $\hat{\varphi}_{G_i}$. The corrector $\hat{\varphi}_{G_i}$ is characterized as the unique solution to the Euler-Lagrange equation
\begin{equation*}
	\int_Y\big(G_i + \dif{\hat{\varphi}_{G_i}}(y) \bar{A}^{-1}\big) : \mathbb{L}_{\hat{Q}}(y) \dif{\delta\varphi}(y) \bar{A}^{-1} = 0, \quad \text{for all } \delta\varphi \in \SobH^1_\mathrm{per}(Y,\R^3).
\end{equation*}
Since we consider laminates, it is a good Ansatz to assume, $\hat{\varphi}_{G_i}$ depends only on $y_1$, i.e., $\hat{\varphi}_{G_i}(y) = \phi_{G_i}(y_1)$ for some $\phi_{G_i} \in \SobH^1_\mathrm{per}([0,1), \R^3)$. Indeed, in this case we have $\dif{\hat{\varphi}_{G_i}}(y) = \phi_{G_i}'(y_1)e_1^T$ and it is not hard to show that the Euler-Lagrange equation presented above is equivalent to the one restricted to the space $\class{\varphi}{\exists \phi \in \SobH^1_\mathrm{per}([0,1), \R^3) \text{ with } \varphi(y) = \phi(y_1) \text{ for all } y \in \R^3}$, using that the integrals w.r.t.~$y_2$ and $y_3$ vanish. For maps in this space we obtain the decomposition $\sym(\dif{\varphi}\bar{A}^{-1}) = \sum_{j=1}^6 a_j G_j$ with
\begin{align*}
	a_1 &= \alpha_1 \phi_1' + \alpha_2 \phi_2' + \alpha_3 \phi_3',& a_2 &= \alpha_1 \phi_1' - \tfrac{1}{2} \alpha_2 \phi_2' - \tfrac{1}{2} \alpha_3 \phi_3',& a_3 &= \alpha_3 \phi_3' - \alpha_2 \phi_2',&&&& \\
	a_4 &= \alpha_1 \phi_2' + \alpha_2 \phi_1',& a_5 &= \alpha_1 \phi_3' + \alpha_3 \phi_1',& a_6 &= \alpha_2 \phi_3' + \alpha_3 \phi_2',
\end{align*}
where $\alpha := \bar{A}^{-T}e_1$. Since isotropy yields $G_j : \mathbb{L}_{\hat{Q}} G_k = 0$ for $k \neq j$, we get
\begin{equation*}
	\hat{Q}(y, G_i + \dif{\varphi}(y)\bar{A}^{-1}) = \sum_{j=1}^6 (a_j(y_1) + \delta_{ij})^2\, \hat{Q}(y, G_j),
\end{equation*}
and the Euler-Lagrange equation reduces to
\begin{align*}
	0 = \int_0^1 &(\hat{\lambda} + \tfrac{2}{3}\hat{\mu})(\alpha_1 {\phi_{G_i}'}_1 + \alpha_2 {\phi_{G_i}'}_2 + \alpha_3 {\phi_{G_i}'}_3 + \delta_{i1})(\alpha_1 \delta\phi_1' + \alpha_2 \delta\phi_2' + \alpha_3 \delta\phi_3') \\
	+& \tfrac{4}{3}\hat{\mu}(\alpha_1 {\phi_{G_i}'}_1 - \tfrac{1}{2} \alpha_2 {\phi_{G_i}'}_2 - \tfrac{1}{2} \alpha_3 {\phi_{G_i}'}_3 + \delta_{i2})(\alpha_1 \delta\phi_1' - \tfrac{1}{2} \alpha_2 \delta\phi_2' - \tfrac{1}{2} \alpha_3 \delta\phi_3') \\
	+& \hat{\mu}(\alpha_3 {\phi_{G_i}'}_3 - \alpha_2 {\phi_{G_i}'}_2 + \delta_{i3})(\alpha_3 \delta\phi_3' - \alpha_2 \delta\phi_2') \\
	+& \hat{\mu} (\alpha_1 {\phi_{G_i}'}_2 + \alpha_2 {\phi_{G_i}'}_1 + \delta_{i4})(\alpha_1 \delta\phi_2' + \alpha_2 \delta\phi_1') \\
	+& \hat{\mu} (\alpha_1 {\phi_{G_i}'}_3 + \alpha_3 {\phi_{G_i}'}_1 + \delta_{i5})(\alpha_1 \delta\phi_3' + \alpha_3 \delta\phi_1') \\
	+& \hat{\mu} (\alpha_2 {\phi_{G_i}'}_3 + \alpha_3 {\phi_{G_i}'}_2 + \delta_{i6})(\alpha_2 \delta\phi_3' + \alpha_3 \delta\phi_2') \\
	= \int_0^1 & \delta\phi_1' \begin{aligned}[t] \big[
		&\alpha_1(\hat{\lambda} + \tfrac{2}{3}\hat{\mu})(\alpha_1 {\phi_{G_i}'}_1 + \alpha_2 {\phi_{G_i}'}_2 + \alpha_3 {\phi_{G_i}'}_3 + \delta_{i1}) \\
		&+ \tfrac{4}{3} \alpha_1 \hat{\mu}(\alpha_1 {\phi_{G_i}'}_1 - \tfrac{1}{2} \alpha_2 {\phi_{G_i}'}_2 - \tfrac{1}{2} \alpha_3 {\phi_{G_i}'}_3 + \delta_{i2}) \\
		&+ \alpha_2 \hat{\mu} (\alpha_1 {\phi_{G_i}'}_2 + \alpha_2 {\phi_{G_i}'}_1 + \delta_{i4}) + \alpha_3 \hat{\mu} (\alpha_1 {\phi_{G_i}'}_3 + \alpha_3 {\phi_{G_i}'}_1 + \delta_{i5})
	\big] \end{aligned} \\
	+& \delta\phi_2' \begin{aligned}[t] \big[
		&\alpha_2(\hat{\lambda} + \tfrac{2}{3}\hat{\mu})(\alpha_1 {\phi_{G_i}'}_1 + \alpha_2 {\phi_{G_i}'}_2 + \alpha_3 {\phi_{G_i}'}_3 + \delta_{i1}) \\
		&- \tfrac{2}{3} \alpha_2 \hat{\mu}(\alpha_1 {\phi_{G_i}'}_1 - \tfrac{1}{2} \alpha_2 {\phi_{G_i}'}_2 - \tfrac{1}{2} \alpha_3 {\phi_{G_i}'}_3 + \delta_{i2}) \\
		&- \alpha_2 \hat{\mu}(\alpha_3 {\phi_{G_i}'}_3 - \alpha_2 {\phi_{G_i}'}_2 + \delta_{i3})
		+ \alpha_1 \hat{\mu} (\alpha_1 {\phi_{G_i}'}_2 + \alpha_2 {\phi_{G_i}'}_1 + \delta_{i4}) \\
		&+ \alpha_3 \hat{\mu} (\alpha_2 {\phi_{G_i}'}_3 + \alpha_3 {\phi_{G_i}'}_2 + \delta_{i6})
	\big] \end{aligned} \\
	+& \delta\phi_3' \begin{aligned}[t] \big[
		&\alpha_3(\hat{\lambda} + \tfrac{2}{3}\hat{\mu})(\alpha_1 {\phi_{G_i}'}_1 + \alpha_2 {\phi_{G_i}'}_2 + \alpha_3 {\phi_{G_i}'}_3 + \delta_{i1}) \\
		&- \tfrac{2}{3} \alpha_3 \hat{\mu}(\alpha_1 {\phi_{G_i}'}_1 - \tfrac{1}{2} \alpha_2 {\phi_{G_i}'}_2 - \tfrac{1}{2} \alpha_3 {\phi_{G_i}'}_3 + \delta_{i2}) \\
		&+ \alpha_3 \hat{\mu}(\alpha_3 {\phi_{G_i}'}_3 - \alpha_2 {\phi_{G_i}'}_2 + \delta_{i3})
		+ \alpha_1 \hat{\mu} (\alpha_1 {\phi_{G_i}'}_3 + \alpha_3 {\phi_{G_i}'}_1 + \delta_{i5}) \\
		&+ \alpha_2 \hat{\mu} (\alpha_2 {\phi_{G_i}'}_3 + \alpha_3 {\phi_{G_i}'}_2 + \delta_{i6})
	\big] \end{aligned},
\end{align*}
for all $\delta\phi \in \SobH^1_\mathrm{per}([0,1), \R^3)$. It follows from a variant of the fundamental lemma of the calculus of variations that then necessarily the factors after $\delta\phi_j'$ are constant (cf.~\cite[Section 1.4]{MW89}). This yields a system of linear equations for ${\phi_{G_i}'}_j$ given by
\begin{gather*}
	\begin{pmatrix}
		\alpha_1^2 (\hat{\lambda} + 2\hat{\mu}) + (\alpha_2^2 + \alpha_3^2)\hat{\mu} & 
		\alpha_1\alpha_2 (\hat{\lambda} + \hat{\mu}) &
		\alpha_1\alpha_3 (\hat{\lambda} + \hat{\mu}) \\
		\alpha_1\alpha_2 (\hat{\lambda} + \hat{\mu}) &
		\alpha_2^2 (\hat{\lambda} + 2\hat{\mu}) + (\alpha_1^2 + \alpha_3^2)\hat{\mu} &
		\alpha_2\alpha_3 (\hat{\lambda} + \hat{\mu}) \\
		\alpha_1\alpha_3 (\hat{\lambda} + \hat{\mu}) &
		\alpha_2\alpha_3 (\hat{\lambda} + \hat{\mu}) &
		\alpha_3^2 (\hat{\lambda} + 2\hat{\mu}) + (\alpha_1^2 + \alpha_2^2)\hat{\mu}
	\end{pmatrix} 
	\begin{pmatrix}
		{\phi_{G_i}'}_1 \\ {\phi_{G_i}'}_2 \\ {\phi_{G_i}'}_3
	\end{pmatrix} \\
	= \begin{pmatrix}
		c_1 \\ c_2 \\ c_3
	\end{pmatrix} - 
	\begin{pmatrix}
		\alpha_1 (\hat{\lambda} - \tfrac{2}{3}\hat{\mu}) \delta_{i1} +  \tfrac{4}{3}\alpha_1\hat{\mu}\delta_{i2} + \alpha_2\hat{\mu}\delta_{i4} + \alpha_3\hat{\mu}\delta_{i5}\\ 
		\alpha_2 (\hat{\lambda} - \tfrac{2}{3}\hat{\mu}) \delta_{i1} - \tfrac{2}{3}\alpha_2 \hat{\mu}\delta_{i2} - \alpha_2 \hat{\mu}\delta_{i3} + \alpha_1 \hat{\mu}\delta_{i4} + \alpha_3\hat{\mu}\delta_{i6} \\ 
		\alpha_3 (\hat{\lambda} - \tfrac{2}{3}\hat{\mu}) \delta_{i1} - \tfrac{2}{3}\alpha_3 \hat{\mu}\delta_{i2} + \alpha_3 \hat{\mu}\delta_{i3} + \alpha_1 \hat{\mu}\delta_{i5} - \alpha_2\hat{\mu}\delta_{i6}
	\end{pmatrix} \\
	\Leftrightarrow: C\phi' = c - r,
\end{gather*}
for some constant $c := (c_1, c_2, c_3)^T \in \R^3$. Note that this constant is not arbitrary, but can be calculated from $\phi'$ as we shall see below. The matrix $C$ is of the form
\begin{equation*}
	C = (\hat{\lambda} + \hat{\mu}) \alpha\alpha^T + \abs{\alpha}^2\hat{\mu} I = \abs{\alpha}^2 \hat{\mu} \left( \frac{\hat{\lambda} + \hat{\mu}}{\abs{\alpha}^2 \hat{\mu}} \alpha\alpha^T + I\right)
\end{equation*}
The inverses of matrizes of such a form have been calculated in \cite{Mil81}. Indeed, $C$ is invertible with inverse
\begin{equation*}
	C^{-1} = \frac{1}{\abs{\alpha}^2 \hat{\mu}} \left( I - \frac{\hat{\lambda} + \hat{\mu}}{\abs{\alpha}^2 \hat{M}}\right).
\end{equation*}
The constant $c$ can be calculated as follows. Since $\phi$ is periodic, we obtain
\begin{equation*}
	0 = \int_0^1 \phi'(t) \,\dx[t] = \mean{\phi'} = \mean{C^{-1}} c - \mean{C^{-1}r}.
\end{equation*}
Hence, $c = \mean{C^{-1}}^{-1} \mean{C^{-1}r}$. The inverse of $\mean{C^{-1}}$ can be calculated similarly to the inverse of $C$ and is given by
\begin{equation*}
	\mean{C^{-1}}^{-1} = \abs{\alpha}^2 \meanHarm{\hat{\mu}} \left(I + \left( \frac{\meanHarm{\hat{M}}}{\meanHarm{\hat{\mu}}} - 1\right)\abs{\alpha}^2 \alpha\alpha^T\right).
\end{equation*}
The claim now follows from calculating $\phi' = C^{-1} \left(\mean{C^{-1}}^{-1} \mean{C^{-1}r} - r\right)$.
\end{proof}

\section{Mixed growth estimates}\label{Sec:mixed-growth-estimates}

In this section we sketch the proof of \cref{Lem:ReflectionOfPolynomials} and provide a mixed growth version of the Poincaré-Wirtinger inequality.

\begin{proof}[Proof of \cref{Lem:ReflectionOfPolynomials}]
For $p = q$ (i.e.\ with decompositions) the lemma is \cite[Lem.~2.1]{Jon81}. We have to check that this lemma also holds in the mixed growth sense. For this we can use the equivalence of the norm $\norm{\placeholder}_{\Leb^p}$ and $\norm{\placeholder}_{\Leb^q}$ on the finite dimensional vector space of polynomials of degree at most $m$. In fact it is easy to show that we can use the following choices
\begin{equation*}
	F_{P|_F} = \begin{cases}
		P & \text{if } \norm{G_{P|_E}}_{\Leb^q(E)} \leq \norm{F_{P|_E}}_{\Leb^q(E)}, \\
		0 & \text{else},
	\end{cases} \qquad
	G_{P|_F} = \begin{cases}
		0 & \text{if } \norm{G_{P|_E}}_{\Leb^q(E)} \leq \norm{F_{P|_E}}_{\Leb^q(E)}, \\
		P & \text{else}.
	\end{cases}
\end{equation*}
\end{proof}

\begin{proposition}[Poincaré-Wirtinger inequality] \label{Prop:Poincare_interpol}
Let $Q \subset \R^d$ be a finite union of overlapping cubes, $1 \leq p \leq q \leq \infty$, $v \in \SobW^{1,p}(Q)$ and a decomposition $\dif{v} = F_{\dif{v}} + G_{\dif{v}}$ in $\Leb^p+\Leb^q(Q,\R^d)$. Then, we find a decomposition $v - \fint_Q v = F_{v - \fint_Q v} + G_{v - \fint_Q v}$ in $\Leb^p+\Leb^q(Q)$ with 
\begin{equation}
\begin{aligned}
	\norm{F_{v - \fint_Q v}}_{\Leb^p(Q)} &\leq c\, \operatorname{diam}(Q) \norm{F_{\dif{v}}}_{\Leb^p(Q)}, \\
	\norm{G_{v - \fint_Q v}}_{\Leb^q(Q)} &\leq c\, \operatorname{diam}(Q) \norm{G_{\dif{v}}}_{\Leb^q(Q)}.
\end{aligned}
\end{equation}
for some constant $c > 0$, which is invariant under scaling and translation of $Q$.
\end{proposition}

\begin{proof} The proof is analogous to \cite[Thm.~2.1]{CDM14}. It suffices to show that the statement holds in the case where $Q$ is a cube. Then, we can extend the statement to finite unions of overlapping cubes as in \cref{Rem:Properties_rigidity} (iii). Note that, throughout the proof, constants may depend on the cube $Q$. A posteriori, a standard scaling argument and translation of the domain shows that the constant scales as presented with $Q$.
\medskip

\stepemph{Step 1 -- Extension}: By standard reflection techniques, we can extend $v$ to some smooth domain $Q^+ \supset\supset Q$, such that $v \in \SobW^{1,p}(Q^+)$ and $\dif v = F_{\dif{v}} + G_{\dif{v}}$ in $\Leb^p+\Leb^q(Q^+,\R^d)$, for suitable extensions of $F_{\dif{v}}$ and $G_{\dif{v}}$ with
\begin{align*}
	\norm{F_{\dif{v}}}_{\Leb^p(Q^+)} &\leq c_1 \norm{F_{\dif{v}}}_{\Leb^p(Q)}, \\ \norm{G_{\dif{v}}}_{\Leb^q(Q^+)} &\leq c_1 \norm{G_{\dif{v}}}_{\Leb^q(Q)}.
\end{align*}
For example the extension presented in \cite[Thm.~5.1]{CDM14} can be adapted to show this.
\medskip

\stepemph{Step 2 -- Case $\norm{G_{\dif{v}}}_{\Leb^q(Q)} \leq \norm{F_{\dif{v}}}_{\Leb^p(Q)}$}: Suppose $\norm{G_{\dif{v}}}_{\Leb^q(Q)} \leq \norm{F_{\dif{v}}}_{\Leb^p(Q)}$. Then, the Poincaré-Wirtinger inequality implies
\begin{equation*}
	\norm{v - \textstyle \fint_Q v}_{\Leb^p(Q)} \leq c_1 \norm{\dif v}_{\Leb^p(Q)} \leq c_2 \left(\norm{F_{\dif{v}}}_{\Leb^p(Q)} + \norm{G_{\dif{v}}}_{\Leb^q(Q)}\right) \leq c_3 \norm{F_{\dif{v}}}_{\Leb^p(Q)}.
\end{equation*}
The calculation shows, that the statement holds with $F_{v - \fint_Q v} := v - \fint_Q v$, $G_{v - \fint_Q v} := 0$.
\medskip

\stepemph{Step 3 -- Case $\norm{G_{\dif{v}}}_{\Leb^q(Q)} \geq \norm{F_{\dif{v}}}_{\Leb^p(Q)}$}: We may now assume $\norm{G_{\dif{v}}}_{\Leb^q(Q)} \geq \norm{F_{\dif{v}}}_{\Leb^p(Q)}$. We denote by $v_F$ and $v_G$ weak solutions to $\Delta v_F = \divergence (F_{\dif{v}}\mathds{1}_{Q^+})$ and $\Delta v_G = \divergence (G_{\dif{v}}\mathds{1}_{Q^+})$ in $\R^d$ as presented in Step 3 of the proof of \cite[Thm.~2.1]{CDM14}. Let $\bar{v}_F := \fint_Q v_F$ and $\bar{v}_G := \fint_Q v_G$. The Poincaré-Wirtinger inequality shows
\begin{align*}
	\norm{v_F - \bar{v}_F}_{\Leb^p(Q)} &\leq c_4 \norm{\dif{v_F}}_{\Leb^p(\R^d)} \leq c_5 \norm{F_{\dif{v}}}_{\Leb^p(Q^+)} \leq c_6 \norm{F_{\dif{v}}}_{\Leb^p(Q)}, \\
	\norm{v_G - \bar{v}_G}_{\Leb^q(Q)} &\leq c_4 \norm{\dif{v_G}}_{\Leb^q(\R^d)} \leq c_6 \norm{G_{\dif{v}}}_{\Leb^q(Q)}.
\end{align*}
The map $w := v - v_F - v_G$ is a harmonic map in $\R^d$ and smooth by Weyl's lemma. Let $\bar{w} := \fint_Q w$. By the Poincare-Wirtinger, Caccioppoli and Sobolev inequalities, we obtain
\begin{align*}
	\norm{w - \bar{w}}_{\Leb^q(Q)} &\leq c_7 \norm{\dif w}_{\Leb^q(Q)} \leq c_8 \norm{\dif w}_{\Leb^p(Q^+)} \\
	&\leq c_9 \left(\norm{F_{\dif{v}}}_{\Leb^p(Q)} + \norm{G_{\dif{v}}}_{\Leb^q(Q)}\right) \leq c_{10} \norm{G_{\dif{v}}}_{\Leb^q(Q)}.
\end{align*}
Finally, note that $\bar{w} + \bar{v}_F + \bar{v}_G = \fint_Q v$. The inequalities shown above now establish the claim with $F_{v - \fint_Q v} := v_F - \bar{v_F}$ and $G_{v - \fint_Q v} := w - \bar{w} + v_G - \bar{v_G}$.
\end{proof}

\clearpage
\printbibliography[heading=bibintoc]

@Article{NS18,
  author    = {Neukamm, Stefan AND Sch{\"a}ffner, Mathias},
  journal   = {Archive for Rational Mechanics and Analysis},
  title     = {Quantitative Homogenization in Nonlinear Elasticity for Small Loads},
  year      = {2018},
  number    = {1},
  pages     = {343--396},
  volume    = {230},
  doi       = {10.1007/s00205-018-1247-z},
  publisher = {Springer},
}

@Article{NS19,
  author    = {Neukamm, Stefan and Sch{\"a}ffner, Mathias},
  journal   = {Calc. Var. Partial Differential Equations},
  title     = {Lipschitz estimates and existence of correctors for nonlinearly elastic, periodic composites subject to small strains},
  year      = {2019},
  number    = {2},
  pages     = {46},
  volume    = {58},
  doi       = {10.1007/s00526-019-1495-2},
  publisher = {Springer},
}

@Book{MW89,
  author    = {Mawhin, Jean L. AND Willem, Michel},
  publisher = {Springer New York, NY},
  title     = {Critical Point Theory and Hamiltonian Systems},
  year      = {1989},
  address   = {New York, NY},
  isbn      = {978-1-4757-2061-7},
  series    = {Applied Mathematical Sciences},
  volume    = {74},
  doi       = {10.1007/978-1-4757-2061-7},
}

@Article{Zha97,
  author   = {Zhang, Kewei},
  journal  = {Ann. Inst. H. Poincar{\'e} Anal. Non Lin{\'e}aire},
  title    = {Quasiconvex functions, SO(n) and two elastic wells},
  year     = {1997},
  number   = {6},
  pages    = {759--785},
  volume   = {14},
  doi      = {10.1016/S0294-1449(97)80132-1},
  fjournal = {Annales de l'Institut Henri Poincar{\'e} Analyse Non Lin{\'e}aire},
}

@InBook{Rue22,
  author    = {Rüland, Angkana},
  pages     = {501--515},
  publisher = {Springer International Publishing},
  title     = {Rigidity and Flexibility in the Modelling of Shape-Memory Alloys},
  year      = {2022},
  isbn      = {9783031044960},
  booktitle = {Research in Mathematics of Materials Science},
  doi       = {10.1007/978-3-031-04496-0_21},
  issn      = {2364-5741},
}

@Book{Ste71,
  author    = {Stein, Elias M.},
  publisher = {Princeton University Press},
  title     = {Singular Integrals and Differentiability Properties of Functions},
  year      = {1971},
  isbn      = {9781400883882},
  doi       = {10.1515/9781400883882},
}

@Article{MPT19,
  author    = {Maddalena, Francesco AND Percivale, Danilo AND Tomarelli, Franco},
  journal   = {Journal of Optimization Theory and Applications},
  title     = {A new variational approach to linearization of traction problems in elasticity},
  year      = {2019},
  number    = {1},
  pages     = {383--403},
  volume    = {182},
  doi       = {10.1007/s10957-019-01533-8},
  publisher = {Springer},
}

@Article{All92,
  author    = {Allaire, Gr{\`e}goire},
  journal   = {SIAM Journal on Mathematical Analysis},
  title     = {Homogenization and Two-Scale Convergence},
  year      = {1992},
  number    = {6},
  pages     = {1482--1518},
  volume    = {23},
  doi       = {10.1137/0523084},
  publisher = {SIAM},
}

@Article{DGF75,
  author    = {De Giorgi, Ennio and Franzoni, Tullio},
  journal   = {Atti della Accademia Nazionale dei Lincei. Classe di Scienze Fisiche, Matematiche e Naturali. Rendiconti},
  title     = {Su un tipo di convergenza variazionale},
  year      = {1975},
  number    = {6},
  pages     = {842--850},
  volume    = {58},
  publisher = {Accademia Nazionale dei Lincei},
}

@Book{Ada75,
  author    = {Adams, Robert A.},
  publisher = {Academic Press},
  title     = {Sobolev Spaces},
  year      = {1975},
  address   = {New York},
  isbn      = {978-0-120-44150-1},
  series    = {Pure and Applied Mathematics},
  volume    = {65},
}

@Article{BGNPP23,
  author    = {S{\"o}ren Bartels and Max Griehl and Stefan Neukamm and David Padilla-Garza and Christian Palus},
  journal   = {Mathematical Models and Methods in Applied Sciences},
  title     = {A nonlinear bending theory for nematic {LCE} plates},
  year      = {2023},
  month     = {may},
  number    = {07},
  pages     = {1437--1516},
  volume    = {33},
  doi       = {10.1142/s0218202523500331},
  publisher = {World Scientific Pub Co Pte Ltd},
}

@Book{Dal93,
  author    = {Dal Maso, Gianni},
  publisher = {Birkhäuser Boston},
  title     = {An Introduction to $\Gamma$-Convergence},
  year      = {1993},
  edition   = {1},
  isbn      = {978-1-4612-0327-8},
  series    = {Progress in Nonlinear Differential Equations and Their Applications},
  volume    = {8},
  doi       = {10.1007/978-1-4612-0327-8},
}

@Article{MJZ18,
  author  = {Teunis {van Manen} and Shahram Janbaz and Amir A. Zadpoor},
  journal = {Materials Today},
  title   = {Programming the shape-shifting of flat soft matter},
  year    = {2018},
  issn    = {1369-7021},
  number  = {2},
  pages   = {144-163},
  volume  = {21},
  doi     = {10.1016/j.mattod.2017.08.026},
}

@Article{Jon81,
  author    = {Jones, Peter W.},
  journal   = {Acta Mathematica},
  title     = {Quasiconformal mappings and extendability of functions in sobolev spaces},
  year      = {1981},
  pages     = {71--88},
  volume    = {147},
  doi       = {10.1007/BF02392869},
  publisher = {Institut Mittag-Leffler},
}

@Article{MN11,
  author    = {M{\"u}ller, Stefan AND Neukamm, Stefan},
  journal   = {Archive for Rational Mechanics and Analysis},
  title     = {On the Commutability of Homogenization and Linearization in Finite Elasticity},
  year      = {2011},
  number    = {2},
  pages     = {465--500},
  volume    = {201},
  doi       = {10.1007/s00205-011-0438-7},
  publisher = {Springer},
}

@Article{Pom03,
  author    = {Pompe, Waldemar},
  journal   = {Commentationes Mathematicae Universitatis Carolinae},
  title     = {Korn's first inequality with variable coefficients and its generalization},
  year      = {2003},
  issn      = {1213-7243},
  number    = {1},
  pages     = {57--70},
  volume    = {44},
  publisher = {Charles University in Prague, Faculty of Mathematics and Physics},
}

@Article{Muel87,
  author    = {M{\"u}ller, Stefan},
  journal   = {Archive for Rational Mechanics and Analysis},
  title     = {Homogenization of nonconvex integral functionals and cellular elastic materials},
  year      = {1987},
  number    = {3},
  pages     = {189--212},
  volume    = {99},
  doi       = {10.1007/BF00284506},
  publisher = {Springer},
}

@Book{Ric93,
  author    = {Rickman, Seppo},
  publisher = {Springer Berlin, Heidelberg},
  title     = {Quasiregular Mappings},
  year      = {1993},
  address   = {Berlin, Heidelberg},
  isbn      = {978-3-642-78201-5},
  series    = {Ergebnisse der Mathematik und ihrer Grenzgebiete, A Series of Modern Surveys in Mathematics},
  volume    = {26},
  doi       = {10.1007/978-3-642-78201-5},
}

@InProceedings{Jam86,
  author    = {James, Richard D.},
  title     = {Stress-Free Joints and Polycrystals},
  year      = {1986},
  pages     = {381--405},
  publisher = {Springer Berlin Heidelberg},
  series    = {The Breadth and Depth of Continuum Mechanics},
  doi       = {10.1007/978-3-642-61634-1_17},
}

@Article{FJM02,
  author    = {Friesecke, Gero and James, Richard D. and M{\"u}ller, Stefan},
  journal   = {Communications on Pure and Applied Mathematics},
  title     = {A theorem on geometric rigidity and the derivation of nonlinear plate theory from three-dimensional elasticity},
  year      = {2002},
  number    = {11},
  pages     = {1461--1506},
  volume    = {55},
  doi       = {10.1002/cpa.10048},
  fjournal  = {Communications on Pure and Applied Mathematics: A Journal Issued by the Courant Institute of Mathematical Sciences},
  publisher = {Wiley Online Library},
}

@Article{Eri83,
  author    = {Ericksen, J. L.},
  journal   = {Journal of Elasticity},
  title     = {Theory of stress-free joints},
  year      = {1983},
  number    = {1},
  pages     = {3--15},
  volume    = {13},
  doi       = {10.1007/BF00041311},
  publisher = {Springer},
}

@Article{MJ22,
  author  = {Winston Mmari and Björn Johannesson},
  journal = {Applications in Engineering Science},
  title   = {A model for multiphase moisture and heat transport below and above the saturation point of deformable and swelling wood fibers-II: Hygro-mechanical response},
  year    = {2022},
  issn    = {2666-4968},
  pages   = {100118},
  volume  = {12},
  doi     = {10.1016/j.apples.2022.100118},
}

@Book{Cia88,
  author    = {Ciarlet, Philippe G.},
  publisher = {Elsevier Science Publishers B.V.},
  title     = {Mathematical Elasticity Vol.\ I},
  year      = {1988},
  series    = {Studies in Mathematics and Its Applications},
  volume    = {20},
  issn      = {0168-2024},
  subtitle  = {Three-Dimensional Elasticity},
}

@Book{War03,
  author    = {Warner, Mark AND Terentjev, Eugene Michael},
  publisher = {Oxford university press},
  title     = {Liquid crystal elastomers},
  year      = {2003},
  address   = {Oxford},
  isbn      = {978-0-198-52767-1},
  series    = {International series of monographs on physics},
  volume    = {120},
  doi       = {10.1093/oso/9780198527671.001.0001},
}

@Article{FG95,
  author    = {Fonseca, I. and Gangbo, W.},
  journal   = {SIAM Journal on Mathematical Analysis},
  title     = {Local Invertibility of Sobolev Functions},
  year      = {1995},
  issn      = {1095-7154},
  month     = mar,
  number    = {2},
  pages     = {280--304},
  volume    = {26},
  doi       = {10.1137/s0036141093257416},
  publisher = {Society for Industrial & Applied Mathematics (SIAM)},
}

@Article{KES07,
  author    = {Klein, Yael and Efrati, Efi and Sharon, Eran},
  journal   = {Science},
  title     = {Shaping of Elastic Sheets by Prescription of Non-Euclidean Metrics},
  year      = {2007},
  issn      = {1095-9203},
  month     = feb,
  number    = {5815},
  pages     = {1116--1120},
  volume    = {315},
  doi       = {10.1126/science.1135994},
  publisher = {American Association for the Advancement of Science (AAAS)},
}

@Article{BNS20,
  author    = {Bauer, Robert AND Neukamm, Stefan AND Sch{\"a}ffner, Mathias},
  journal   = {Journal of Elasticity},
  title     = {Derivation of a Homogenized Bending--Torsion Theory for Rods with Micro-Heterogeneous Prestrain},
  year      = {2020},
  issn      = {0374-3535},
  number    = {1},
  pages     = {109--145},
  volume    = {141},
  doi       = {10.1007/s10659-020-09777-6},
  publisher = {Springer},
}

@Article{Ngu89,
  author    = {Nguetseng, Gabriel},
  journal   = {SIAM Journal on Mathematical Analysis},
  title     = {A General Convergence Result for a Functional Related to the Theory of Homogenization},
  year      = {1989},
  number    = {3},
  pages     = {608--623},
  volume    = {20},
  doi       = {10.1137/0520043},
  publisher = {SIAM},
}

@Article{Mil81,
  author    = {Miller, Kenneth S.},
  journal   = {Mathematics Magazine},
  title     = {On the Inverse of the Sum of Matrices},
  year      = {1981},
  number    = {2},
  pages     = {67--72},
  volume    = {54},
  doi       = {10.1080/0025570X.1981.11976898},
  publisher = {Taylor \& Francis},
}

@Book{HK14,
  author    = {Hencl, Stanislav and Koskela, Pekka},
  publisher = {Springer Cham},
  title     = {Lectures on Mappings of Finite Distortion},
  year      = {2014},
  address   = {Cham},
  isbn      = {978-3-319-03173-6},
  series    = {Lecture Notes in Mathematics},
  doi       = {10.1007/978-3-319-03173-6},
}

@Article{MT07,
  author    = {Mielke, Alexander AND Timofte, Aida M.},
  journal   = {SIAM Journal on Mathematical Analysis},
  title     = {Two-Scale Homogenization for Evolutionary Variational Inequalities via the Energetic Formulation},
  year      = {2007},
  issn      = {1095-7154},
  month     = jan,
  number    = {2},
  pages     = {642--668},
  volume    = {39},
  doi       = {10.1137/060672790},
  publisher = {Society for Industrial & Applied Mathematics (SIAM)},
}

@Article{ADD12,
  author  = {Agostiniani, Virginia AND Dal Maso, Gianni AND DeSimone, Antonio},
  journal = {Annales de l'Institut Henri Poincaré C, Analyse non linéaire},
  title   = {Linear elasticity obtained from finite elasticity by $\Gamma$-convergence under weak coerciveness conditions},
  year    = {2012},
  issn    = {0294-1449},
  number  = {5},
  pages   = {715--735},
  volume  = {29},
  doi     = {10.1016/J.ANIHPC.2012.04.001},
}

@PhdThesis{NeuPHD,
  author  = {Neukamm, Stefan},
  school  = {Technische Universität München},
  title   = {Homogenization, linearization and dimension reduction in elasticity with variational methods},
  year    = {2010},
  address = {München},
  type    = {Dissertation},
}

@Article{CDG02,
  author  = {Cioranescu, Doina AND Damlamian, Alain AND Griso, Georges},
  journal = {Comptes Rendus Mathematique},
  title   = {Periodic unfolding and homogenization},
  year    = {2002},
  issn    = {1631-073X},
  number  = {1},
  pages   = {99--104},
  volume  = {335},
  doi     = {10.1016/S1631-073X(02)02429-9},
}

@Article{BNPS22,
  author    = {Klaus Böhnlein and Stefan Neukamm and David Padilla-Garza and Oliver Sander},
  journal   = {Journal of Nonlinear Science},
  title     = {A Homogenized Bending Theory for Prestrained Plates},
  year      = {2022},
  number    = {1},
  volume    = {33},
  doi       = {10.1007/s00332-022-09869-8},
  publisher = {Springer Science and Business Media {LLC}},
}

@InCollection{Bra06,
  author    = {Braides, Andrea},
  publisher = {North-Holland},
  title     = {Chapter 2 A handbook of $\Gamma$-convergence},
  year      = {2006},
  pages     = {101--213},
  series    = {Handbook of Differential Equations: Stationary Partial Differential Equations},
  volume    = {3},
  doi       = {10.1016/S1874-5733(06)80006-9},
  issn      = {1874-5733},
}

@Article{GN11,
  author    = {Gloria, Antoine and Neukamm, Stefan},
  journal   = {Annales de l’Institut Henri Poincaré C, Analyse non linéaire},
  title     = {Commutability of homogenization and linearization at identity in finite elasticity and applications},
  year      = {2011},
  issn      = {1873-1430},
  month     = dec,
  number    = {6},
  pages     = {941--964},
  volume    = {28},
  doi       = {10.1016/j.anihpc.2011.07.002},
  publisher = {European Mathematical Society - EMS - Publishing House GmbH},
}

@Book{EG15,
  author    = {Evans, Lawrence C. AND Gariepy, Ronald F.},
  publisher = {CRC Press},
  title     = {Measure Theory and Fine Properties of Functions},
  year      = {2015},
  address   = {Boca Raton, Fla.},
  edition   = {Revised Ed.},
  isbn      = {978-1-482-24238-6},
  series    = {Textbooks in Mathematics},
  doi       = {10.1201/b18333},
}

@Article{Bal81,
  author    = {Ball, J. M.},
  journal   = {Proceedings of the Royal Society of Edinburgh: Section A Mathematics},
  title     = {Global invertibility of Sobolev functions and the interpenetration of matter},
  year      = {1981},
  number    = {3-4},
  pages     = {315--328},
  volume    = {88},
  doi       = {10.1017/S030821050002014X},
  publisher = {Royal Society of Edinburgh Scotland Foundation},
}

@Article{DNP02,
  author  = {Dal Maso, Gianni AND Negri, Matteo AND Percivale, Danilo},
  journal = {Set-Valued Analysis},
  title   = {Linearized Elasticity as $\Gamma$-Limit of Finite Elasticity},
  year    = {2002},
  pages   = {165--183},
  volume  = {10},
  doi     = {10.1023/A:1016577431636},
}

@Article{Vis07,
  author    = {Visintin, Augusto},
  journal   = {Calculus of Variations and Partial Differential Equations},
  title     = {Two-scale convergence of some integral functionals},
  year      = {2007},
  number    = {2},
  pages     = {239--265},
  volume    = {29},
  doi       = {10.1007/s00526-006-0068-3},
  publisher = {Springer},
}

@Article{Vis06,
  author    = {Visintin, Augusto},
  journal   = {ESAIM: COCV},
  title     = {Towards a two-scale calculus},
  year      = {2006},
  number    = {3},
  pages     = {371--397},
  volume    = {12},
  doi       = {10.1051/cocv:2006012},
  fjournal  = {ESAIM: Control, Optimisation and Calculus of Variations},
  publisher = {EDP Sciences},
}

@Book{KR19,
  author    = {Kru\v{z}\'{i}k, Martin and Roub\'{i}\v{c}ek, Tom\'{a}\v{s}},
  publisher = {Springer Cham},
  title     = {Mathematical Methods in Continuum Mechanics of Solids},
  year      = {2019},
  address   = {Cham},
  edition   = {1},
  isbn      = {978-3-030-02065-1},
  series    = {Interaction of Mechanics and Mathematics},
  doi       = {10.1007/978-3-030-02065-1},
}

@Article{DM04,
  author  = {Dur{\'a}n, Ricardo G. AND Muschietti, Maria Amelia},
  journal = {Electronic Journal of Differential Equations (EJDE)},
  title   = {The Korn inequality for Jones domains},
  year    = {2004},
  number  = {127},
  pages   = {1--10},
  volume  = {2004},
}

@Article{Lic19,
  author    = {Licht, Martin W.},
  journal   = {Mathematics of Computation},
  title     = {Smoothed projections over weakly Lipschitz domains},
  year      = {2019},
  number    = {315},
  pages     = {179--210},
  volume    = {88},
  doi       = {10.1090/mcom/3329},
  publisher = {American Mathematical Society (AMS)},
}

@Article{DT68,
  author    = {Davidge, R. W. and Tappin, G.},
  journal   = {Journal of Materials Science},
  title     = {Internal strain energy and the strength of brittle materials},
  year      = {1968},
  issn      = {1573-4803},
  month     = may,
  number    = {3},
  pages     = {297--301},
  volume    = {3},
  doi       = {10.1007/bf00741965},
  publisher = {Springer Science and Business Media LLC},
}

@Article{Mar78,
  author    = {Marcellini, Paolo},
  journal   = {Annali di Matematica Pura ed Applicata},
  title     = {Periodic solutions and homogenization of non linear variational problems},
  year      = {1978},
  number    = {1},
  pages     = {139--152},
  volume    = {117},
  doi       = {10.1007/BF02417888},
  publisher = {Springer},
}

@Article{PWHPG94,
  author    = {Post, D. and Wood, J. D. and Han, B. and Parks, V. J. and Gerstle, F. P.},
  journal   = {Journal of Applied Mechanics},
  title     = {Thermal Stresses in a Bimaterial Joint: An Experimental Analysis},
  year      = {1994},
  issn      = {1528-9036},
  month     = mar,
  number    = {1},
  pages     = {192--198},
  volume    = {61},
  doi       = {10.1115/1.2901397},
  publisher = {ASME International},
}

@Article{CK88,
  author    = {Chipot, Michel and Kinderlehrer, David},
  journal   = {Archive for Rational Mechanics and Analysis},
  title     = {Equilibrium configurations of crystals},
  year      = {1988},
  issn      = {1432-0673},
  month     = sep,
  number    = {3},
  pages     = {237--277},
  volume    = {103},
  doi       = {10.1007/bf00251759},
  publisher = {Springer Science and Business Media LLC},
}

@Book{Att84,
  author    = {Attouch, Hedy},
  publisher = {Pitman Advanced Publishing Program},
  title     = {Variational convergence for functions and operators},
  year      = {1984},
  volume    = {1},
}

@Article{Bra85,
  author  = {Braides, Andrea},
  journal = {Rend. Accad. Naz. Sci. XL},
  title   = {Homogenization of some almost periodic coercive functional},
  year    = {1985},
  pages   = {313--322},
  volume  = {103},
}

@Article{Schm07,
  author   = {Schmidt, Bernd},
  journal  = {Calc. Var. Partial Differential Equations},
  title    = {Minimal energy configurations of strained multi-layers},
  year     = {2007},
  issn     = {0944-2669},
  number   = {4},
  pages    = {477--497},
  volume   = {30},
  doi      = {10.1007/s00526-007-0099-4},
  fjournal = {Calculus of Variations and Partial Differential Equations},
}

@Article{Zan90,
  author    = {Zanzotto, Giovanni},
  journal   = {Journal of Elasticity},
  title     = {Thermoelastic stability of multiple growth twins in quartz and general geobarothermometric implications},
  year      = {1990},
  issn      = {1573-2681},
  month     = may,
  number    = {2–3},
  pages     = {253--287},
  volume    = {23},
  doi       = {10.1007/bf00054806},
  publisher = {Springer Science and Business Media LLC},
}

@Article{CDM14,
  author    = {Conti, Sergio AND Dolzmann, Georg AND M{\"u}ller, Stefan},
  journal   = {Calculus of Variations and Partial Differential Equations},
  title     = {Korn's second inequality and geometric rigidity with mixed growth conditions},
  year      = {2014},
  number    = {1},
  pages     = {437--454},
  volume    = {50},
  doi       = {10.1007/s00526-013-0641-5},
  publisher = {Springer},
}

@article{hassani2015rheological,
	title={Rheological model for wood},
	author={Hassani, Mohammad Masoud and Wittel, Falk K and Hering, Stefan and Herrmann, Hans J},
	journal={Computer Methods in Applied Mechanics and Engineering},
	volume={283},
	pages={1032--1060},
	year={2015},
	publisher={Elsevier}
}

@article{zhang2017characterization,
	title={Characterization of residual stress and deformation in additively manufactured ABS polymer and composite specimens},
	author={Zhang, Wei and Wu, Amanda S and Sun, Jessica and Quan, Zhenzhen and Gu, Bohong and Sun, Baozhong and Cotton, Chase and Heider, Dirk and Chou, Tsu-Wei},
	journal={Composites Science and Technology},
	volume={150},
	pages={102--110},
	year={2017},
	publisher={Elsevier}
}

\end{document}